\documentclass[ejs]{imsart}

\RequirePackage{amsthm,amsmath,amsfonts}
\RequirePackage[numbers]{natbib}
\RequirePackage[colorlinks,citecolor=blue,urlcolor=blue]{hyperref}
\RequirePackage{graphicx,paralist,microtype,dsfont,longtable,booktabs,subcaption,xcolor,ragged2e}

\graphicspath{{plots/}}

\pubyear{2024}
\volume{0}
\issue{0}
\firstpage{1}
\lastpage{46}
\arxiv{2212.04071}

\startlocaldefs
\numberwithin{equation}{section}
\theoremstyle{plain}

\endlocaldefs

\theoremstyle{plain}

\newtheorem{lemma}{Lemma}[section]
\newtheorem{corollary}{Corollary}[section]
\newtheorem{proposition}{Proposition}[section]

\theoremstyle{definition}
\newtheorem{assumption}{Assumption}
\newtheorem{definition}{Definition}
\newtheorem{remark}{Remark}[section]

\theoremstyle{definition}
\newtheorem{assumpA}{Assumption}

\newtheorem{assumpLW}{Assumption}

\newtheorem*{assumption*}{Assumption}
\newtheorem*{proof*}{Proof of}

\renewcommand\tilde{\widetilde}
\renewcommand\hat{\widehat}
\renewcommand\bar{\overline}
\DeclareMathOperator{\ran}{range}

\DeclareMathOperator{\spn}{span}

\DeclareMathOperator{\rank}{rank}

\DeclareMathOperator{\sgn}{sgn}

\newcommand{\dn}{d_{\hspace{-0.025em}\scalebox{0.65}{\text{$N$}}}}
\newcommand{\ds}{d_{\hspace{-0.025em}\scalebox{0.65}{\text{$S$}}}}
\newcommand{\dnn}{\scalebox{0.65}{\text{$N$}}}
\newcommand{\dss}{\scalebox{0.65}{\text{$S$}}}
\makeatletter
\DeclareRobustCommand\citepos
{\begingroup\def\NAT@nmfmt##1{{\NAT@up##1's}}%
	\NAT@swafalse\let\NAT@ctype\z@\NAT@partrue
	\@ifstar{\NAT@fulltrue\NAT@citetp}{\NAT@fullfalse\NAT@citetp}}
\makeatother

\newcommand{\pto}{\stackrel{\to}{p}} \newcommand{\dto}{\stackrel{\to}{d}} 
\newcommand{\truc}[1]{\lfloor #1 \rfloor}
\def\dto{\underset{d}\rightarrow}
\def\pto{\underset{p}\rightarrow}

\newcommand{\argmax}{\operatornamewithlimits{argmax}}

\newcommand{\commWS}[1]{{\leavevmode\color{black}#1}}
\newcommand{\commHS}[1]{{\leavevmode\color{black}#1}}
\newcommand{\commRV}[1]{{\leavevmode\color{black}#1}}
\newcommand{\commRVV}[1]{{\leavevmode\color{black}#1}}
\newcommand{\Rlogo}{\protect\includegraphics[height=1.8ex,keepaspectratio]{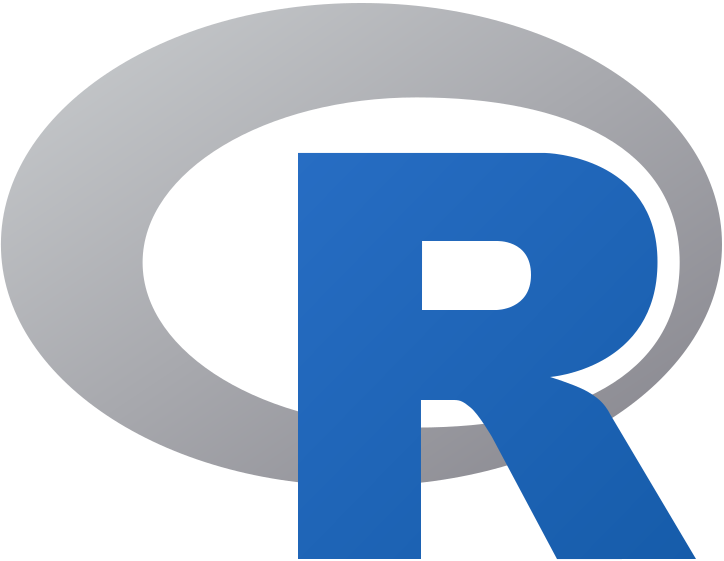}}

\begin{document}
	
	\begin{frontmatter}
		\title{Fractionally integrated curve time series with cointegration}
		\runtitle{Fractionally integrated curve time series}
		
		\begin{aug}
			\author{\fnms{Won-Ki} \snm{Seo}\thanksref{t1}\ead[label=e1]{won-ki.seo@sydney.edu.au}}
			\and
			\author{\fnms{Han Lin} \snm{Shang}\thanksref{t2}\ead[label=e2]{hanlin.shang@mq.edu.au}}
			
			\address{University of Sydney\\
				Macquarie University\\
				\printead{e1,e2}}
			
			
			
			\thankstext{t1}{Won-Ki Seo is the corresponding author. Data and \Rlogo\ code used in this paper are available at \url{https://github.com/hanshang/FCIFTS}. This paper is accepted at the \it{Electronic Journal of Statistics}. \normalfont The authors express deep appreciation to the editor and anonymous referees for their invaluable and insightful suggestion.}
			\thankstext{t2}{Shang's research is supported in part by Australian Research Council Discovery Project DP230102250}
			\runauthor{Seo and Shang}
			
		\end{aug}
		
		\begin{abstract}
			We introduce methods and theory for fractionally cointegrated curve time series. We develop a variance-ratio test to determine the dimensions associated with the nonstationary and stationary subspaces. For each subspace, we apply a local Whittle estimator to estimate the long-memory parameter and establish its consistency. A Monte Carlo study of finite-sample performance is included, along with two empirical applications.
		\end{abstract}
		
		\begin{keyword}[class=MSC]
			\kwd[Primary ]{37M10}
			\kwd[; secondary ]{62F03}
		\end{keyword}
		
		\begin{keyword}
			\kwd{fractional cointegration}
			\kwd{long memory time series}
			\kwd{functional data} 
			\kwd{functional principal component analysis}
			\kwd{limit theorems}
		\end{keyword}
		\tableofcontents
	\end{frontmatter}
	
	\section{Introduction}\label{sintro} 

	Functional time series analysis balances functional data and time series analyses. Similar to univariate and multivariate time series, a temporal dependence structure exists in functional observations. For example, intraday volatility functions are serially dependent and often exhibit long-memory features \citep{Lobato97}. Time series of airway pressure are used to monitor patients undergoing mechanical ventilation, with series exhibiting periodically strong dependence \citep{BNF+23}. 
	
In functional time series analysis, most studies assume stationarity over the short-range temporal dependence (see, e.g., \cite{BYZ10, Bosq2000, FV06,  Hoermann2010, Horvath2012, Horvath2014, Klepsch2017a, Laurini2014}). Only in recent years has there been some development on long-memory functional time series models (see, e.g., \cite{Li2020a, Li2021, Li2020, RP20}). The long-memory functional time series describes processes with greater persistence than short-range dependent ones, such that, in the stationary case, autocovariances decay very slowly, and the spectral density is unbounded, especially at frequency zero. While \cite{Li2020a, Li2021} consider inference and estimation of a long-memory parameter in stationary curve time series, \cite{Li2020} studies inferential results for nonstationary curve time series. Based on the mean squared error, \cite{Shang20, Shang22} evaluate and compare various long-memory parameter estimators for stationary and nonstationary curve time series, respectively. In these comparisons, the local Whittle estimator of \cite{robinson1995gaussian} is recommended. While \cite{Li2021} presents the asymptotic properties of the local Whittle estimator, \cite{Shang22b} applies a sieve bootstrap method from \cite{Paparoditis18} to nonparametrically construct the confidence intervals of the memory parameter.
	
	Given the recent surge of interest in functional time series analysis, cointegration methods have been extended to a functional time series setting by \cite{Chang2016}. They define cointegration for curve time series and develop statistical methods based on functional principal component analysis. \commRVV{\cite{Beare2017b} and \cite{Seo2019} provide suitable notions of cointegration in a Hilbert space and a Bayes Hilbert space, respectively. \cite{Beare2017b} and \cite{Beare2020} study a general solution to a functional autoregressive law of motion allowing for nonstationarity and cointegration and present extensions of the Granger-Johansen representation theorem; see also \cite{Franchi2020}, \cite{Seo2022}, and \cite{Seo2023} for similar theoretical results in a more general setting.}
	
	We study a fractionally cointegrated curve time series by developing inferential and estimation methods for such a time series. The curve time series consists of nonstationary and stationary components. For each component, we estimate the long-memory parameter via the local Whittle estimator. Through a variance-ratio test, we determine the subspaces spanned by nonstationary and stationary components. We compare our method with the existing one based on the eigenvalue-ratio estimator by \cite{Li2020a}. In addition, this paper develops statistical methods for the case when the stationary component can further be decomposed into the long-memory and short-memory components. 
	
	In Section~\ref{sec:2}, we present our notations and preliminaries. In Section~\ref{sec_frac}, we introduce the fractionally cointegrated functional time series. The estimation procedure is given in Section~\ref{sec_stat}. Illustrated by a series of simulation studies in Section~\ref{sec_simulation}, we evaluate the estimation accuracy of the proposed method and compare the result with \cite{Li2020}. The empirical performance of our proposed method is also validated through application to a Swedish human mortality data set in Section~\ref{sec:mortality} and a Canadian yield curve data set in Section~\ref{sec_yield}. In Section~\ref{sec:conclude}, we conclude and present some ideas on how the methodology can be further extended.

\section{Preliminaries}\label{sec:2}
	
In the subsequent discussion, we assume that the curve-valued time series $Z_t$ of interest takes values in the Hilbert space $\mathcal H$ of square-integrable functions defined on $[0,1]$. We let $\langle h_1,h_2 \rangle$ denote the inner product of $h_1,h_2 \in \mathcal H$, and then let $\|h\|$ denote the norm of $h \in \mathcal H$, which is defined by $\langle h,h \rangle^{1/2}$.   Given a set $G \subseteq \mathcal H$, we let $G^\perp$ denote the orthogonal complement to $G$.	We denote by $\mathcal L_{\mathcal H}$ the space of bounded linear operators acting on $\mathcal H$ equipped with the uniform norm $\|A\|_{\mathcal L_{\mathcal H}}=\sup_{\|h\|\leq1}\|Ah\|$. The adjoint $A^\ast$ of a linear operator $A\in\mathcal L_{\mathcal H}$ is the unique linear operator satisfying \(\langle Ah_1,h_2 \rangle=\langle h_1,A^\ast h_2 \rangle\) for all $h_1,h_2\in \mathcal H$. We will say that an operator $A\in\mathcal L_{\mathcal H}$ is nonnegative (resp.\ positive) definite if \(\langle Ah,h\rangle \geq 0\) (resp.\ \(\langle Ah,h\rangle > 0\)) for all nonzero \(h\in \mathcal H\). We let $\otimes$ denote the tensor product of elements in $\mathcal H$, that is, $h_1 \otimes h_2(\cdot)$ denotes the linear map given by $\langle h_1, \cdot \rangle h_2$ for any $h_1,h_2\in \mathcal H$. We let the range of $A\in \mathcal L_{\mathcal H}$ be denoted by \(\ran A\). The dimension of $\ran A$ is called the rank of $A$, denoted by $\rank A$. We will consider convergence of a sequence of random bounded linear operators as the sample size $T$ tends to infinity. \phantomsection\label{rvlabel5}\commRV{For such a sequence of operators $A_j$, we write $A_j \pto A$ to denote the convergence in probability of $A_j$ to $A$ with respect to the uniform operator norm, that is, 
$$A_j \pto A \quad \Leftrightarrow  \quad \|A_j-A\|_{\mathcal L_{\mathcal H}} \pto 0.$$}
	
We define the I$(d)$ property of a time series, taking values in $\mathcal H$. As a crucial building block, we first introduce the I$(0)$ property adopted from \cite{Beare2017b}.  
\begin{definition}[I$(0)$-ness] \label{defi0}
The time series $X_t$ taking values in $\mathcal H$ is said to be I$(0)$, and denoted $X_t \in$ I$(0)$, if 
\begin{inparaenum}
\item[(i)] it allows the representation
\begin{equation*}
X_t = \sum_{j=0}^\infty \psi_j \varepsilon_{t-j}, 
\end{equation*}
where $\psi_j \in \mathcal L_{\mathcal H}$ for all $j$ and $\varepsilon_t$ is an i.i.d.\ sequence with \commWS{$\mathbb{E}\varepsilon_t = 0$} and positive definite covariance \commWS{$C_{\varepsilon}$, defined as $\mathbb{E}[\varepsilon_t\otimes \varepsilon_t]$}, and 
\item[(ii)] $\sum_{j=0}^{\infty} \psi_j$ is a nonzero element in $\mathcal L_{\mathcal H}$.
\end{inparaenum} 
\end{definition}
The I$(0)$ time series is necessarily stationary, and its long-run covariance $\sum_{s=-\infty}^\infty  \mathbb{E}[X_{t-s}\otimes X_t]$ is a well-defined bounded linear operator (see, e.g., \cite{Beare2017b}). \phantomsection\label{rv203}In this paper, an I(0) time series is also referred to as a short-range dependent (SRD) process, as in the literature (see, e.g., \cite{Li2020a}). Based on the I$(0)$ property, we define $\mathcal H$-valued I$(d)$ processes, which will subsequently be considered as follows: 
	\begin{definition}[I$(d)$-ness] \label{def1}
		\commRV{For $d \geq 0$, the time series $Y_t$ is said to be I$(d)$ (or equivalently, fractionally integrated of order $d$), and denoted $Y_t \in$ I$(d)$,} \commRVV{if  $\Delta^d Y_t$ is I$(0)$}, where  $\Delta^{d}$ is a power series of the lag operator defined by 
		\begin{equation*}
			\Delta^{d} = \sum_{j=0}^{\infty} \frac{\Gamma(j-d)}{\Gamma(-d)\Gamma(j+1)} L^j.
		\end{equation*}
	\end{definition}
	\commRVV{Cointegration is a property of a multivariate nonstationary dynamic system, indicating the existence of a stationary linear combination of the variables in the system. Cointegration (or the so-called cointegrating relationship) between the variables is often interpreted as their stable long-run dynamic relationship, and various empirical examples have been proposed, especially in studies of macroeconomic dynamics. The notion of cointegration has been recently extended to and studied in a Hilbert space setting by \cite{Beare2017b}, \cite{Li2020}, \cite{Nielsen2019}, \cite{Nielsen2023},  \cite{Seo2022} and \cite{Seo2020}.}  Extending these former notions of cointegration, we may define fractional cointegration in $\mathcal H$ as follows: 
	\begin{definition}[Fractional cointegration] \label{deffc}
	Suppose that  $Y_t\in$ I$(\dn)$ and there exists a projection $\mathcal P$ such that $\mathcal PY_t \in$ I$(\ds)$ for some $\dn\geq 1/2$,  $\ds\in [0,1/2)$. We then say that $Y_t$ is (fractionally) cointegrated and call $v \in \ran \mathcal P$ a (fractional) cointegrating vector.
	\end{definition}
	If $\dn \geq 1/2$, an I$(\dn)$ time series $Y_t$ taking values in $\mathcal H$ is nonstationary. \commRV{However, given that we can find a subspace $\ran \mathcal P \subset \mathcal H$ such that $\mathcal PY_t$ is stationary process, Definition~\ref{deffc} suitably extends the conventional notion of cointegration in a Euclidean space.}
	\section{Fractionally cointegrated functional time series} \label{sec_frac}
	
	We consider modeling nonstationary dependent curve-valued observations but exhibiting stable long-run linear relationships as fractionally cointegrated time series. A potential example of such a time series may be yield curves over time; it turns out that this time series tends to evolve like a nonstationary process \commRVV{(see the working paper version of \citep{Li2020} and \citep{Nielsen2023})}, but due to \commWS{the expectations hypothesis}, many linear functionals of such time series are expected not to exhibit nonstationarity; \phantomsection\label{rv201} see e.g. 
\cite{GT18}, \cite{HAG92} and \cite[Section 6]{Nielsen2010}. 
	
	Even if Definition~\ref{deffc} gives us a proper notion of $\mathcal H$-valued time series allowing fractional cointegration, the definition itself is, of course, not sufficient for the inferential methods to be developed. For statistical analysis, we employ the following assumptions for the observed time series $Z_t$:
	\begin{assumption}\label{assum1} 
		The observed time series $Z_t$, taking values in $\mathcal H$, satisfies the following:
		\begin{enumerate}[(a)] 
			\item \label{assum1a} $Z_t = \mu + Y_t$ for some $\mu \in \mathcal H$.
			\item \label{assum1b}  For some $\dn \in (1/2,3/2)$, $\ds \in [0,1/2)$, there exists an orthogonal projection $P$ and an I(0) sequence $X_t$ given by $X_t= \sum_{j=0}^\infty \psi_j \varepsilon_{t-j}$ satisfying   
\commRV{		
	\begin{align}
				\Delta^{\dn} P(Y_t-Y_0) &= P X_t  \mathds{1}\{t\geq 1\}, \label{eqstart1add} \\
				\Delta^{\ds} (I-P)Y_t & =  (I-P)X_t, \label{eqstart2add} 
			\end{align}}where $\sum_{j=0}^\infty j\|\psi_j\|^2 < \infty$, and 	$\varepsilon_t$ satisfies that  $\mathbb{E}\|\varepsilon_t\|^{\tau} < \infty$ for some $\tau> \max\{4,2/(2\dn-1)\}$, and $\mathds{1}\{\cdot\}$ denotes binary indicator function.
			\item\label{assum1c} $\rank P\sum_{j=0}^\infty \psi_j = q_{\dnn} < \infty$.		 
		\end{enumerate} 
	\end{assumption}
	
	Some comments on Assumption~\ref{assum1} are in order. First, in most empirical applications, a functional time series tends to have a nonzero intercept. Thus, in~(\ref{assum1a}), we assume that the observed time series is given by the sum of an I$(\dn)$ process $Y_t$ and an unobserved intercept $\mu \in \mathcal H$. Moreover, of course, it might sometimes be of interest to consider a linear time trend component; even if we do not explicitly deal with this case,  most of the results to be subsequently given may be extended to accommodate this possibility with moderate modifications. We describe the cointegrating properties of the stochastic part of the observed time series $Z_t$ in~(\ref{assum1b}) with some other necessary conditions for our mathematical development. Here we restrict our interest to the case with $\dn \in (1/2,3/2)$, which seems relevant in most empirical applications involving nonstationary \commWS{fractionally} integrated time series. \commRV{Note that the process given in \eqref{eqstart1add} is nonstationary since $\dn>1/2$ and thus the trucation operator $\mathds{1}\{t\geq 1\}$ needs to be used to have a well-defined process (see, e.g., Section 4 of \cite{Li2020})}. \phantomsection{\label{rvlabel0}	\commRV{If we let $\Delta^{d}_+$ denote the truncated fractional difference operator defined by $\Delta^{d}_+ =	\sum_{j=0}^{t-1} \frac{\Gamma(j-d)}{\Gamma(-d)\Gamma(j+1)} L^j$, then \eqref{eqstart1add} and \eqref{eqstart2add} can be conveniently rewritten as	\begin{align}
				P(Y_t-Y_0) &= \Delta^{-\dn}_+ P X_t, \label{eqstart1} \\
				(I-P)Y_t & = \Delta^{-\ds} (I-P)X_t = \sum_{j=0}^\infty b_j  \varepsilon_{t-j}, \label{eqstart2}
			\end{align}
where $b_j= \sum_{k=0}^\infty \frac{\Gamma(j-k+\ds)}{\Gamma(\ds)\Gamma(j-k+1)}(I-P)\psi_j$.} 	Note that we require $(I-P)Y_t$ to be \commRVV{a stationary long-range dependent (LRD) process with $\ds \in (0,1/2)$, or short-range dependent (SRD) process with $\ds =0$}. Given that any orthogonal projection may be understood as a bipartite decomposition of a Hilbert space, what~\eqref{eqstart1} and~\eqref{eqstart2} imply is that our observed time series may be understood as the sum of two heterogeneous components: the nonstationary component $P(Y_t-Y_0)$ and the stationary component $(I-P)Y_t$. Associated with this condition,~(\ref{assum1c}) identifies the collection of the cointegrating vectors as $\ran (I-P)$; under this condition, $\langle Y_t,v \rangle$ is stationary if and only if $v \in \ran (I-P)$ (see Proposition 3.1 of \cite{Beare2017b}). 	However, given that
\begin{inparaenum}
\item[(i)] \phantomsection\label{rv202} our I$(\ds)$ property does not exclude the possibility of an additional memory deduction (on $\ran (I-P)$) and
\item[(ii)] $\ran (I-P)\sum_{j=0}^\infty \psi_j$ may not be equal to $\ran (I-P)$, there may exist another orthogonal projection $Q$ such that $\ran Q \subset \ran (I-P)$ and $QY_t \in$ I($d_{\ell}$) for $d_{\ell} \in [0,\ds]$. That is, the time series $(I-P)Y_t$ is a quite general stationary process. Given this time series, we are interested in identifying the nonstationary and stationary components from the observed time series, which will be discussed in the next section.
	\end{inparaenum}		

\begin{remark} \label{remempirical} 
	We assume that the functional time series of interest contains a finite dimensional nonstationary component. This assumption appears reasonable in many empirical applications, as evidenced by \cite{Chang2016}, \cite{Nielsen2019}, and \cite{Seo2020}. 
These studies provide strong statistical evidence supporting the existence of finite dimensional I(1) components in various time series, including cross-sectional densities of individual earnings, return densities of a financial asset, age-specific employment rates, and intraday electricity consumption. Moreover, for certain economic functional time series, the assumption $q_{\dnn} < \infty$ is natural. As a representative example, let $Y_t$ be a sequence of yield curves over time, which is a well-known example of nonstationary functional time series.  According to the expectations hypothesis of the term structure of interest rates (see, e.g., Section 6 of \cite{Nielsen2010}), we can reasonably assume that $q_{\dnn}< \infty$. Specifically, bonds with different maturities will be nearly perfect substitutes for each other and, thus, their nonstationarity can reasonably be attributed to a small number of nonstationary components. This argument can be extended to other examples where $Y_t(u)$ and $Y_{t}(u')$ are substitutes for each other for $u,u' \in [0,1]$.   
\end{remark}

\begin{remark} \label{remfarima}
The nonstationary component $P(Y_t-Y_0)$ may be understood as a natural extension of the FARIMA($p,d,q$) processes considered in \cite{Li2020a}. Specifically, consider the following FARMA process: 
\begin{align}
\Phi(L)V_t = \Theta (L) \varepsilon_t,
\end{align}
where $\Phi(L) = I-\Phi_1L - \cdots - \Phi_p L^p$,  $\Theta(L) = I-\Theta_1L - \cdots - \Theta_q L^q$, and $\Phi_j$ and $\Theta_j$ are all bounded linear operators. We let \commRVV{$W_t$} be the process defined by 
\begin{align}
\Delta^{d}\commRVV{W_t} = V_t.
\end{align}
If $d \in (-1/2,1/2)$ and $V_t$ is stationary I(0), then \commRVV{$W_t$} becomes a FARIMA($p,d,q$) process considered in \cite{Li2020a}. On the other hand, if $d=0$ (and thus $\Delta^d=I$) but $\Phi(L)$ allows a unit root instead, then \commRVV{$W_t$} becomes a cointegrated time series considered in \cite{Beare2020}.  Now suppose that  $d \in (-1/2,1/2)$ and $\Phi(L)$ allows a unit root, and the AR operators $\Phi_j$ are compact (note that this compactness assumption is quite common in the literature). We then know from the Granger-Johansen representation theorem (see e.g., Theorem 3.1 of \cite{Beare2020}) that $\Delta V_t$ can be written as $\Delta V_t=\Upsilon \varepsilon_t + V^S_t - V^S_{t-1}$ for $t\geq 1$,  some finite rank operator $\Upsilon$, and an I(0) stationary sequence $V^S_t$. Thus \commRVV{$W_t$} defined by $\Delta^{\dn} \commRVV{W_t} = \Delta V_t \mathds{1}\{t\geq 1\}$ for $\dn:=d+1 \in (1/2,3/2)$ only contains a finite dimensional I($\dn$) component; that is,  $q_{\dnn} < \infty$ as in our Assumption~\ref{assum1}. In fact, this result holds under even weaker conditions, requiring neither compactness of $\Phi_j$ nor the Hilbert space structure of $\mathcal H$; see, for example, \cite{Seo2022} and \cite{Seo2023}.     
\end{remark}

Sometimes, practitioners may also be interested in the case where the stationary part of $Y_t$ can be decomposed into the SRD and LRD components. We will also  consider this case by imposing  the following additional conditions on the stationary component:
	\begin{assumpA}\label{assum2add} 
		The observed time series $Z_t$ satisfies Assumption~\ref{assum1}, and there exists an orthogonal projection $Q$ such that 
\commRV{		
\begin{align*} 
    QP &= 0\\ 
		\Delta^{\ds}	QY_t &=  Q X_t, \\
			(I-Q)(I-P)Y_t &= (I-Q)(I-P) X_t, 	 
		\end{align*}}
		and
		\begin{equation*}
			\rank Q\sum_{j=0}^\infty \psi_j= q_{\dss} < \infty.
		\end{equation*}
	\end{assumpA}
		Under Assumption~\ref{assum2add}, the time series $\{Z_t\}_{t\geq 1}$ satisfying Assumption~\ref{assum1} can be decomposed into three different components: $q_{\dnn}-$dimensional nonstationary, $q_{\dss}-$dimensional LRD and infinite-dimensional SRD components. \phantomsection\label{rvlabel7}\commRV{In empirical applications involving nonstationary functional time series, even after extracting the nonstationary component, it is reasonable to expect that the resulting residual time series may still contain an LRD component. Assumption~\ref{assum2add} is to accommodate this possibility to our model.} In this case, practitioners may be interested in decomposing the nonstationary component from the stationary component (given by the sum of the LRD and SRD components) and in decomposing the LRD component from the SRD component. Moreover, the memory parameters $\dn$ and $\ds$ may also be of interest in practice. We will discuss these issues.
	
	It is useful to introduce some additional notation and terminology. Under Assumption~\ref{assum1}, we have the bipartite decomposition $\mathcal H = \ran P \oplus \ran (I-P)$. As clarified above, the collection of cointegrating vectors is given by $\ran (I-P)$, which is called the cointegrating space and denoted by $\mathcal H_S$. The orthogonal complement to $\mathcal H_S$ is called the dominant subspace (as in \cite{Li2020}) and denoted by $\mathcal H_N$. If Assumption~\ref{assum2add} is satisfied, then $\mathcal H_S$ can also be decomposed into $Q\mathcal H_S$ and $(I-Q)\mathcal H_S$, which are called the LRD and SRD subspaces and denoted by $\mathcal H_{LRD}$ and $\mathcal H_{SRD}$. To sum up, we have
	\begin{align*}
		\mathcal H &= \mathcal H_N \oplus \mathcal H_S \quad \hspace{.8in} \text{under Assumption~\ref{assum1}}, \\
		\mathcal H &= \mathcal H_N \oplus \mathcal H_{LRD} \oplus \mathcal H_{SRD} \quad \text{under Assumption~\ref{assum2add}},
	\end{align*}
\phantomsection\label{rvlabel4}\commRV{where, obviously,}
	\begin{align*}
\mathcal H_N &= \ran P, \,\,\, \mathcal H_S = \ran (I-P), \\ 
	\mathcal H_{LRD} &= \ran Q, \,\,\, \mathcal H_{SRD} = \ran (I-Q)(I-P).
	\end{align*}
\begin{remark}\label{remaddRV2} 
The finite dimensionality of the dominant component is a prevalent assumption in the literature concerning stationary or nonstationary integrated time series of possibly fractional order; see \cite{Chang2016}, \cite{Li2020a}, \cite{Nielsen2019}, \cite{Nielsen2023} and \cite{Seo2020}. This assumption is incorporated into our context (as reflected in Assumptions~\ref{assum1} and~\ref{assum2add}), positioning our paper between existing works and presenting it as an extension, at least to some extent. 
\end{remark}

\begin{remark}\label{remaddRV2add} 
\commRVV{As noted by an anonymous referee, to whom we are indebted, it may be of interest to explore the LRD component $QY_t$ with additional potential memory reduction on $\mathcal{H}_{LRD}$, unlike in Assumption \ref{assum2add}. While this could be a more realistic assumption in practice, developing statistical methods for such a general LRD component is beyond the scope of the present paper. This topic is left for future study, pending further theoretical advancements in the literature on LRD functional time series. }  
\end{remark}


\section{Statistical methods}\label{sec_stat} 
	
	To make our statistical inference invariant with respect to a (possibly) nonzero intercept $\mu$, we will consider $Z_t^0$ or $\bar{Z}_t$ depending on the context, which is defined as follows: for $t\geq 1$,  
	\begin{equation}\label{eqzbar} 
		Z_t^0 = Z_t - Z_0 \quad \text{and} \quad	\bar{Z}_t = Z_t - \frac{1}{T}\sum_{t=1}^T Z_t, 
	\end{equation} 
	where we assume that $Z_0$ is observed. Of course, in practice, $Z_t^0$ will be replaced by $Z_t-Z_1$ by putting the first observation aside to initialize the time series. Thus, no essential restriction is placed by using $Z_t^0$ in analysis. 
	
\subsection{Decomposition of $\mathcal H_N$ and $\mathcal H_S$}\label{sec_fracdecom} 
	
In this section, we consider the decomposition of the nonstationary and stationary components, which essentially boils down to identifying the cointegrating space $\mathcal H_S$ or the dominant subspace $\mathcal H_N$. As may be deduced from the existing literature (see, e.g., \cite{Chang2016,Li2020,Nielsen2019}), the dominant subspace $\mathcal H_N$ can be estimated by the span of the eigenvectors corresponding to the first $q_{\dnn}$ largest eigenvalues of a certain sample operator. For this reason, the estimation of $\mathcal H_S$ reduces to the estimation of $q_{\dnn}$, which will be subsequently discussed. The quantity $q_{\dnn}$ itself may be of interest to practitioners since it represents the number of linearly independent fractional unit root processes embedded into the time series; in the literature considering $n$-dimensional vector-valued fractionally integrated time series, the quantity $n-q_{\dnn}$ is commonly called the (fractional) cointegrating rank and has attracted significant attention. We will develop statistical inference on $q_{\dnn}$ in this section and obtain the desired decomposition.  
	
	\subsubsection{Eigenvalue-ratio-based estimation of $q_{\dnn}$} \label{sec_eresti}
	\commHS{We will first consider an eigenvalue-ratio-based estimator, similar to the estimator of the dimension of the dominant subspace proposed by \citep{Li2020}}. In our Monte Carlo study, this estimator performs worse in finite samples compared to our second estimator, which is obtained using the proposed sequential testing procedure. Nevertheless, the subsequent discussion becomes a crucial input to the aforementioned testing procedure and helps form a deeper understanding of fractionally cointegrated time series. 
	
	Under Assumption~\ref{assum1}, an element $v$ included in the dominant subspace $\mathcal H_N$ is differentiated with any other element $\tilde{v} \in \mathcal H_S$ in the sense that the sample variance of $\langle Y_t,v\rangle$ tends to be higher than that of $\langle Y_t,\tilde{v}\rangle$; more specifically, we have
	\commRV{\begin{equation*}
		\frac{T^{-1}\sum_{t=1}^T \langle Y_t,v\rangle^2}{T^{-1}\sum_{t=1}^T \langle Y_t,\tilde{v}\rangle^2}  \pto  \infty.
	\end{equation*}}Based on the above intuition combined with Lemma~\ref{lemwkc} and the asymptotic properties of the covariance operator of nonstationary fractionally integrated functional time series, we may establish the following result: 
	\begin{proposition}\label{prop3} 
		Suppose that Assumption~\ref{assum1} holds, $K$ is a finite integer satisfying $K>q_{\dnn}$ and the $(K-q_{\dnn})$-th largest eigenvalue of $\mathbb{E}[(I-P)Y_t\otimes (I-P)Y_t]$ is nonzero and distinct from the next one. Let $(\hat{\mu}_1,\ldots,\hat{\mu}_K)$ be the ordered (from the largest)  eigenvalues  of the sample covariance operator of $Z_t$ given by  
		\begin{equation*}
			\hat{C}_{\bar{Z}}=	\frac{1}{T} \sum_{t=1}^T \bar{Z_t} \otimes \bar{Z}_t. 
		\end{equation*}
		Then the following holds:
		\begin{enumerate}[(i)] 
			\item\label{eqde1} ${\hat{\mu}_j}/{\hat{\mu}_{j+1}}  \pto \infty$ if $j=q_{\dnn}$ while ${\hat{\mu}_j}/{\hat{\mu}_{j+1}}  = O_p(1)$ if $j\neq q_{\dnn}$.
			\item The corresponding eigenvectors $(\hat{v}_1,\ldots,\hat{v}_{q_{\dnn}})$ satisfy that 
			\begin{align}
				\sum_{j=1}^{q_{\dnn}} \hat{v}_j \otimes \hat{v}_j   \pto  P. \label{eqde3}
			\end{align}
		\end{enumerate}
	\end{proposition}
	Some direct consequences of the results given in Proposition~\ref{prop3} are given as follows:
	\begin{corollary} \label{cor1}
		Let everything be as in Proposition~\ref{prop3}. Then the following hold.
		\begin{enumerate}[(i)]
			\item\label{cor1a} $\hat{q}_{\dnn}:= \argmax_{1\leq j\leq K} \left(\frac{\hat{\mu}_j}{\hat{\mu}_{j+1}}\right)  \pto q_{\dnn}$.
			\item\label{cor1b}	$\hat{P}:=	\sum_{j=1}^{\hat{q}_{\dnn}} \hat{v}_j \otimes \hat{v}_j   \pto  P$.
		\end{enumerate}
	\end{corollary}
	Note that Proposition~\ref{prop3} requires a careful choice of $K$ satisfying some mathematical conditions, which is crucial to achieve the consistency results in Corollary~\ref{cor1} (see Remark~\ref{remev1}). However, such a choice can be obtained without difficulty (see Remark~\ref{remchoose}). \commWS{It is also worth emphasizing that our results given in Proposition~\ref{prop3} and Corollary~\ref{cor1} require that only the $(K-q_{\dnn})$-th  eigenvalue of $\mathbb{E}[(I-P)Y_t\otimes (I-P)Y_t]$ is different from the next one; that is, the results allow the case where some of the first $(K-q_{\dnn})$ eigenvalues are tied. Of course, this requirement does not seem to be restrictive in practice.} Given that any closed subspace of $\mathcal H$ may be identified as the unique orthogonal projection onto the space,~(\ref{cor1b}) in Corollary~\ref{cor1} may be understood as the convergence of $\ran \hat{P}$ (resp.\ $\ran (I-\hat{P})$) to $\mathcal H_N$ (resp.\ $\mathcal H_S$), and we thus may write
	\begin{equation*}
		\ran \hat{P} \pto \mathcal H_N \quad \text{and} \quad  \ran (I-\hat{P}) \pto \mathcal H_S.
	\end{equation*}
	
	\begin{remark} \label{remev1}
		In our proof of Proposition~\ref{prop3}, it is shown that, for some $\Omega$ which is symmetric and positive definite on $\mathcal H_N$, 
		\begin{align}
			T^{1-2\dn} \hat{\mu}_j  &\dto    \text{$j$-th largest eigenvalue of } \int_{0}^1 \Omega \bar{W}_{\dn}(s) \otimes \Omega \bar{W}_{\dn}(s)ds,   \label{eqconsis1}  
		\end{align}
	\text{jointly for $j\leq q_{\dnn}$}, and 
	\begin{align}
			\hat{\mu}_j  &\pto   \text{$j$-th  largest eigenvalue of }  \mathbb{E}[(I-P){Y}_t\otimes (I-P){Y}_t],    \label{eqconsis2}
		\end{align}
	\text{jointly for $j>q_{\dnn}$}. In \eqref{eqconsis1}, $\bar{W}_{\dn}$ is a demeaned Type II  fractional Brownian motion of order $\dn$ defined on $\mathcal H_N$, which will be introduced in detail in Section~\ref{sec_appen}. 
		The results given in Proposition~\ref{prop3}(\ref{eqde1})  are consequences of the above convergence results. Moreover, this shows why we require the $(K-q_{\dnn})$-th eigenvalue of $\mathbb{E}[(I-P)Y_t\otimes (I-P)Y_t]$ to be nonzero in Proposition~\ref{prop3}; if the  $(K-q_{\dnn})$-th eigenvalue is zero, then $\hat{\mu}_{K-1}/\hat{\mu}_{K} \pto \infty$, which is not desirable for consistency of $\hat{q}_{\dnn}$. 
	\end{remark}
	\begin{remark}\label{remprop3}
		\commHS{The estimator given in Proposition~\ref{prop3} may be understood as a tailored version of the eigenvalue-ratio estimator of the dimension of the dominant subspace proposed by \cite{Li2020}.} It is worth noting two important differences in theoretical and practical aspects. First, due to cointegration, we can explain more about the role of $K$ (an upper bound of $q_{\dnn}$, which needs to be chosen by a researcher), while its role is not sufficiently discovered in the setting of \cite{Li2020}. Due to this, the estimator of \cite{Li2020} requires an additional and arbitrary penalty parameter to suppress the possibility that two small eigenvalues result in a large ratio and hence may give a misleading estimate. On the other hand, we can provide a feasible and less arbitrary way to choose $K$ (see Remark~\ref{remchoose} below).
	\end{remark}
	
	
	\begin{remark}\label{remchoose}
		In Proposition~\ref{prop3}, $K$ needs to be greater than $q_{\dnn}$. We know from Remark~\ref{remev1} that the first $q_{\dnn}$ largest eigenvalues of $\hat{C}_{\bar{Z}}$ have bigger stochastic orders than the remaining eigenvalues. It thus may not be difficult to conjecture a slightly bigger integer than $q_{\dnn}$ from the estimated eigenvalues, and $K$ can be set to such an integer.  Note that we also require the $(K-q_{\dnn})$-th eigenvalue of $\mathbb{E}[(I-P)Y_t\otimes (I-P)Y_t]$ to be nonzero and distinct from the next one.  Given that we consider a functional time series, $\mathbb{E}[(I-P)Y_t\otimes (I-P)Y_t]$ tends to allow many nonzero eigenvalues in most empirical applications. Moreover, violation of this condition may be avoided by checking if $\hat{\mu}_{K}-\hat{\mu}_{K+1}$ is bounded away from zero; see~\eqref{eqconsis2}. \commRV{As an extreme case, if the eigenvalues of $\mathbb{E}[(I-P)Y_t\otimes (I-P)Y_t]$ are all nonzero and distinct, we do not require any condition on $K$  in our asymptotic analysis; that is, $K$ can be any arbitrary finite integer greater than $q_{\dnn}$.} 
		\end{remark}

Even if we can consistently estimate $q_{\dnn}$ (and thus $P$) based on Proposition~\ref{prop3}, practitioners may be more interested in a statistical test for $q_{\dnn}$, which demonstrates how strongly the data support a certain hypothesis about $q_{\dnn}$. In the next section, we provide a variance-ratio-type test for $q_{\dnn}$ that can be applied to our functional time series setting and propose an alternative estimator $\bar{q}_{\dnn}$ obtained by sequential application of the test. Our simulation results show that this new estimator tends to outperform $\hat{q}_{\dnn}$.

\subsubsection{Variance-ratio test on $q_{\dnn}$} \label{sec_vrtest}

The limiting behavior of the sample covariance operator $\hat{C}_{\bar{Z}} = T^{-1}\sum_{t=1}^T \bar{Z}_t \otimes \bar{Z}_t$ under the existence of cointegration enables us to implement a statistical test about $q_{\dnn}$, which will be discussed in this section. 

As the first step to developing our test, we consider a fractionally integrated variable $\widetilde{Z}_t$ as follows:
\begin{equation}\label{eqztilde}
	\widetilde{Z}_t = \Delta_+^{-\alpha} \bar{Z}_t.
\end{equation}
\commRV{The constant $\alpha > 0$ is specific to the user. As will be discussed later, the selection of $\alpha$ affects the limiting distribution and appears to be important for the finite sample properties of our test to be detailed. However, from a theoretical perspective, $\alpha$ is permitted to take any positive real value (see Remark~\ref{remdalpha} for a more detailed discussion).}   
For any positive integer $K$, let $\hat{P}_K$ denote the orthogonal projection given by
\begin{align} \label{defpk}
	\hat{P}_K = \sum_{j=1}^{K} \hat{v}_j \otimes \hat{v}_j,
\end{align}
where $(\hat{v}_1,\ldots,\hat{v}_K)$ is the eigenvectors corresponding to the first $K$ largest eigenvalues of $\hat{C}_{\bar{Z}}$. Let $A_T$ and $B_T$ be defined by 
\begin{equation*}
	A_T = \sum_{t=1}^T 	\hat{P}_K\bar{Z}_t \otimes	\hat{P}_K \bar{Z}_t, \quad B_T = \sum_{t=1}^T 	\hat{P}_K\widetilde{Z}_t \otimes	\hat{P}_K \widetilde{Z}_t.
\end{equation*}
We then define the following generalized eigenvalue problem:
\begin{equation}\label{eqv}
	\hat{\nu}_j  B_T \hat{w}_j=  A_T \hat{w}_j, \quad \hat{\nu}_1\leq \hat{\nu}_2\leq\ldots \leq \hat{\nu}_K, \quad   \hat{w}_j \in \ran \hat{P}_K.
\end{equation}
Since the domain and the codomain of each of $A_T$ and $B_T$ are restricted to the span of the first $K$ eigenvectors of $\hat{C}_{\bar{Z}}$, we may compute $K$ (almost surely) positive eigenvalues from~\eqref{eqv}. Our main result in this section is given as follows: in the proposition below, $B_{\delta}(s)$ denotes a $q_{\dnn}$-dimensional type II fractional standard Brownian motion defined by $B_{\delta}(0) = 0$ almost surely and $B_{\delta}(s)=\frac{1}{\Gamma(\delta)}\int_{0}^s (s-r)^{\delta-1}dW_0(r)$ for  $s>0$ and the standard Brownian motion $W_0$.
\begin{proposition}\label{prop4} 
	Suppose that Assumption~\ref{assum1} holds,  $K$ is a finite integer satisfying $K>q_{\dnn}$ and the $(K-q_{\dnn})$-th largest eigenvalue of $\mathbb{E}[(I-P)Y_t\otimes (I-P)Y_t]$ is nonzero and distinct from the next one, and $\alpha>0$. \commRV{Let $(\nu_{1,\alpha},\ldots,\nu_{q_{\dnn},\alpha})$} be the ordered eigenvalues (from the smallest) of 
	\begin{align}\label{eqprop4}
		\left(\int_{0}^{1}  \tilde{B}_{\dn+\alpha}(s)  \tilde{B}_{\dn+\alpha}(s)'ds \right)^{-1} \int_{0}^{1}  \bar{B}_{\dn}(s)  \bar{B}_{\dn}(s)'ds, 
	\end{align}
	where\commRV{
	\begin{align*}
		\bar{B}_{\dn}(s) &= 	 B_{\dn}(s) - \int_{0}^1 B_{\dn}(r)dr,\\
		\widetilde{B}_{\dn+\alpha}(s) &=  B_{\dn+\alpha}(s) - \left(\int_{0}^1 B_{\dn+\alpha}(r)dr\right)\left(\int_{0}^s \frac{(s-r)^{\alpha-1}}{\Gamma(\alpha)}dr\right).
	\end{align*}}
	Then \commRV{
	\begin{equation*} 
		T^{2\alpha} (\hat{\nu}_1,\ldots,\hat{\nu}_{q_{\dnn}})   \dto  (\nu_{1,\alpha},\ldots,\nu_{q_{\dnn},\alpha})
	\end{equation*}}
	and 
	\begin{equation*}
		T^{2\alpha}\hat{\nu}_{q_{\dnn}+j}   \pto  \infty, \quad j =1,\ldots,K-q_{\dnn}.
	\end{equation*}
\end{proposition}
	The asymptotic results given in Proposition~\ref{prop4} enable us to implement a more detailed statistical inference on $q_{\dnn}$ beyond consistent estimation of it. Specifically, let us consider the following null and alternative hypotheses:
	\begin{equation}\label{eqhypo2}
		H_0: q_{\dnn} = q \quad \text{v.s.} \quad 	H_1: q_{\dnn} < q.
	\end{equation}
	Based on the asymptotic results given in Proposition~\ref{prop4}, we know that, for example,  
	\begin{align}
		\Lambda_{q,\alpha}^{0} = T^{2\alpha} \max_{1\leq j\leq q} \hat{\nu}_j  \quad \text{and} \quad	\Lambda_{q,\alpha}^{1} = T^{2\alpha}\sum_{j=1}^{q} \hat{\nu}_j \label{stats}
	\end{align}
	have well-defined limiting distributions under $H_0$ while they diverge to infinity under $H_1$. Using these statistics, we may easily evaluate the plausibility of the null hypothesis. Moreover, as an alternative way to estimate $q_{\dnn}$, we may sequentially examine~\eqref{eqhypo2} for $q = q_{\max},q_{\max}-1,\ldots,1$, where $q_{\max}$ is a reasonable upper bound. In practice, $q_{\max}$ may be chosen based on the estimated eigenvalues of $\hat{C}_{\bar{Z}}$ (see  Remark~\ref{remchoose}) or can be set to $\hat{q}_{\dnn}+\epsilon$ using the modified eigenvalue ratio estimator of \cite{Li2020} and small finite integer $\epsilon$. This sequential procedure is consistent in the following sense:

	\begin{corollary} \label{cor1prop4} 
		Suppose that the assumptions in Proposition~\ref{prop4} hold, and let $\bar{q}_{\dnn}$ be the estimator obtained from this sequential procedure with fixed significance level $\eta>0$. Then, 
		\begin{equation*}
			\mathbb{P}(\bar{q}_{\dnn}=q_{\dnn}) \to 1- \eta \cdot \mathds{1}\{q_{\dnn}\geq 1\}. 
		\end{equation*}
		By letting $\eta \to 0$ as $T \to \infty$, we have $\mathbb{P}(\bar{q}_{\dnn}=q_{\dnn}) \to 1$ for all possible values of $q_{\dnn}$. 
	\end{corollary} 
\commRV{
Even though the estimator of $q_{\dnn}$ in Corollary~\ref{cor1prop4} is constructed using sequential tests, it is important to note that the testing procedure guarantees correct size asymptotically. This property has been well-documented in previous works, such as \cite{Johansen1995}, \cite{Nielsen2010}, and \cite{Nielsen2019}, which studied similar procedures.} 
	Our proof of Proposition~\ref{prop4} also shows that the first $q_{\dnn}$ eigenvectors computed from~\eqref{eqv} converge to a random orthonormal basis of $\mathcal H_N$. Therefore, we can also obtain the following: 
	\begin{corollary}\label{cor2prop4}
		Suppose that the assumptions in Proposition~\ref{prop4} hold and \commRV{let $\tilde{q}_{\dnn}$ be any consistent estimator of $q_{\dnn}$}. Then,
		\begin{equation*}
			\widetilde{P} = \sum_{j=1}^{\tilde{q}_{\dnn}} \hat{w}_j \otimes \hat{w}_j  \pto  P.
		\end{equation*}
	\end{corollary}
	Suppose that practitioners are only interested in consistent estimation of $\mathcal H_N$ or $\mathcal H_S$; they then might prefer to use the eigenvalue-ratio estimator $\hat{q}_{\dnn}$ developed in Section~\ref{sec_eresti} since it is much easier to implement. However, our simulation study shows that $\bar{q}_{\dnn}$ substantially outperforms $\hat{q}_{\dnn}$. Since the estimation of $\mathcal H_N$ (or $\mathcal H_S$) can be affected by inaccuracy in the estimator of its dimension, this result \commWS{supports} the use of the testing procedure given in Corollary~\ref{cor1prop4} in practice.   

\begin{remark} \label{remdalphaadded}
In Corollary~\ref{cor1prop4}, $q_{\max}$ is required to be a reasonable upper bound for $q_{\dnn}$. As suggested by \citet[Remark 6]{Nielsen2019} for their testing procedure to determine the number of I(1) stochastic trends, if the testing procedure yields $\bar{q}_{\dnn}=q_{\max}$, it is advisable to restart with a higher value of $q_{\max}$ to mitigate the risk of selecting a value that is too small.
\end{remark}

\begin{remark} \label{remdalpha}
	The limiting distribution given in~Proposition~\ref{prop4} depends on $\alpha$ and $\dn$.  Note that $\dn$ is an unknown parameter of interest. Therefore, in implementing the proposed test in practice, $\dn$ needs to be replaced by a consistent estimator of $\dn$, such as the local Whittle estimator that we will consider later in Section~\ref{sec_est_d}. On the other hand, $\alpha$ is known and needs to be chosen by the researcher. Even if our asymptotic theory allows for any $\alpha>0$, it would be prudent to use $\alpha$ which is not very close to $0$ given the nature of the test. In our simulation study, we consider an ad-hoc choice $\alpha=0.5$ and compute the test statistics using $T\hat{\nu}_j$ for $j=1,\ldots,q_{\dnn}$. We find that the test with this choice of $\alpha$ performs reasonably and present a sensitivity analysis in Tables~\ref{tab1addrv1},~\ref{tabsizepowerrv1} and~\ref{tabsizepowerrv2}.   
\end{remark}

\begin{remark} \label{remadd}
	If we consider a finite-dimensional Euclidean space setting, our test based on $\Lambda_{q,\alpha}^1$ reduces to the test of \cite{Nielsen2010} developed for fractionally cointegrated time series. Even if there are some moderate differences in the cointegrating properties assumed in the present paper and that of \cite{Nielsen2010} (e.g., in that paper, the considered time series is written as the sum of the nonstationary and asymptotically stationary components), our tests developed in this section may be viewed as generalizations of \citepos{Nielsen2010}'s test to some degree.  
\end{remark}


\subsection{Decomposition of $\mathcal H_{LRD}$ and $\mathcal H_{SRD}$} \label{sec_further}

We, in this section, consider the estimation of $\mathcal H_{LRD}$ and $\mathcal H_{SRD}$ in the case where $\mathcal H_{S}$ can be further decomposed as in Assumption~\ref{assum2add}; of course, this requires a consistent estimator of $\mathcal H_{S}$ in advance. The variance-ratio test developed in Section~\ref{sec_vrtest} cannot be directly used for this problem since it requires nonstationarity of the underlying time series. As an alternative method, we provide a consistent estimator of $q_{\dss}$, similar to the eigenvalue-ratio estimator considered in Section~\ref{sec_eresti}. 

Suppose that $P$ is known. We then know that the long-run variance of $\langle Y_t,v\rangle$ for $v \in \mathcal H_{LRD} = \ran Q (=\ran Q(I-P))$ is unbounded while that of $\langle Y_t, \tilde{v}\rangle$ for $\tilde{v} \in  \mathcal H_{SRD} = \ran (I-Q) (I-P)$ is bounded \commRV{(see e.g., Section 2.1 of \cite{Li2020a})}. Using this property, we may distinguish $v \in \mathcal H_{LRD}$ from any element in $\mathcal H_{SRD}$. Our proposed estimator of $q_{\dss}$ is obtained by extending this idea, and then $\mathcal H_{LRD}$ can also be estimated by the span of certain $q_{\dss}$ eigenvectors as in Section~\ref{sec_eresti}. \commWS{Of course, in practical applications, $P$ is unknown. This issue can be addressed by replacing $P$ with its consistent estimator, as indicated by the subsequent asymptotic results.}

Let $\Lambda_0$ denote the operator defined by 
\begin{equation*}
	\Lambda_0 =  \sum_{s=-\infty}^\infty \mathbb{E}[(I-Q) (I-P) Y_{t-s}\otimes (I-Q) (I-P) Y_t],
\end{equation*}
which is the population long-run covariance of the SRD component of $Y_t$ and a well-defined bounded linear operator. We also let $\hat{\Lambda}$ be the sample operator defined by
\begin{equation} \label{eqslrv}
	\hat{\Lambda} =  \sum_{s=-T+1} ^{T-1}  \left (1-\frac{|s|}{h}\right ) \hat{C}_s,
\end{equation}  
where 	
\begin{align*}
	\hat{C}_s = 
	\begin{cases}
		T^{-1}\sum_{t=s+1} ^T   \bar{Z}_{t-s} \otimes  \bar{Z}_t, \quad &\text{ if } s \geq 0,\\
		T^{-1}\sum_{t=-s+1} ^T  \bar{Z}_t\otimes   \bar{Z}_{t+s}, \quad &\text{ if } s < 0.
	\end{cases} 
\end{align*} 
We then establish the following result:
\begin{proposition} \label{prop6} 
	Suppose that Assumption~\ref{assum2add} holds, $h=o(T^{1/2})$, $K$ is a finite integer satisfying $K>q_{\dss}$ and the $(K-q_{\dss})$-th largest eigenvalue of $\Lambda_0$ is nonzero and distinct from the next one. Let $(\hat{\mu}_1,\ldots,\hat{\mu}_K)$ be the ordered (from the largest) eigenvalues of $(I-\overline{P})\hat{\Lambda}(I-\overline{P}^\ast)$ for any  $\overline{P} \pto P$ as $T \to\infty$. Then the following hold: 
	\begin{enumerate}[(i)] 
		\item\label{eqdel1} ${\hat{\mu}_j}/{\hat{\mu}_{j+1}}  \pto \infty$ if $j=q_{\dss}$ while ${\hat{\mu}_j}/{\hat{\mu}_{j+1}}  = O_p(1)$ if $j\neq q_{\dss}$.
		\item The corresponding eigenvectors $(\hat{v}_1,\ldots,\hat{v}_{q_{\dss}})$ of $(I-\overline{P})\hat{\Lambda}(I-\overline{P}^\ast)$ satisfy that 
		\begin{equation*}
			\sum_{j=1}^{{q}_{\dss}} \hat{v}_j \otimes \hat{v}_j   \pto  Q. 
		\end{equation*}
	\end{enumerate}
\end{proposition}
Some direct consequences of Proposition~\ref{prop6} are given as follows:
\begin{corollary} \label{cor3}
	Let everything be as in Proposition~\ref{prop6}. Then the following hold.
	\begin{enumerate}[(i)]
		\item\label{cor3a} $	\hat{q}_{\dss}= \argmax_{1\leq j\leq K} \left(\frac{\hat{\mu}_j}{\hat{\mu}_{j+1}}\right)\pto  q_{\dss}$.
		\item\label{cor3b}	$\sum_{j=1}^{\hat{q}_{\dss}} \hat{v}_j \otimes \hat{v}_j  \pto  Q$ and $I - \overline{P} -	\sum_{j=1}^{\hat{q}_{\dss}} \hat{v}_j \otimes \hat{v}_j    \pto  (I-Q) (I-P)$.
	\end{enumerate}
\end{corollary}
In Proposition~\ref{prop6} and Corollary~\ref{cor3},  $\overline{P}$ may be replaced by $\hat{P}$ or $\widetilde{P}$ obtained earlier. The role of $K$ in the estimation of $q_{\dss}$ and $Q$ is somewhat similar to that as described in Remark~\ref{remev1}, which will be detailed in Remark~\ref{remev2} below.
\begin{remark} \label{remev2}
	In our proof of Proposition~\ref{prop6}, we show that  $ h^{-2\ds} (\hat{\mu}_1,\ldots,\hat{\mu}_{q_{\dss}})$ converge in probability to the eigenvalues of a well-defined linear operator while  $(\hat{\mu}_{q_{\dss}+1},\ldots,\hat{\mu}_{K})$ converge to some eigenvalues of $\Lambda_0$.  This shows why we require the $(K-q_{\dss})$-th largest eigenvalue of $\Lambda_0$ to be distinct from the next one; if there is no such a distinction, $\hat{\mu}_{K-1}/\hat{\mu}_{K} \pto \infty$, which is not desirable to establish consistency of $\hat{q}_{\dss}$. 
\end{remark}

\subsection{Estimation of the memory parameters}\label{sec_est_d}

As shown in Section~\ref{sec_vrtest}, a consistent estimator of the memory $\dn$ is necessary to implement our variance-ratio test in practice \phantomsection\label{rvlabel6}\commRV{(see Remark~\ref{remdalpha})}. Moreover, practitioners may be interested in $\dn$ and $\ds$ \commWS{for their own sake}. In this section, we briefly discuss estimation results for these memory parameters via the local Whittle method. A more detailed discussion of our estimation results can be found in Appendix~\ref{sec_appen_d}. 

For convenience, we let $\hat{d}_{LW}(z_t)$ denote the local Whittle estimator computed from a time series $z_t$ with tuning parameter $m$ (depending on the sample size $T$) and a proper range of admissible values (this range depends on $z_t$ and will be detailed in Appendix~\ref{sec_appen_d}). \commRV{We postpone  the detailed discussion on the local Whittle estimation of the memory parameter to Section~\ref{sec_lw}}. 

\subsubsection{Estimation of $\dn$}\label{sec_est_d1}
\commRV{With a simplifying assumption that  $\psi_j = \phi_j A$ for some $\phi_j \in \mathbb{R}$ and $A \in \mathcal L_{\mathcal H}$ \phantomsection\label{rvlabel8}(\commRV{this assumption might look restrictive, but it is still more general than the assumption, requiring scalar coefficients, in the recent article by \cite{Li2021} that develops the local Whittle estimator for  LRD functional time series)} and some standard regularity conditions imposed on the time series $\langle	X_t ,v \rangle$  for $v$ satisfying $\mathbb{P}(v \notin  \mathcal H_{S})=1$, we note that for any $v \in \mathcal H$, $\langle Z^0_t,  v\rangle$ is I$(\dn)$ as long as $v \notin \mathcal H_{S} =\ran (I-P)$. Given that $\langle Z^0_t,v \rangle $ is a univariate I$(\dn)$ process, our goal reduces to the estimation of the memory parameter of $\langle Z^0_t,v \rangle$. 

We establish that}
\begin{equation}\label{eqdhat1}
	\hat{d}_{LW}(\langle	{Z}^0_t ,v \rangle) \pto  \dn  
\end{equation}
for $\dn \in (1/2,1]$ if $m$ grows with an appropriate rate  (see Proposition~\ref{prop1}(\ref{prop1a})).  Unfortunately, $\hat{d}_{LW}(\langle	{Z}^0_t ,v \rangle)$ is not consistent if  $\dn>1$  (see Proposition~\ref{prop1}(\ref{prop1b})), but in this case, we may use the following result for consistent estimation:
\begin{equation}\label{eqdhat1a}
	1+ \hat{d}_{LW}(\langle	\Delta {Z}_t ,v \rangle) \pto  \dn, 
\end{equation}
where $\dn$ can be all possible values in $(1/2, 3/2)$ (see Proposition~\ref{prop1}(\ref{prop1c})).

For the consistency results given in~\eqref{eqdhat1} and~\eqref{eqdhat1a},  $v$ is required not to be included in $ \mathcal H_{S}$ with probability one. Choosing such $v$ may not be difficult in practice since the probability that any $v$, randomly picked from $\mathcal H$, is exactly orthogonal to $ \mathcal H_{N}$ is zero. In practical implementation, we may conveniently set $v = \sum_{j=1}^J a_j v_j$ for some orthonormal set $\{v_j\}_{j=1}^J$ (e.g.,  the first $J$ elements of the Legendre polynomial basis of $\mathcal H$) and nonzero coefficients $\{a_j\}_{j=1}^J$. This choice is valid as long as at least one $v_j$ is not exactly orthogonal to $ \mathcal H_{N}$, and thus will be valid even with a moderate integer $J$. {If we choose $v$ in this way and thus $v \notin \mathcal H^N$, the proposed estimator in \eqref{eqdhat1} or \eqref{eqdhat1a} might underestimate $\dn$ in a finite sample since $\dn$ is the maximal memory of the time series $\langle Y_t,v \rangle$, which is as expected from the discussion in \cite{Li2020} on their infeasible local Whittle estimator requiring a choice of $v \in \mathcal H_N$. 
To avoid this possibility in practice, it seems prudent to consider a few different choices of $v$, say $v^{(\ell)}$ for $\ell=1,\ldots, L$, and then disregard choices that give us relatively smaller estimates. In our numerical studies in Section~\ref{sec_simulation}, we set $v=v^{(\ell)}$, which maximizes $\hat{d}_{LW} (\langle {Z}^0_t,v^{(\ell)} \rangle)$ from a finite set of $v^{(\ell)}$ constructed from the first few Legendre polynomial basis functions, and our simulation results show that this approach works well even though the functional time series is generated by a different basis system; see Section~\ref{sec_simulation}.}  

Asymptotic inference on $\dn$ can also be implemented; in particular, under some additional assumptions to be detailed in Appendix~\ref{sec_appen_d}, we may use the following result:
\begin{equation} \label{eqhat02}
	m^{1/2}(1+ \hat{d}_{LW}(\langle	\Delta {Z}_t ,v \rangle)  - \dn) \dto N(0,1/4);
\end{equation}
see Proposition~\ref{prop2add}. The asymptotic distribution of $\hat{d}_{LW}(\langle	{Z}^0_t ,v \rangle)$ can also be obtained, but it turns out to be dependent on the true value of $\dn$, which is not desirable in the practical use of the asymptotic result. Therefore,~\eqref{eqhat02} may be more convenient for practitioners. 

The simplifying assumption imposed on $\psi_j$ to obtain the results given in this section seems to be strong and significantly restricts the data generating process, but we conjecture that this assumption is not necessarily required; our simulation results show that our estimator performs well even if the assumption is not satisfied. The assumption is imposed only to ensure that $\langle X_t ,v \rangle$ is a stationary linear process. In more general cases where $\langle X_t ,v \rangle$ is allowed to be a stationary nonlinear process, we may conjecture from the results given by \cite{shao2007local} that the local Whittle estimator will be consistent if some additional assumptions are satisfied.  

\begin{remark}
	\cite{Li2020} provided a procedure to consistently estimate $\dn$. Let $\hat{v}_1$ be the eigenvector corresponding to the largest eigenvalue of $T^{-1} \sum_{t=1}^T Z^0_t \otimes Z^0_t$. Then Theorem 4.2 of \cite{Li2020} implies that the local Whittle estimator computed from $\langle {Z}_t^0,\hat{v}_1 \rangle$ with a proper range of admissible values converges to $\dn$ if $\dn \in (1/2,1]$; in the case where $\dn>1$, they proposed an integer-order differencing algorithm to estimate $\dn$.  Even if this estimator can be used in our model, our simulation results show that our estimator performs better than theirs.
\end{remark}

\subsubsection{Estimation of $\ds$}\label{sec_est_d2}

Estimation of the memory parameter of the LRD component, $\ds$, requires prior knowledge of $q_{\dnn}$ or its consistent estimator. However, as shown in the previous sections, we may construct a consistent estimator of $q_{\dnn}$, so it is assumed to be known in this section for simplicity. 

Let $\{\hat{v}_j\}_{j=q_{\dnn}+1}^\infty$ be the eigenvectors of $\hat{C}_{\bar{Z}}=T^{-1} \sum_{t=1}^T \bar{Z_t} \otimes \bar{Z}_t$ corresponding to the eigenvalues except for the first  $q_{\dnn}$ largest ones, and let $\{{v}_j\}_{j=q_{\dnn}+1}^\infty$ be the eigenvectors of $\mathbb{E}[(I-P){Y}_t\otimes (I-P){Y}_t]$. 
Then we may establish the following result under a similar set of assumptions employed for estimation of $\dn$: if the largest eigenvalue of $\mathbb{E}[(I-P){Y}_t\otimes (I-P){Y}_t]$ is distinct from the second one and $v_{q_{\dnn}+1}$ satisfies certain regularity conditions (to be detailed in Appendix~\ref{sec_appen_d2}),
\begin{equation} \label{eq001aa}
	\hat{d}_{LW}(\langle \bar{Z}_t,\hat{v}_{q_{\dnn}+1} \rangle)  \pto  \ds. 
\end{equation}
Under some additional conditions stated in Assumption~2$^\ast$ of \cite{Li2021}, we may also deduce the following  from Theorem 1 of \cite{Li2021}:
\begin{equation} \label{eq001aaa}
	m^{1/2}[\hat{d}_{LW}(\langle \bar{Z}_t,\hat{v}_{q_{\dnn}+1} \rangle) - \ds] \dto N(0,1/4).
\end{equation}
If the first $J$ largest eigenvalues of $\mathbb{E}[(I-P){Y}_t\otimes (I-P){Y}_t]$ are distinct, it can be shown that $\hat{v}_{q_{\dnn}+1}$ can be replaced by a linear combination of $\hat{v}_{q_{\dnn}+1}$ and the next $(J-1)$ eigenvectors, such as $\tilde{v} = \sum_{j=1}^{J} a_j \hat{v}_{q_{\dnn}+j}$ (see Appendix~\ref{sec_appen_d2}). In this case, we may also implement estimation of $\ds$ by considering a few different choices of $\tilde{v}$ as in Section~\ref{sec_est_d}. \commWS{As will be further  discussed in Section~\ref{sec_simulation} with a specific example, employing such a linear combination instead of relying solely on $\hat{v}_{q_{\dnn}+1}$ may be beneficial for (i) achieving more accurate estimation of $\ds$ and (ii) enhancing the accuracy of the estimators of the eigen-elements (see Section~\ref{sec_sim_result}).} 

\commRV{
\begin{remark} \label{remadd2}
For consistent estimation of $d_{\dss}$, we are required to have a consistent estimator of $q_{\dnn}$. More specifically, we need some vector $\hat{v}_j$ that converges to a vector of $\mathcal H_{LRD}$  (up to sign changes), and such a $\hat{v}_j$ can be determined from a consistent estimate $\bar{q}_{\dnn}$ obtained from our testing procedure. In finite samples, if  $\bar{q}_{\dnn} < q_{\dnn}$, then we expect that $\hat{d}_{\dss}$, obtained by replacing $q_{\dnn}$ with $\bar{q}_{\dnn}$, is close to 0.5, which may be used as evidence of underestimation of $\bar{q}_{\dnn}$. If $\bar{q}_{\dnn} > q_{\dnn}$,   $\hat{d}_{\dss}$ underestimates $d_{\dss}$ unless $\dim(\mathcal H_{LRD}) = q_{\dss} \geq  \bar{q}_{\dnn} - q_{\dnn} + 1$ and thus $v_{\bar{q}_{\dnn} + 1}$ converges to a vector $v \in \mathcal H_{LRD}$. As will be shown in Tables~\ref{tab1} and~\ref{tab1add}, our testing procedure seems to rarely underestimate $q_{\dnn}$ and the magnitude of overestimation ($\bar{q}_{\dnn}-q_{\dnn}$) is quite small; particularly, note that the reported relative frequency of the occurrence of $0 \leq \bar{q}_{\dnn}-q_N \leq 1$ is close to one for all the considered sample sizes.  Therefore, as long as $\mathcal H_{LRD}$ is multidimensional (i.e. $q_{\dss}\geq 2$), $\hat{d}_{\dss}$, obtained by replacing $q_{\dnn}$ with $\bar{q}_{\dnn}$, is expected to perform well.
\end{remark}
}

\section{Numerical studies}\label{sec_simulation}

\subsection{Monte Carlo Simulation studies}\label{sec_simulation1}

\subsubsection{Simulation data generating process (DGP)}

Let $(v_1,\ldots,v_{25})$ \commWS{form an orthonormal set,} where  $(v_1,\ldots,v_{q_{\dnn}})$ is \commWS{an} orthonormal basis of $\mathcal H_{N}$,  $(v_{q_{\dnn}+1},\ldots, v_{q_{\dnn}+q_{\dss}})$ is \commWS{an} orthonormal basis of $\mathcal H_{LRD}$ and the remaining vectors are contained in $\mathcal H_{SRD}$. 

We generate the nonstationary part of the time series $P(Y_t-Y_0)= \Delta^{-\dn}_+ P X_t$ as follows:
\begin{equation*}
	\Delta^{-\dn}_+ P X_t =  \Delta^{-\dn}_+  \sum_{j=1}^{q_{\dnn}} a^N_{j,t} v_j, \quad a^N_{j,t} \sim \text{ARMA}(1,1), 
\end{equation*}  
where each of the coefficients of $\text{ARMA}(1,1)$ processes is a uniform random variable supported on $[-0.15,0.15]$, \commWS{with no dependence on any other variables}. The LRD part of the time series $(I-P)Y_t  = \Delta^{-\ds} (I-P)X_t$ is generated as follows:
\begin{equation*}
	\Delta^{-\ds} (I-P)X_t =  \Delta^{-\ds}  \sum_{j=q_{\dnn}+1}^{q_{\dnn}+q_{\dss}+1} a^L_{j,t} v_j, \quad a^L_{j,t} \sim \text{ARMA}(1, 1),
\end{equation*} 
where ARMA(1,1) processes $a^L_{j,t}$ are similarly determined as   $a_{j,t}^{N}$. The stationary part $\tilde{X}_t = (I-Q) (I-P) X_t$ is generated by the following FARMA model with banded coefficient operators:
\begin{equation*}
	\tilde{X}_t = A \tilde{X}_{t-1} + \varepsilon_t + B \varepsilon_{t-1},
\end{equation*}
where $A$ and $B$ are defined as 
\begin{equation}
	\langle v_j, Av_k \rangle \sim u_{j,k,A}^S \mathds{1}\{|j-k|\leq 2\} \quad \text{and} \quad  \langle v_j, Bv_k \rangle \sim u_{j,k,B}^S\mathds{1}\{|j-k|\leq 2\} \label{equnif}
\end{equation}
for $q_{\dnn}+q_{\dss}+1 \leq j,k \leq 25$, with sequences of uniform random variables ${u_{j,k,A}^S}$ and ${u_{j,k,B}^S}$;  ${u_{j,k,A}^S}$ and ${u_{j,k,B}^S}$ are assumed to be supported on $[-0.15,0.15]$ and independent across $j$ and $k$ (as well as of any other variables). Moreover, $\varepsilon_t$ is generated by $\varepsilon_t=\sum_{j=1}^{20} a_j v_{q_{\dnn}+q_{\dss}+j}$, where $a_j \sim N(0,0.97^{j-1})$ for $j=1,\ldots,20$. 

We set $\dn=0.95$, $\ds=0.3$, $q_{\dnn}=3$, $q_{\dss} = 2$, and let $(v_1,\ldots,v_{25})$ be the orthonormal \commWS{set} obtained by first permuting the first $5$ Fourier basis functions and then adding the next 20 basis functions, which are randomly ordered. By doing so, we fix $\mathcal H_N \oplus \mathcal H_{LRD}$ to $\spn\{v_1,\ldots,v_5\}$, but let $\mathcal H_N$ and $\mathcal H_{LRD}$ be differently realized. In all of our simulation experiments, $\mu$ is set to the quadratic function defined by $\mu(s) =  -2(s-1/2)^2 + 0.5$ for $s \in [0,1]$. 

\subsubsection{Results} \label{sec_sim_result}

We examine finite-sample properties of various estimators and tests that are considered in the previous sections. We consider the following:
\begin{enumerate}[(i)]
	\item the estimators $\hat{q}_{\dnn}$ and $\bar{q}_{\dnn}$ of $q_{\dnn}$ (Table~\ref{tab1}); 
	\item the estimator $\hat{q}_{\dss}$ of $q_{\dss}$ (Table~\ref{tab2}); 
	\item the local Whittle estimators of $\dn$ and $\ds$ (Tables~\ref{tab3} and~\ref{tab4});
	\item the coverage probability difference, which is the absolute difference between empirical and nominal coverage probabilities, as well as the interval scores (see \cite{gneiting2007}) of the confidence intervals constructed from~\eqref{eqhat02} and~\eqref{eq001aaa} (Table~\ref{tab5}).
\end{enumerate}
\phantomsection\label{rv204}The bandwidth parameter $h$ used to compute the estimator (see Proposition~\ref{prop6} and Corollary~\ref{cor3}) is set to $h=\lfloor 1+T^{0.3} \rfloor$ or $\lfloor 1+T^{0.4} \rfloor$. These are ad-hoc choices employed to assess the sensitivity of the proposed estimator. \phantomsection\label{rv205} In order to obtain $\bar{q}_{\dnn}$, we consider the variance-ratio test with $\alpha=0.5$. The test results do not appear to be sensitive to moderate changes of $\alpha$ from $0.5$, as shown in Tables~\ref{tab1addrv1},~\ref{tabsizepowerrv1} and ~\ref{tabsizepowerrv2}. 
More detailed information on implementing our statistical methods can be found in each table. \phantomsection\label{rvlabel2}\commRV{Particularly, for the local Whittle estimation, we adopt the choice of $m$ as $\lfloor 1+T^{0.65} \rfloor$, as previously proposed by \cite{Li2020} in their study of nonstationary fractionally integrated functional time series.} Some additional simulation results, including sensitivity analysis of the local Whittle estimators to the choice of $m$ and the size-power properties of the variance-ratio test, are reported in Appendix~\ref{sec_sensitivity}. 

To summarize the results, the estimator $\bar{q}_{\dnn}$ obtained from our variance-ratio testing procedure outperforms the eigenvalue-ratio  estimator which is similar to the estimator of \cite{Li2020}. This performance gap seems huge, particularly in small samples, which makes $\bar{q}_{\dnn}$ attractive in practice where we do not always have sufficiently many observations.  Given that $q_{\dnn}$ and $P$ characterize the dominant part of the time series (see, e.g., \cite{Li2020}) and they are used in inferential problems of other parameters (such as $q_{\dss}$ and $\ds$), it may be recommended for practitioners to use our testing procedure. Note that $\hat{q}_{\dnn}$ underestimates $q_{\dnn}$ significantly in small samples while $\bar{q}_{\dnn}$ does so only slightly. As may be deduced from Corollary~\ref{cor1prop4} and the fact that we are employing a $5\%$ significance level, the relative frequency of underestimation for $\hat{q}_{\dnn}$ must approach 0.05 as $T$ increases, as shown in Table~\ref{tab1}.
On the other hand, $\hat{q}_{\dss}$ does not perform quite as well in small samples (the relative frequency of correct determination is only around $30\%$  when $T=200$), but Table~\ref{tab2} shows that its performance improves as the sample size increases. The local Whittle estimator~\eqref{eqhat02}, which we propose in Section~\ref{sec_est_d}, seems to perform better in small samples than the existing competitor developed by \cite{Li2020}; even if their difference seems to converge as $T$ gets larger, this result suggests that our estimator may be a better alternative in practice where the sample size is limited. Table~\ref{tab4} shows the performances of the estimator $\hat{d}_{LW}(\langle \bar{Z}_t,\tilde{v} \rangle)$ when $\tilde{v} = \hat{v}_{q_{\dnn}+1}$ and $\tilde{v}$ is set to a linear combination of $\hat{v}_{q_{\dnn}+1}$ and $\hat{v}_{q_{\dnn}+2}$. Our simulation results show that the estimator performs better in the latter case. This may be because, in the latter case, we use the information of the other I$(\ds)$ component characterized by $v_{q_{\dnn}+2}$ in this simulation setup. \commHS{Another reason may be found in the observation of \cite{Nielsen2019} in the I$(1)$/I$(0)$ system;} obtaining $\hat{v}_{q_{\dnn}+1}$ in this statistical test may be understood as pre-estimation of $v_{q_{\dnn}+1}$ such that $\langle Y_t,v_{q_{\dnn}+1} \rangle $ is I$(\ds)$, but this estimation may not be accurate in a finite sample. Thus, sometimes $\langle Y_t,\hat{v}_{q_{\dnn}+1+j} \rangle $ for some positive $j$ may behave more like an I$(\ds)$ process than  $\langle Y_t,\hat{v}_{q_{\dnn}+1} \rangle$ does.    Lastly, the empirical coverage rates and interval scores based on our proposed method for $\dn$ or $\ds$ 
overall seem to be better than those of its supposed competitor. \phantomsection\label{rvlabel1}\commRV{Figure~\ref{fig3} (resp.\ \ref{fig4}) displays the histograms of the estimates of $\dn$ (resp.\ $\ds$) obtained from the two methods. Notably, the histograms for our proposed method tend to be better centered around the true values and exhibit a decreased occurrence of extreme values, both for $\dn$ and $\ds$. This observation, coupled with the findings from  Tables~\ref{tab3} and~\ref{tab4}, suggests that our proposed methods are attractive, particularly in small samples.}

\begin{center}
	\tabcolsep 0.15in
	\begin{longtable}{@{\extracolsep{\fill}}lccccc@{}}
		\caption{Finite sample performance of the estimators of $q_{\dnn}$} \label{tab1}\\
		\toprule 
		& \multicolumn{5}{c}{Relative frequency of correct determination of $q_{\dnn}$}\\\midrule 
		$q_{\max}$ or $K$ & Method	& $T=200$ & $T=350$ & $T=500$  & $T=1000$  \\ \midrule 
		\endfirsthead
		\toprule
		\endhead
		\hline \multicolumn{6}{r}{{Continued on next page}} \\ 
		\endfoot 
		\endlastfoot
		$4$ 	&Proposed	&  0.800 &0.914& 0.933 &0.945\\ 
		&LRS-type &0.348& 0.621& 0.775 &0.924\\ \midrule 
		$5,6$ 	&Proposed	&0.796& 0.914& 0.933& 0.945\\ 
		&LRS-type & 0.348 &0.621& 0.775 &0.924 \\ \midrule
		& \multicolumn{5}{c}{Relative frequency of underestimation of $q_{\dnn}$}\\\midrule 
		$q_{\max}$ or $K$ & Method	& $T=200$ & $T=350$ & $T=500$  & $T=1000$  \\ \midrule 
		$4$, $5$, $6$ 	&Proposed	& 0.022 &0.029& 0.040& 0.052		\\ 
		&LRS-type & 0.651& 0.380& 0.225 &0.076 \\ \midrule 
	\end{longtable}
	\vspace{-.2in}
	\begin{justify}
		{\footnotesize Notes: Based on 2,000 Monte Carlo replications. The proposed estimator is obtained from the sequential application of the variance-ratio test based on $\Lambda_{s,\alpha}^0$ with $\alpha = 0.5$ and significance level $\eta = 0.05$. Moreover, $K$ is set to $q + 2$ for each $H_0:q_{\dnn} = q$, and $H_0:q_{\dnn} = q_{\max}$ is first examined in this procedure. The LRS-type estimator is the eigenvalue-ratio estimator with tuning parameter $K$, which is considered in Proposition~\ref{prop3}. As noted in Remark~\ref{remprop3}, the eigenvalue-ratio estimator given in Proposition~\ref{prop3} is not identical to the estimator in \cite{Li2020}, but the two are very similar and can be equivalent under some choice of tuning parameters. The reported frequencies are rounded to the third decimal place, and the results are reported in the same row if there are no differences in these rounded numbers.} 
	\end{justify}
\end{center}

\begin{center}
	\tabcolsep 0.166in
	\begin{longtable}{@{\extracolsep{\fill}}lccccc@{}}
		\caption{Finite sample performance of the estimator of $q_{\dss}$}\label{tab2} \\
		\toprule
		& \multicolumn{5}{c}{Relative frequency of correct determination of $q_{\dss}$}\\ \midrule
		$h$ 	&$K$&$T=200$ & $T=350$ & $T=500$ & $T=1000$  \\ \midrule 
		\endfirsthead
		\toprule
		\endhead
		\hline \multicolumn{6}{r}{{Continued on next page}} \\ 
		\endfoot 
		\endlastfoot
		$\lfloor 1+ T^{0.3} \rfloor$ & $4$  &  0.326& 0.422& 0.512& 0.657     \\ 
		&	$5$   & 0.262& 0.369& 0.481& 0.643		\\ 
		& 	$6$  & 0.234& 0.358 &0.472& 0.645
		\\\midrule 
		$\lfloor 1+ T^{0.4} \rfloor$ &	
		$4$  &   0.340 &0.456& 0.570 &0.740   \\
		&	$5$     & 0.276& 0.425& 0.538& 0.742\\ 
		& 	$6$  &  0.262& 0.411& 0.540& 0.753	\\ \midrule
		&		\multicolumn{5}{c}{Relative frequency of underestimation of $q_{\dss}$}\\\midrule
		$h$ 	&$K$&$T=200$ & $T=350$ & $T=500$ & $T=1000$  \\ \midrule 
		$\lfloor 1+ T^{0.3} \rfloor$ & $4$  &  0.398& 0.332 &0.270 &0.151\\ 
		&	$5$   &  0.376& 0.336&0.255& 0.164		\\ 
		& 	$6$  &   0.378 &0.353& 0.285 &0.186		\\\midrule 
		$\lfloor 1+ T^{0.4} \rfloor$ &	
		$4$  & 0.402 &0.331& 0.257& 0.127\\
		&	$5$     & 0.380& 0.322& 0.256& 0.133\\ 
		& 	$6$  &  0.376& 0.340& 0.276& 0.150		\\ \bottomrule
	\end{longtable}
	\vspace{-.2in}
	\begin{justify}
		{\footnotesize Notes: Based on 2,000 Monte Carlo replications. $\overline{P}$ is set to $\sum_{j=1}^{{q}_{\dnn}} \hat{v}_j \otimes \hat{v}_j$ (see~\eqref{eqde3}). $h$ is the bandwidth parameter used in~\eqref{eqslrv} and $K$ is a positive integer introduced in Proposition~\ref{prop6}.} 
	\end{justify}
\end{center}


\begin{center}
	\tabcolsep 0.1in
	\begin{longtable}{@{\extracolsep{\fill}}lcccccc@{}}
		\caption{Simulated bias and variance of the proposed estimators of $\dn$}\label{tab3}\\
		\toprule  
		$m=\lfloor 1+T^{0.65} \rfloor$ &	Method & & $T=200$ & $T=350$ & $T=500$ & $T=1000$ \\ \midrule 
		\endfirsthead
		\toprule  
		$m=\lfloor 1+T^{0.65} \rfloor$ &	Method & & $T=200$ & $T=350$ & $T=500$ & $T=1000$ \\ \midrule 
		\endhead
		\hline \multicolumn{7}{r}{{Continued on next page}} \\ 
		\endfoot 
		\endlastfoot			
		Mean Bias &	Proposed&	&  -0.0413& -0.0262& -0.0207& -0.0084 \\  
		&	LRS-type &&  -0.1111& -0.0684& -0.0507& -0.0240	
		\\ \midrule 
		Variance &	Proposed&	&0.0103& 0.0064& 0.0050 &0.0028
		\\  
		&	LRS-type && 0.0240 &0.0146& 0.0102& 0.0042	
		\\ \midrule 
		MSE &	Proposed&	& 0.0120& 0.0071& 0.0055& 0.0029	\\  
		&	LRS-type &&0.0364& 0.0192& 0.0128& 0.0048 \\ \bottomrule
	\end{longtable}
	\vspace{-.2in}
	\begin{justify}
		{\footnotesize {Notes: Based on 2,000 Monte Carlo replications. The proposed estimator is constructed from $v=v^{(\ell)}$ which maximizes $\hat{d}_{LW} (\langle {Z}^0_t,v^{(\ell)} \rangle)$ for $v^{(\ell)} = \sum_{j=1}^3 a_{j,\ell} p_j$, where $a_{j,\ell} \sim N(1,1)$ for $\ell=1,\ldots,20$, $p_j$ is the Legendre polynomial of order $j-1$, and $a_{j,\ell}$ is independent across $j$ and $\ell$. The LRS-type estimator is  $\hat{d}_{LW} (\langle {Z}_t^0,\hat{v}_1 \rangle)$, where $\hat{v}_1$ with the leading eigenvector of $T^{-1}\sum_{t=1}^T{Z}_t^0\otimes {Z}_t^0$.}}
	\end{justify}
\end{center}

\begin{center}
	\tabcolsep 0.09in
	\begin{longtable}{@{\extracolsep{\fill}}lcccccc@{}}
		\caption{Simulated bias and variance of the proposed estimators of $\ds$}\label{tab4}\\
		\toprule  
		$m=\lfloor 1+T^{0.65} \rfloor$ &	Method & & $T=200$ & $T=350$ & $T=500$ & $T=1000$ \\ \midrule 
		\endfirsthead
		\toprule  
		$m=\lfloor 1+T^{0.65} \rfloor$ &	Method & & $T=200$ & $T=350$ & $T=500$ & $T=1000$ \\ \midrule 
		\endhead
		\hline \multicolumn{7}{r}{{Continued on next page}} \\ 
		\endfoot 
		\endlastfoot
		Mean Bias &	Proposed&	& -0.0718& -0.0442& -0.0332& -0.0210
		\\  
		&	LRS-type &&  -0.1304& -0.0867& -0.0704& -0.0498 \\ \midrule 
		Variance &	Proposed&&	0.0128 &0.0083& 0.0061& 0.0035
		\\  
		&	LRS-type && 0.0154 &0.0115& 0.0086& 0.0046
		\\ \midrule 
		MSE &	Proposed&	&    0.0179& 0.0102& 0.0072& 0.0039
		\\  
		&	LRS-type &&  0.0325& 0.0190& 0.0135& 0.0071
		\\ \bottomrule
	\end{longtable}
	\vspace{-.2in}
	\begin{justify}
		{\footnotesize {Notes: Based on 2,000 Monte Carlo replications. The estimator is given by $\hat{d}_{LW} (\langle \bar{Z}_t,\tilde{v}\rangle)$, where $\tilde{v} = \hat{v}_{q_{\dnn}+1}$ (LRS-type) or $\tilde{v}$ is set to a linear combination of $\hat{v}_{q_{\dnn}+1}$ and $\hat{v}_{q_{\dnn}+2}$ (Proposed); more specifically, in order to consider such a linear combination, we first define $v^{(\ell)} = \hat{v}_{q_{\dnn}+1}+ a_{\ell}  \hat{v}_{q_{\dnn}+2}$ (where $a_{\ell} \sim N(0,1)$ for $\ell=1,\ldots,20$ and $a_{\ell}$ is independent of $\ell$) and then $\tilde{v}$ is set to $v^{(\ell)}$ maximizing $\hat{d}_{LW} (\langle \bar{Z}_t,v^{(\ell)}\rangle)$, which is as we do for our proposed estimator of $\dn$ in Table~\ref{tab3}.  $\hat{v}_{q_{\dnn}+1}$ and $\hat{v}_{q_{\dnn}+2}$ denote the eigenvectors of $T^{-1}\sum_{t=1}^T\bar{Z}_t\otimes \bar{Z}_t$ corresponding to the $(q_{\dnn}+1)$-th and $(q_{\dnn}+2)$-th largest eigenvalues, respectively.}}
	\end{justify}
\end{center}

\begin{center}
	\tabcolsep 0.09in
	\begin{longtable}{@{\extracolsep{\fill}}lcccccc@{}}
		\caption{Coverage performance of the pointwise confidence intervals of the memory parameter estimated by the local Whittle estimators with 95\% nominal level}\label{tab5} \\
		\toprule 
		& \multicolumn{6}{c}{Coverage probability differences}\\\midrule 
		$m$ &	Target & Method& $T=200$ & $T=350$ & $T=500$ & $T=1000$ \\ \midrule 
		\endfirsthead
		\toprule 
		& \multicolumn{6}{c}{Coverage probability differences}\\\midrule 
		$m$ &	Target & Method& $T=200$ & $T=350$ & $T=500$ & $T=1000$ \\ \midrule 
		\endhead
		\hline \multicolumn{7}{r}{{Continued on next page}} \\ 
		\endfoot 
		\endlastfoot
		$\lfloor 1+ T^{0.6} \rfloor$ &	$d_{\dnn}$&Proposed	&0.0885&0.0630& 0.0600 &0.0400
\\  
&&LRS-type	&0.3370&0.2660&0.2070&  0.1660			\\  
&	$d_{\dss}$ & Proposed & 0.2090& 0.1450& 0.1125& 0.0705
\\ 
&&LRS-type	& 0.3965& 0.2995	 	& 0.2695	 	 & 0.2265
\\  
\midrule 
$\lfloor 1+ T^{0.65} \rfloor$ & $d_{\dnn}$&Proposed	&0.0995& 0.0785&0.0635& 0.0470
\\  
&&LRS-type	&0.3180&  0.2575&   0.2160 & 0.1710
\\  
&	$d_{\dss}$ &Proposed&0.1600& 0.0990& 0.0975 & 0.0370
\\ 
&&LRS-type	&0.3470 & 0.2415 &0.2010& 0.1460	
\\  
\midrule 
$\lfloor 1+ T^{0.7} \rfloor$ &	$d_{\dnn}$&Proposed	&0.1175& 0.1085& 0.0935 & 0.0980
\\  
&&LRS-type	&   0.3245&  0.2780		&0.2500&  0.2080
\\  
&$d_{\dss}$ &Proposed& 0.1410& 0.0995& 0.0645& 0.0305
\\
&&LRS-type	& 0.3200& 0.2060& 0.1710&  0.1190
\\ \midrule 
& \multicolumn{6}{c}{Interval scores} \\\midrule
$m$ &	Target & Method& $T=200$ & $T=350$ & $T=500$ & $T=1000$ \\ \midrule 
$\lfloor 1+ T^{0.6} \rfloor$ &	$d_{\dnn}$&Proposed	&  0.7854&0.5861&0.4967  &0.3677
\\  
&&LRS-type	&2.6169& 1.8270& 1.3248&0.6999	\\  
&	$d_{\dss}$ & Proposed & 1.0122&0.7729&0.6264 &0.4283 
\\ 
&&LRS-type	&1.7899&  1.4861& 1.2139  &0.8609\\  
\midrule 
$\lfloor 1+ T^{0.65} \rfloor$ & $d_{\dnn}$&Proposed	& 0.7012& 0.5169& 0.4454&   0.3322
\\  
&&LRS-type	&  2.6460&   1.8466& 1.3519 &0.7165
\\  
&	$d_{\dss}$ &Proposed& 0.8450&  0.5899&  0.4663 &  0.3201 
\\ 
&&LRS-type	&  1.6781 &   1.1657&   0.9037 & 0.5591	
\\  
\midrule 
$\lfloor 1+ T^{0.7} \rfloor$ &	$d_{\dnn}$&Proposed	&  0.6848&  0.5266 & 0.4305 &   0.3453 	\\  
&&LRS-type	& 2.7783& 2.0358&  1.5571&  0.9048
\\  
&$d_{\dss}$ &Proposed& 0.7606&  0.5007&   0.3736&    0.2571 
\\
&&LRS-type	& 1.5874&0.9978& 0.7199 &0.4395 \\\bottomrule
	\end{longtable}
	\vspace{-.2in}
	\begin{justify}
		{\footnotesize {Notes: Based on 2,000 Monte Carlo replications. The estimators are computed as in Tables \ref{tab3} and \ref{tab4}. The reported number in each case is computed as the absolute value of the difference between the computed coverage rate and the nominal level 0.95. The interval score in each case is computed as in \citet[][Section 6.2]{gneiting2007} with the quantiles 0.025 and 0.975. An estimator with smaller interval scores is regarded as better.}} 
\end{justify}
\end{center}

\begin{figure}[!htb]
\centering \caption{Histograms of estimates of $\dn$} \label{fig3} 
\begin{subfigure}{.32\linewidth}
	\includegraphics[width = \linewidth, trim={0 0 0 5em},clip]{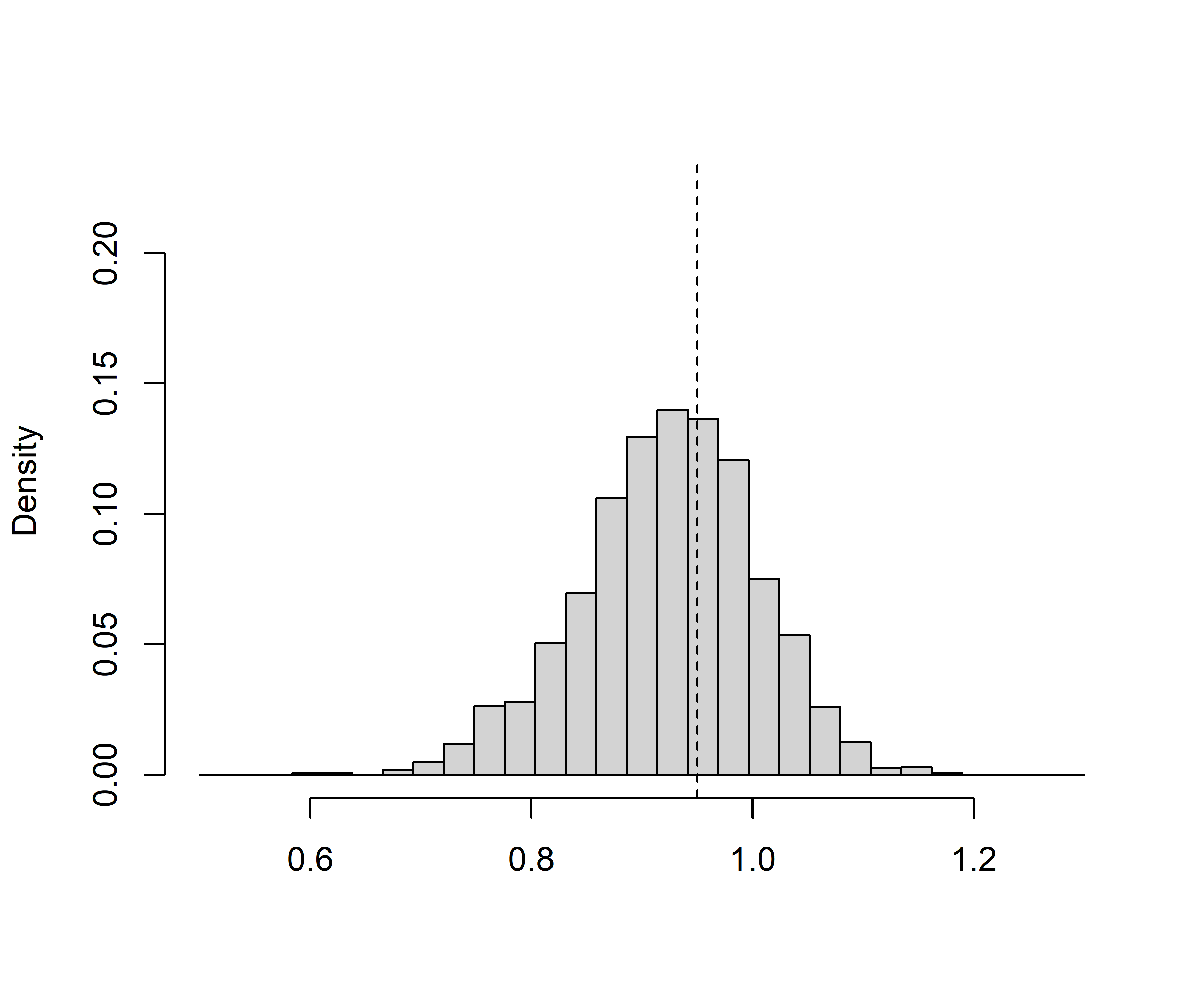}
	\subcaption{Proposed, $T=350$}
\end{subfigure}
\begin{subfigure}{.32\linewidth}
	\includegraphics[width = \linewidth, trim={0 0 0 5em},clip]{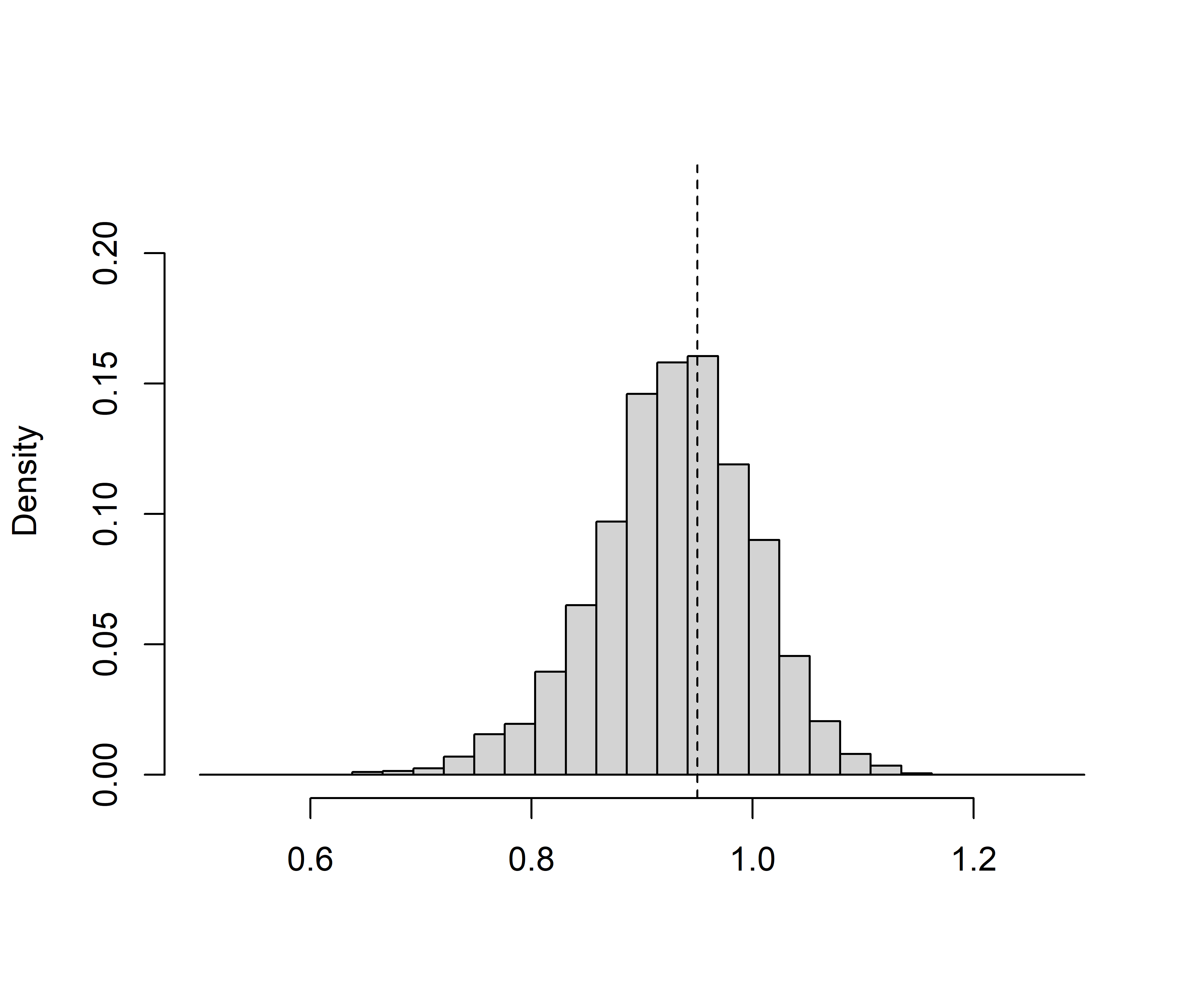}
	\subcaption{Proposed, $T=500$}
\end{subfigure}
\begin{subfigure}{.32\linewidth}
	\includegraphics[width = \linewidth, trim={0 0 0 5em},clip]{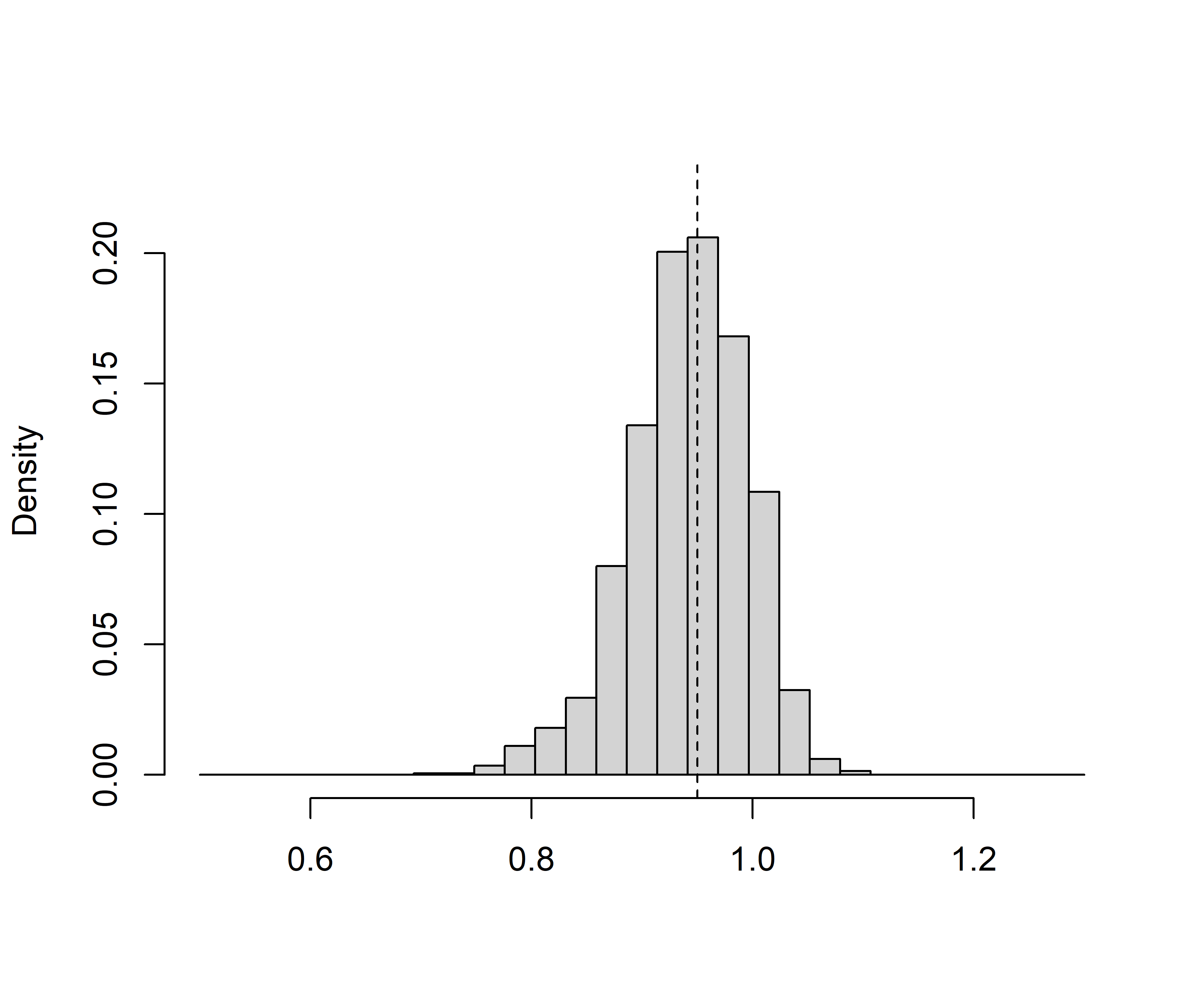}
	\subcaption{Proposed, $T=1000$}
\end{subfigure} \\
\begin{subfigure}{.32\linewidth}
	\includegraphics[width = \linewidth, trim={0 0 0 5em},clip]{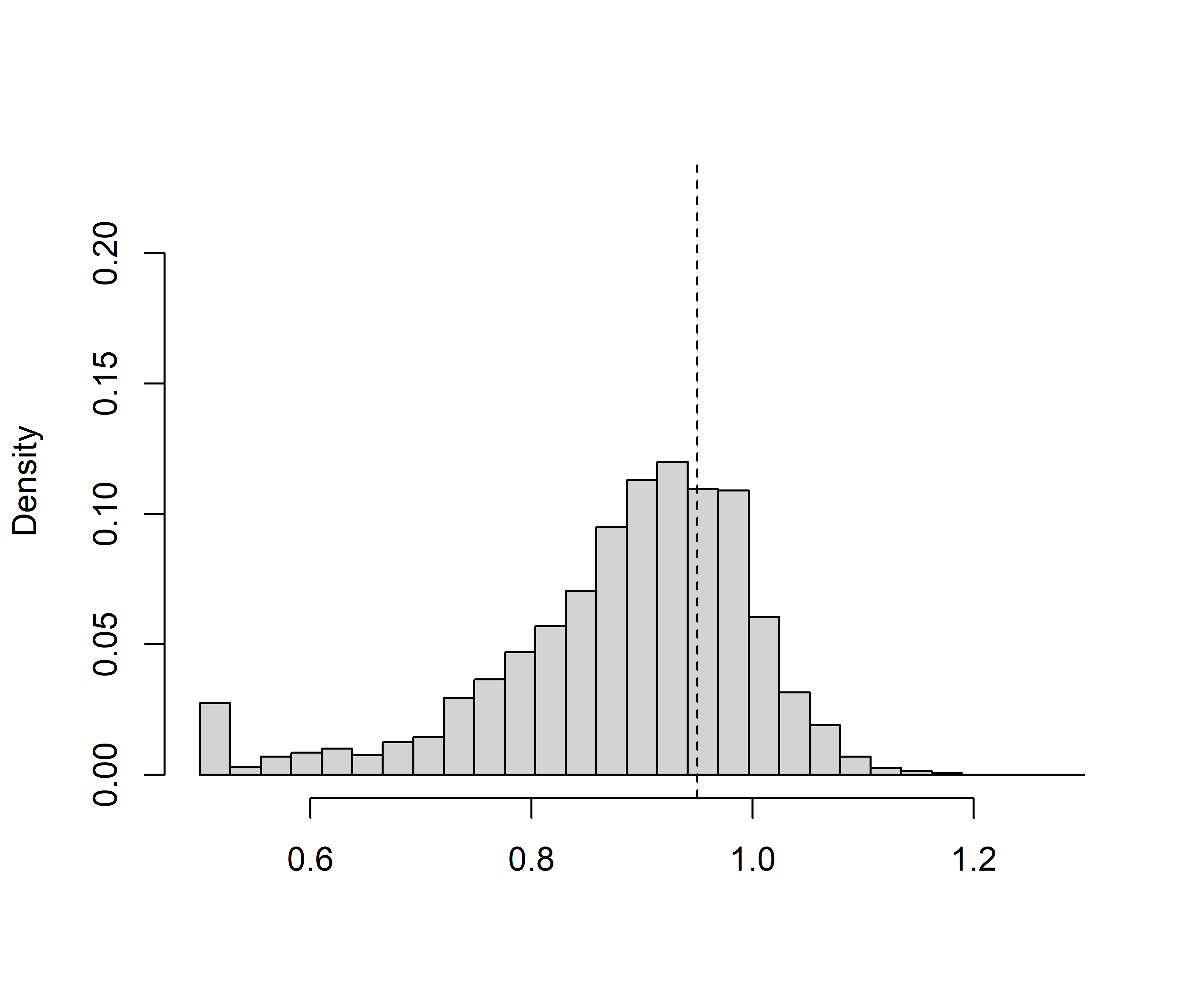}
	\subcaption{LRS-type, $T=350$}
\end{subfigure}
\begin{subfigure}{.32\linewidth}
	\includegraphics[width = \linewidth, trim={0 0 0 5em},clip]{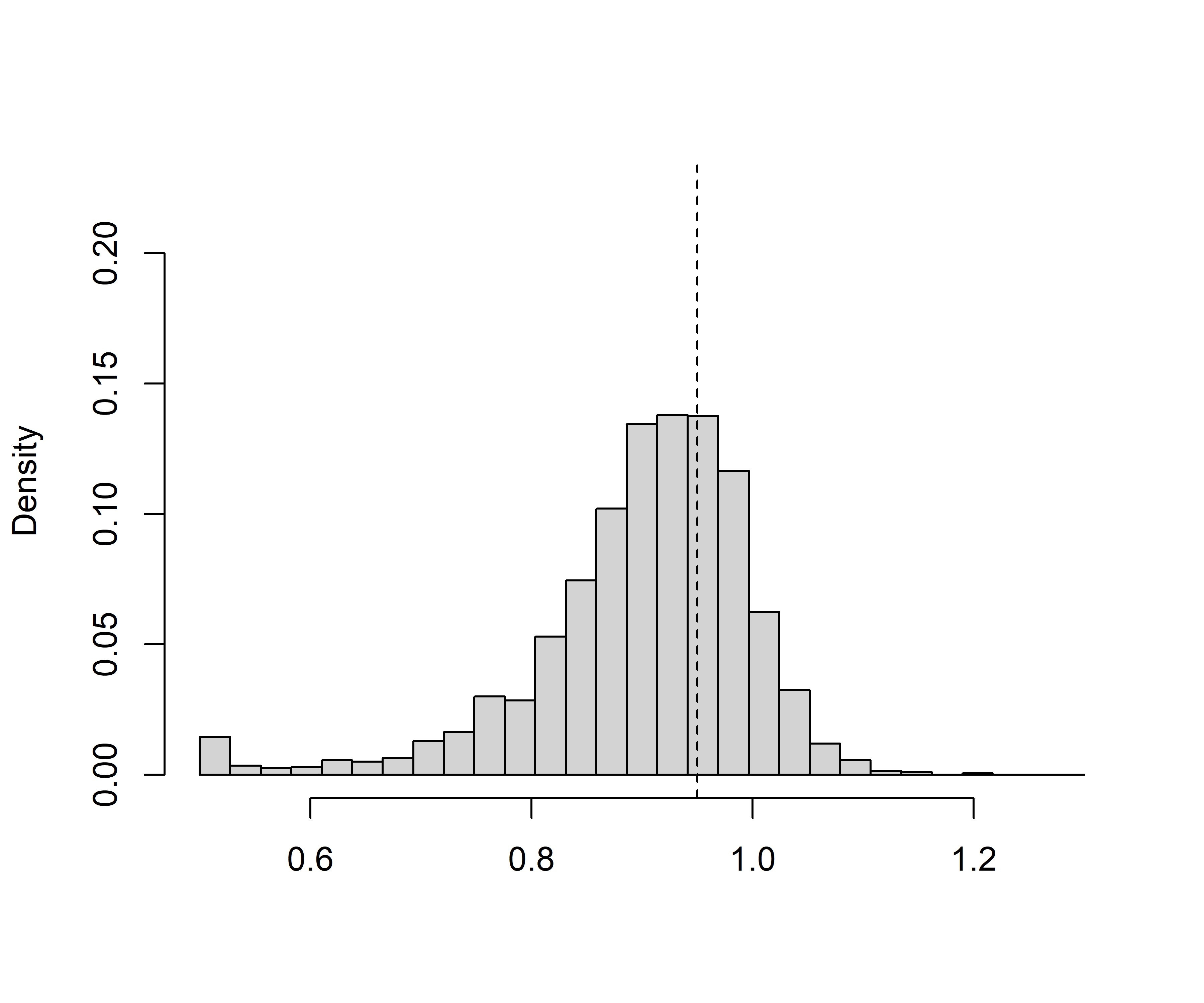}
	\subcaption{LRS-type, $T=500$}
\end{subfigure}
\begin{subfigure}{.32\linewidth}
	\includegraphics[width = \linewidth, trim={0 0 0 5em},clip]{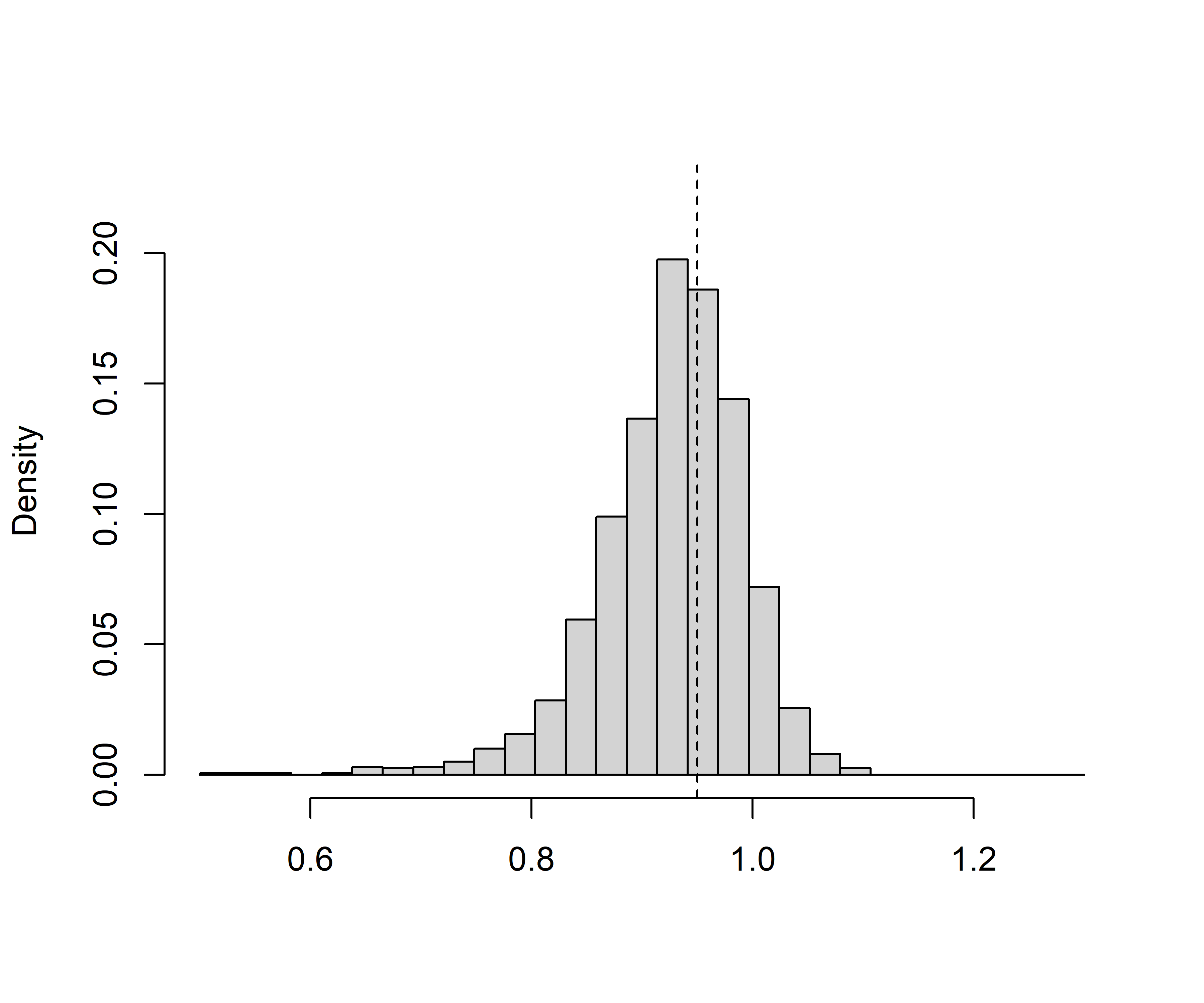}
	\subcaption{LRS-type, $T=1000$}
\end{subfigure} 
\begin{justify}
	{\footnotesize{Notes: The histograms are computed from the Monte Carlo replications used in Table~\ref{tab3}.}}
\end{justify}
\end{figure}

\begin{figure}[!htb]
\centering  \caption{Histograms of estimates of $\ds$} \label{fig4}
\begin{subfigure}{.32\linewidth}
	\includegraphics[width = \linewidth, trim={0 0 0 5em},clip]{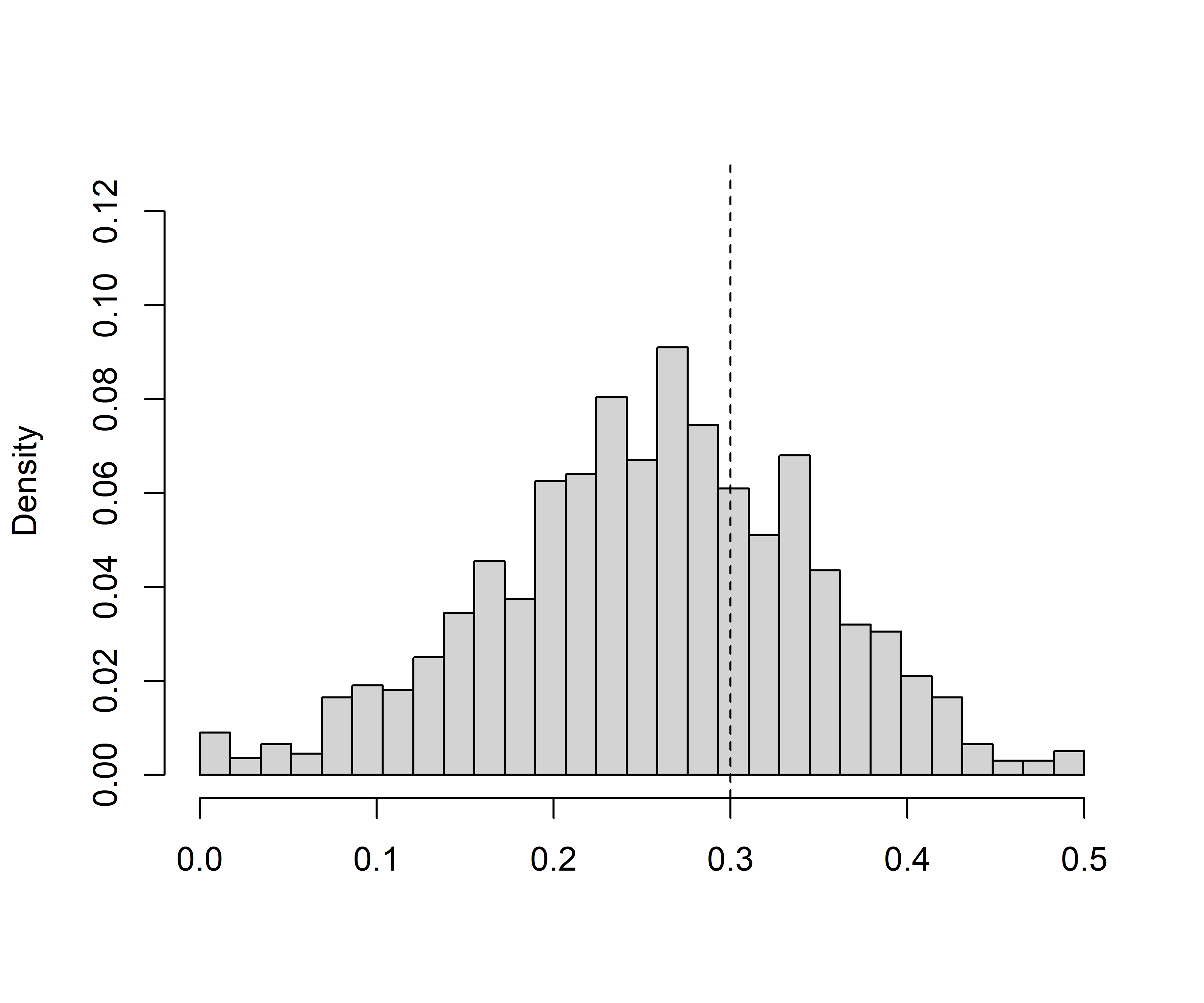}
	\subcaption{Proposed, $T=350$}
\end{subfigure}
\begin{subfigure}{.32\linewidth}
	\includegraphics[width = \linewidth, trim={0 0 0 5em},clip]{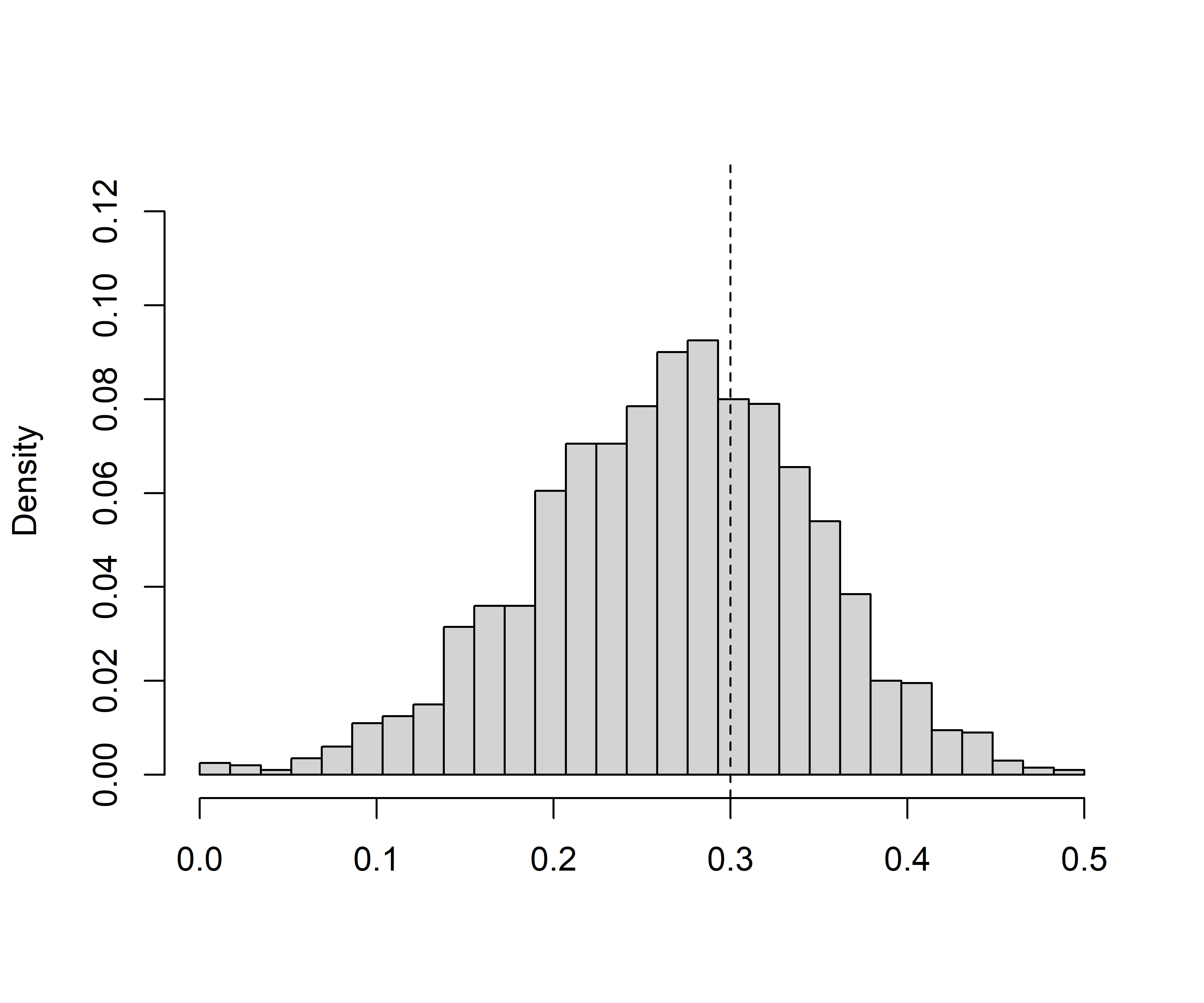}
	\subcaption{Proposed, $T=500$}
\end{subfigure}
\begin{subfigure}{.32\linewidth}
	\includegraphics[width = \linewidth, trim={0 0 0 5em},clip]{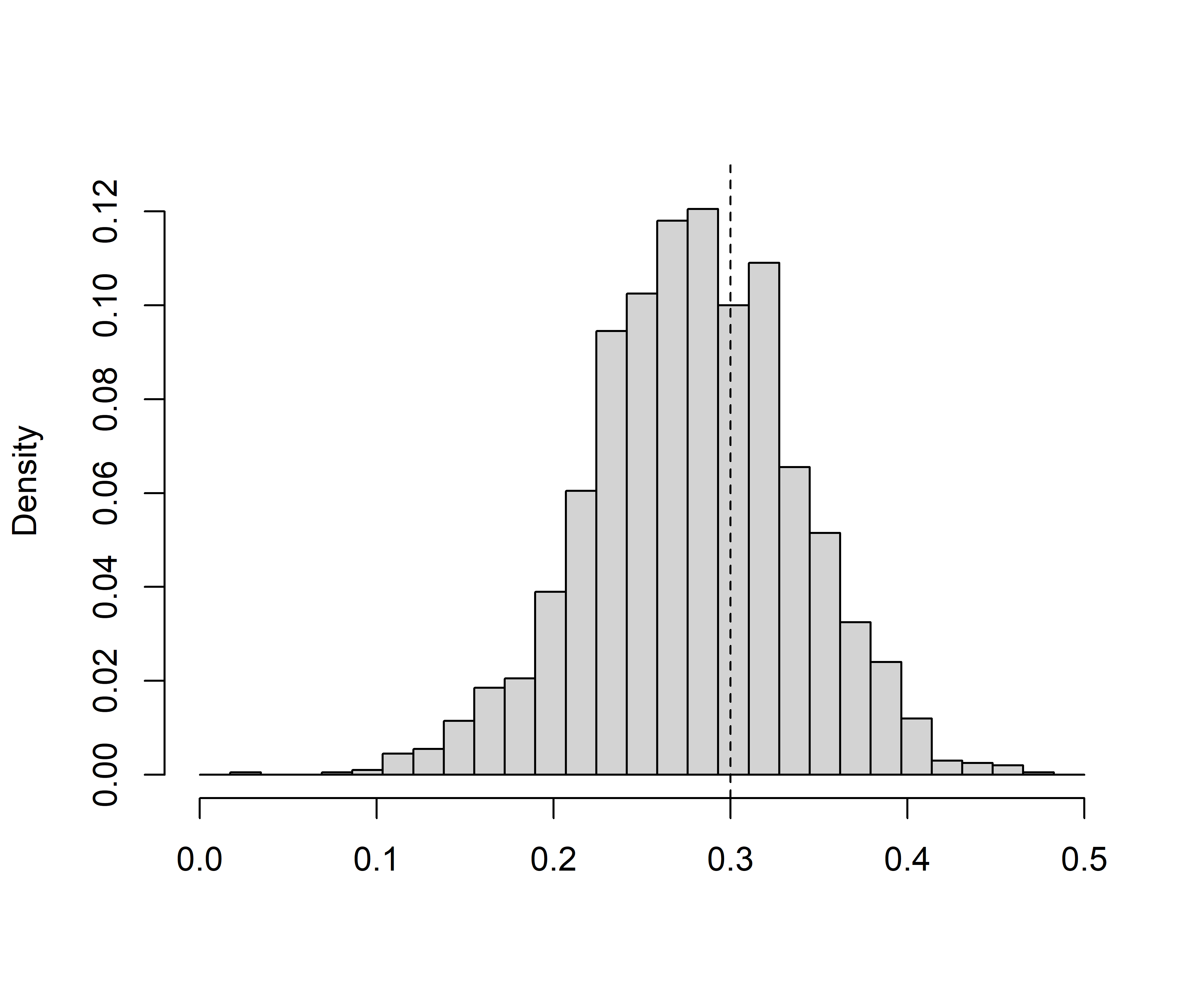}
	\subcaption{Proposed, $T=1000$}
\end{subfigure} \\
\begin{subfigure}{.32\linewidth}
	\includegraphics[width = \linewidth, trim={0 0 0 5em},clip]{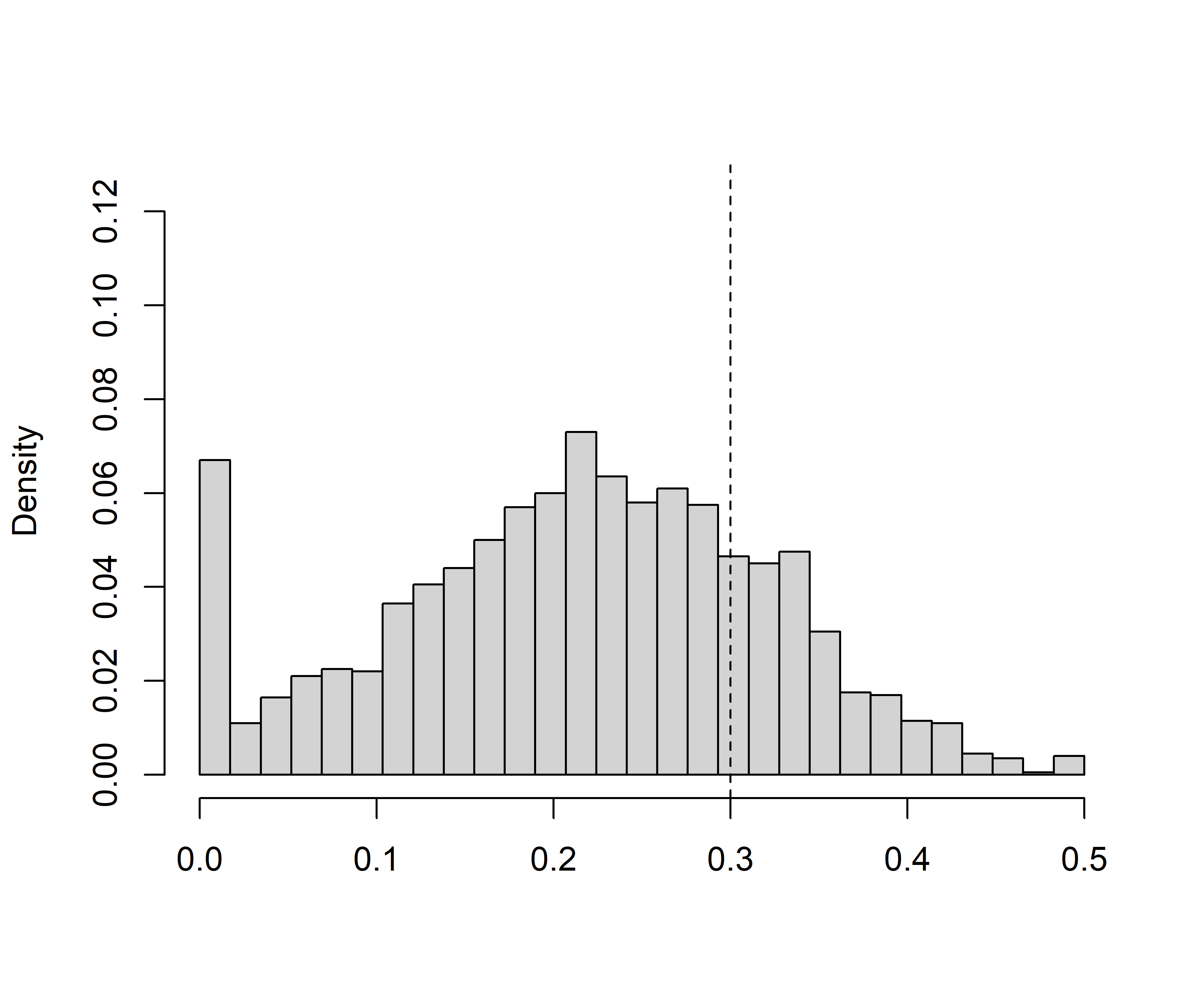}
	\subcaption{LRS, $T=350$}
\end{subfigure}
\begin{subfigure}{.32\linewidth}
	\includegraphics[width = \linewidth, trim={0 0 0 5em},clip]{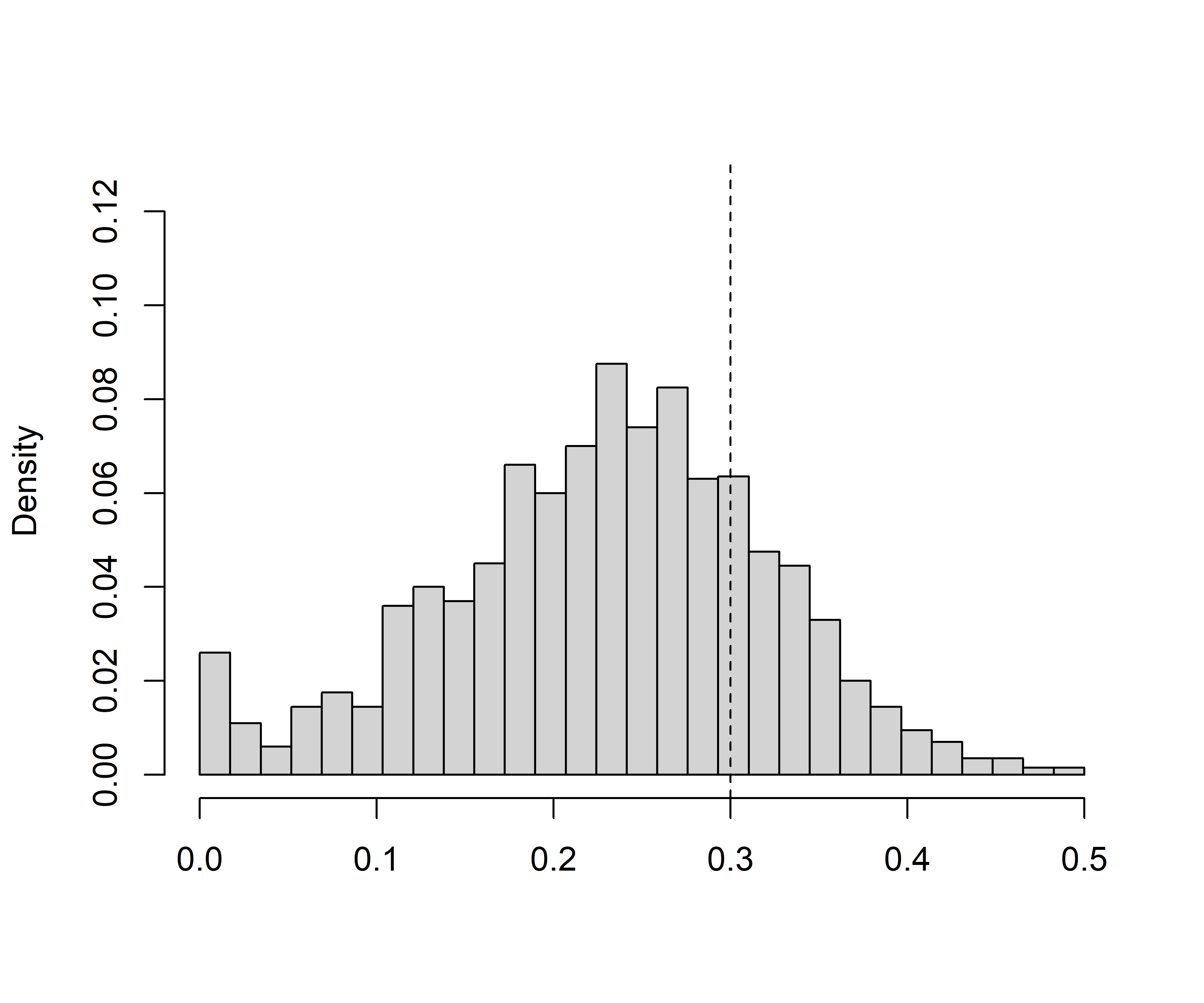}
	\subcaption{LRS, $T=500$}
\end{subfigure}
\begin{subfigure}{.32\linewidth}
	\includegraphics[width = \linewidth, trim={0 0 0 5em},clip]{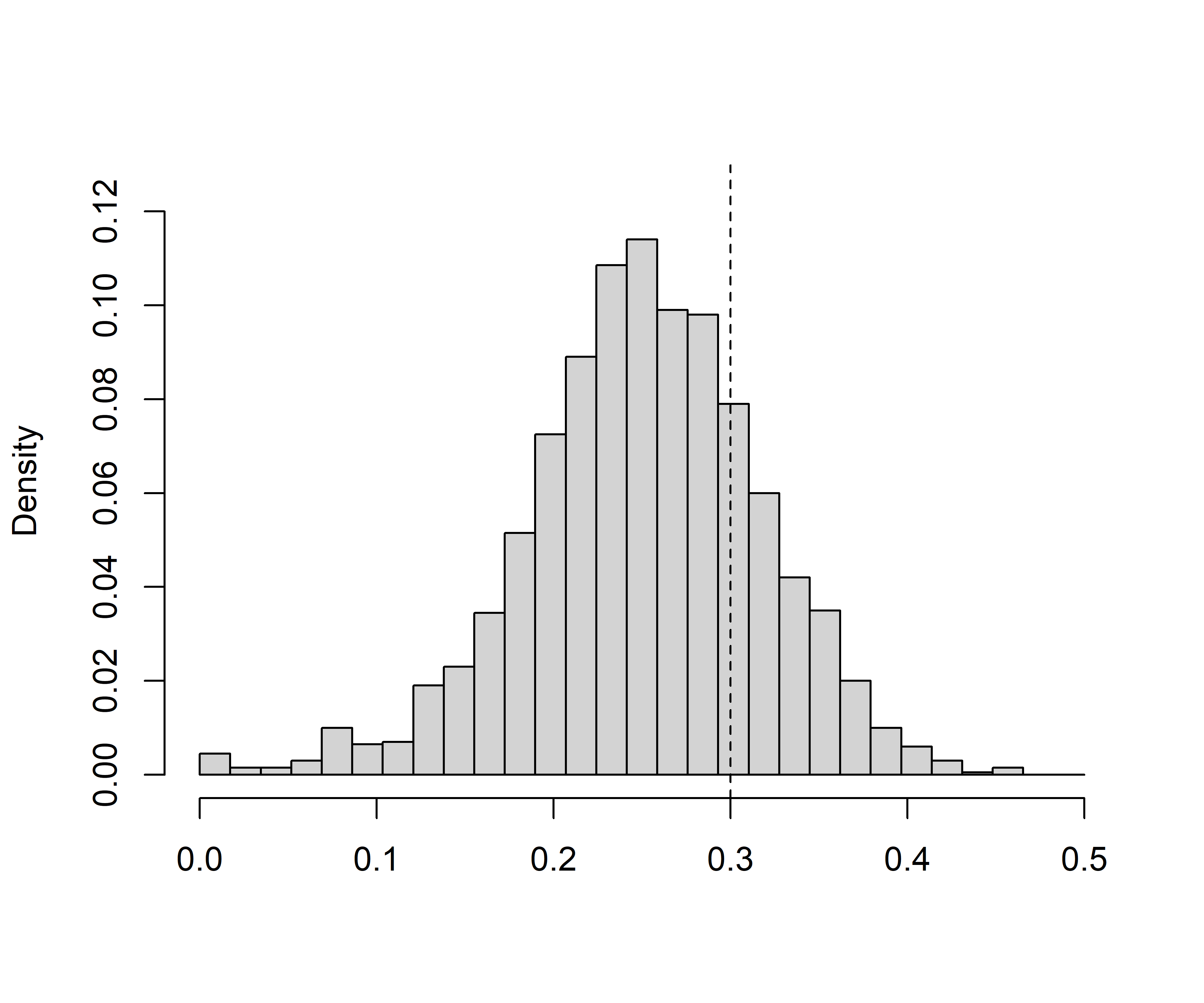}
	\subcaption{LRS, $T=1000$}
\end{subfigure}
\begin{justify}
	{\footnotesize{Notes: The histograms are computed from the Monte Carlo replications used in Table~\ref{tab4}.}}
\end{justify}
\end{figure}

\subsection{Empirical application -- Swedish age-specific mortality rates}\label{sec:mortality}

We apply our methodology to age- and gender-specific mortality data for Sweden observed from 1751 to 2021 \phantomsection\label{rvv2sample1}{($T=271$)}; the data used in this section is available from the Human Mortality Database at \url{https://www.mortality.org/}, and we specifically use the central mortality rates, which are observed at various ages from $0$ to $110$ (and older) for each gender over time. 
Viewing the mortality rates at various ages as functional observations as in, for example, \cite{Hyndman2007}, \cite{shang2016}, and \cite{shang2017grouped}, we may apply our inferential methods to the considered data. As in the aforementioned literature, we hereafter consider the natural logarithms of the observed mortality rates for each gender, which are visualized in Figure~\ref{fig1}.
\begin{figure}[!htb]
\centering
\begin{subfigure}{.495\linewidth}
	\includegraphics[width = \linewidth, trim={0 0 0 5em},clip]{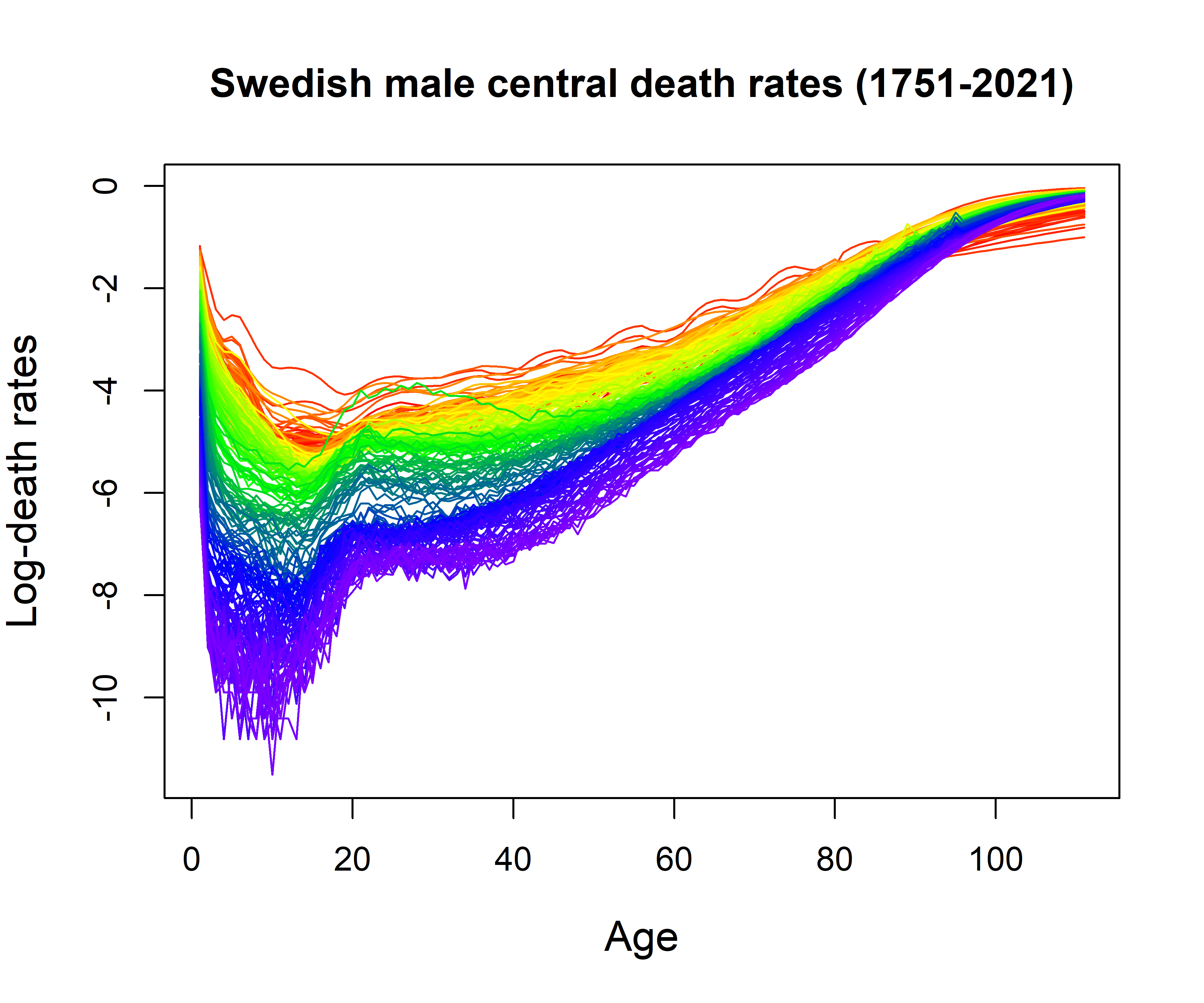}
	\subcaption{Male data}
\end{subfigure}
\begin{subfigure}{.495\linewidth}
	\includegraphics[width = \linewidth, trim={0 0 0 5em},clip]{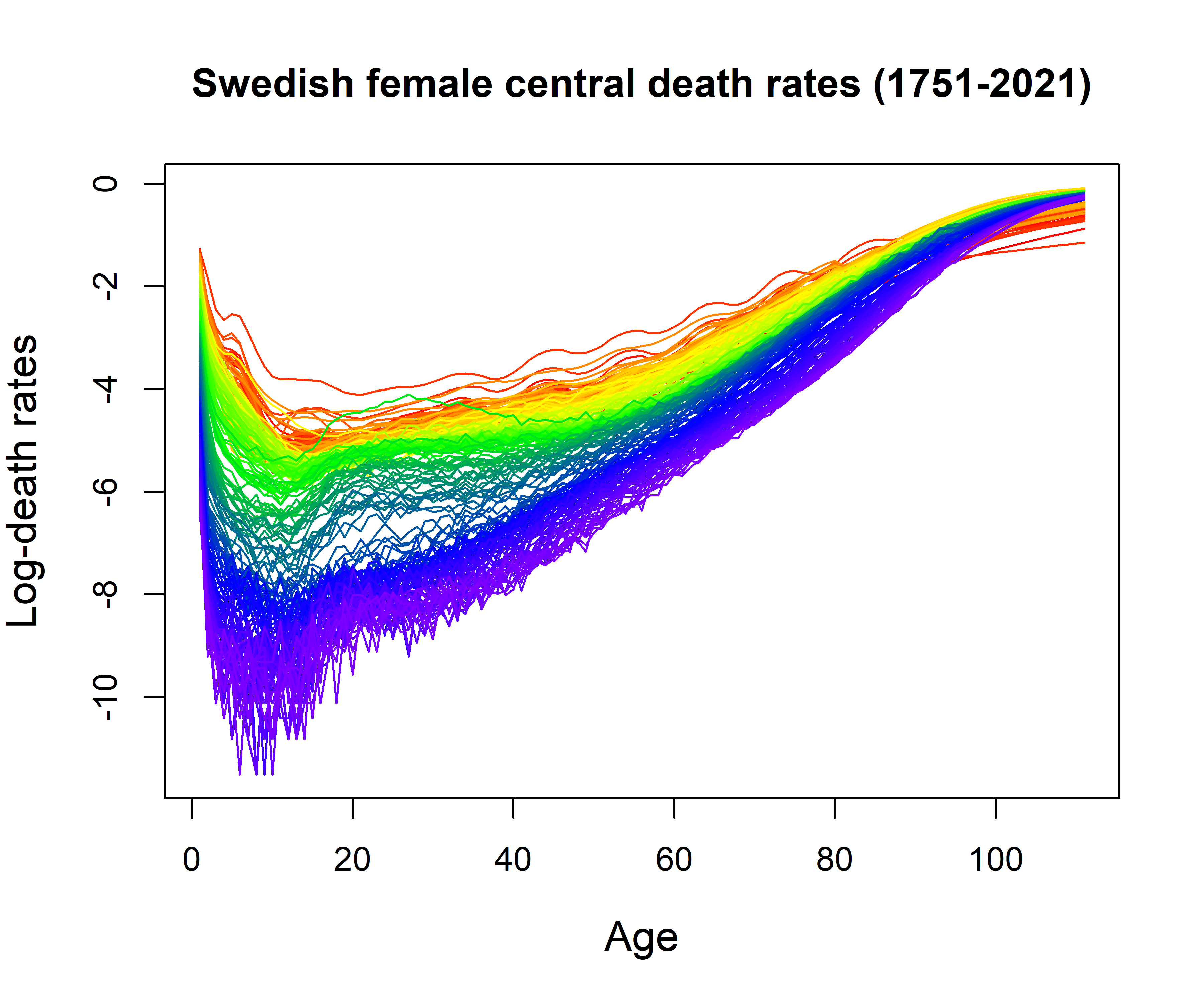}
	\subcaption{Female data}
\end{subfigure}
\caption{Log-mortality rates at various ages} \label{fig1}
\begin{justify}
	{\footnotesize{Notes: The data for a specific year and gender is given by a 111-dimensional vector of mortality rates from age 0 to 110 (and older), and each of such vectors is plotted as a function of age. Zero mortality rates are reported once (resp.\ six times) in the male (resp.\ female) data over the entire time span and all ages, which are replaced by $10^{-5}$ in order to have bounded log-mortality rates.}}
\end{justify}
\end{figure}

For our statistical analysis, we first represent the observed mortality rates at various ages for each gender with~40 Legendre polynomial basis functions; \phantomsection\label{rvlabel3}\commRV{we conducted a sensitivity analysis by varying the number of basis functions from~35 to~55 with an increment of~5, and found that the estimates obtained from our proposed methods are not sensitive to changes in the number of basis functions within this range (only the estimate of $q_{\dss}$ for the male data slightly changes from 3 to 5 as the number of basis functions changes).} We first estimate the memory parameter $\dn$ of the time series for each gender. The top rows of Table~\ref{tabemp1} report the local Whittle estimation results. As is not uncommon in many empirical applications, the memory of each time series is far greater than $1/2$ and quite close to unity. This not only implies that both time series of mortality rates are nonstationary but also justifies, to some degree, the conventional use of the random walk model for mortality in the literature. We then apply our variance-ratio testing procedure to estimate the dimension $q_{\dnn}$ of the dominant subspace for each time series. Of course, to implement the proposed testing procedure, the asymptotic null distribution of the test statistic, which depends on $\dn$, needs to be approximated by a feasible estimate of $\dn$ (see Remark~\ref{remdalpha}). This is done by replacing $\dn$ with the relevant estimate obtained by our proposed local Whittle method (see Table~\ref{tabemp1}). 

\begin{table}[!htb]
\caption{Local Whittle estimation-Swedish mortality data.} \label{tabemp1}
\begin{tabular*}{1\linewidth}{@{\extracolsep{\fill}}lccc@{}}
	\toprule
	Target	& Method  &  Male &  Female  \\ \midrule	
	$\dn$ 	&  Proposed  & 0.962 & 0.989  \\ 
	&  LRS-type  & 0.956& 0.978   \\ \midrule	
	$\ds$ 	&  Proposed  & 0.424 & 0.433  \\ 
	&  LRS-type  & 0.402& 0.275  \\ \bottomrule	
\end{tabular*} 
\begin{justify}
{\footnotesize{Notes: The proposed and LRS-type estimators of $\dn$ are given as in Table~\ref{tab3}, and the bandwidth $m$ is set to $\lfloor 1+T^{0.65}\rfloor$. The estimators of $\ds$ are given as in  Table~\ref{tab4} with $m=\lfloor 1+T^{0.65}\rfloor$ but $\hat{v}_{q_{\dnn+1}}$ and $\hat{v}_{q_{\dnn+2}}$ are replaced by $\hat{v}_{\hat{q}_{\dnn+1}}$ and $\hat{v}_{\hat{q}_{\dnn+2}}$, where $\hat{q}_{\dnn+1}$ is the estimator obtained by our variance-ratio testing procedure.}}
\end{justify}
\end{table}

\begin{table}[!htb]
\caption{Dimension estimation-Swedish mortality data.} \label{tabemp2}
\begin{tabular*}{1\linewidth}{@{\extracolsep{\fill}}lccc@{}}
	\toprule
	Target	& Method  &  Male &  Female  \\ \midrule	
	$q_{\dnn}$ 	&  Proposed  &5 & 5 \\ 
	&  LRS-type  & 1& 1   \\ \midrule	
	$q_{\dss}$ 	&  Proposed  & 5 & 1  \\ \bottomrule	
\end{tabular*} 
\begin{justify}
{\footnotesize{Notes: The proposed estimator of $q_{\dnn}$ is obtained by our variance-ratio testing procedure as in Table~\ref{tab1}}, and $K=q + 2$, for each $H_0:q_{\dnn}=  q$, $\eta = 0.05$, and $q_{\max}$ is set to $6$. The LRS-type estimator of $q_{\dnn}$  is the same as that in Table~\ref{tab1}, and the tuning parameter $K$ is set to $6$. The proposed estimator of $q_{\dss}$ is given as in Table~\ref{tab2}, and \commRV{$h$ is set to $\lfloor 1+ T^{0.4} \rfloor$.}}
\end{justify}
\end{table}

The testing results are reported in the top rows of Table~\ref{tabemp2}, and for comparison, we also report the eigenvalue-ratio estimates ($\hat{q}_{\dnn}$), which are considered in Section~\ref{sec_fracdecom}. The estimated dimension of the dominant subspace by our proposed testing procedure is 5 for each case, but the eigenvalue-ratio estimate is given by 1 for each case. As may be deduced from the simulation studies in \cite{Li2020} considering a similar eigenvalue-ratio estimator (see Section~5 of their paper), this estimator tends to underestimate $q_{\dnn}$ in small samples, and our unreported simulation results also support this; for example, in  Table~\ref{tab1} based on our simulation setting with $q_{\dnn}=3$, the relative frequency of underestimation is {0.651} when $T=200$. Given this evidence and our earlier observation that the proposed variance-ratio testing procedure performs better in our simulation studies, we are inclined to conclude that $q_{\dnn} = 5$. Then, the dominant subspace may be estimated by the span of the eigenvectors corresponding to the first five largest eigenvalues of the sample covariance operator, as discussed in the previous sections.  

Assuming that the additional conditions given in Assumption~\ref{assum2add} hold, we may also estimate $\ds$ and $q_{\dss}$ using the proposed methods, which are, respectively, reported in the bottom rows of Tables~\ref{tabemp1} and~\ref{tabemp2}. Of course, these results might not be meaningful if Assumption~\ref{assum2add} is not satisfied, and, moreover, it may be hard to check if this assumption holds in practice. On top of all these estimation results, we report the time series of $\langle Z_t^0, \hat{v}_{j} \rangle$ for a few selected values of $j$   in Figure~\ref{fig2}, where $\hat{v}_j$ is the eigenvector corresponding to the $j$-th largest eigenvalue of $T^{-1} \sum_{t=1}^T Z^0_T \otimes Z^0_T$; specifically, $j$ is chosen so that each time series has a different integration order based on our estimation results given in Table~\ref{tabemp2} (see Section~3 of \cite{Li2020}). From the previous estimation results, we expect that, in Figure~\ref{fig2}, the persistence of the time series tends to be higher in the left panel ((a) and (d))  and lower in the right panel ((c) and (f)). It is quite clear that the time series reported in the left panel tend to be more persistent than those in the other panels, but it is less clear if the time series in the middle panel are more persistent than those in the right panel. This may be due to violation of the additional assumptions given in Assumption~\ref{assum2add} (i.e., the SRD component may not be clearly distinguishable by a single projection operator $Q$ as assumed in  Assumption~\ref{assum2add})  or insufficient sample size that does not guarantee good performance of our proposed statistical methods for $\ds$ and $q_{\dss}$ (see Section~\ref{sec_sim_result}). 

\begin{figure}[!htb]
\centering
\begin{subfigure}{.32\linewidth}
	\includegraphics[width = \linewidth, trim={0 0 0 4em},clip]{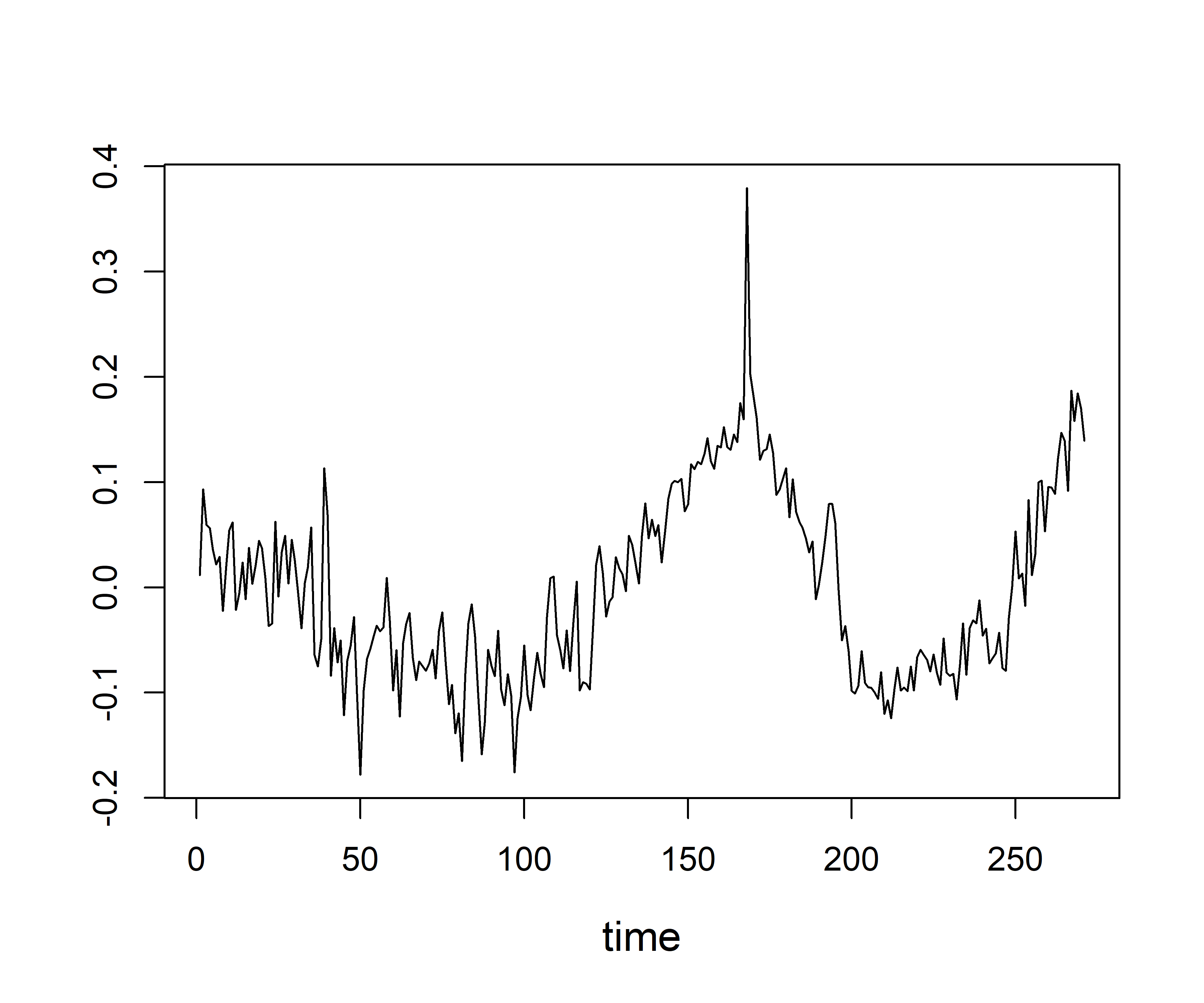}
	\subcaption{$\langle Z_t^0, \hat{v}_{3} \rangle$-male data}
\end{subfigure}
\begin{subfigure}{.32\linewidth}
	\includegraphics[width = \linewidth, trim={0 0 0 4em},clip]{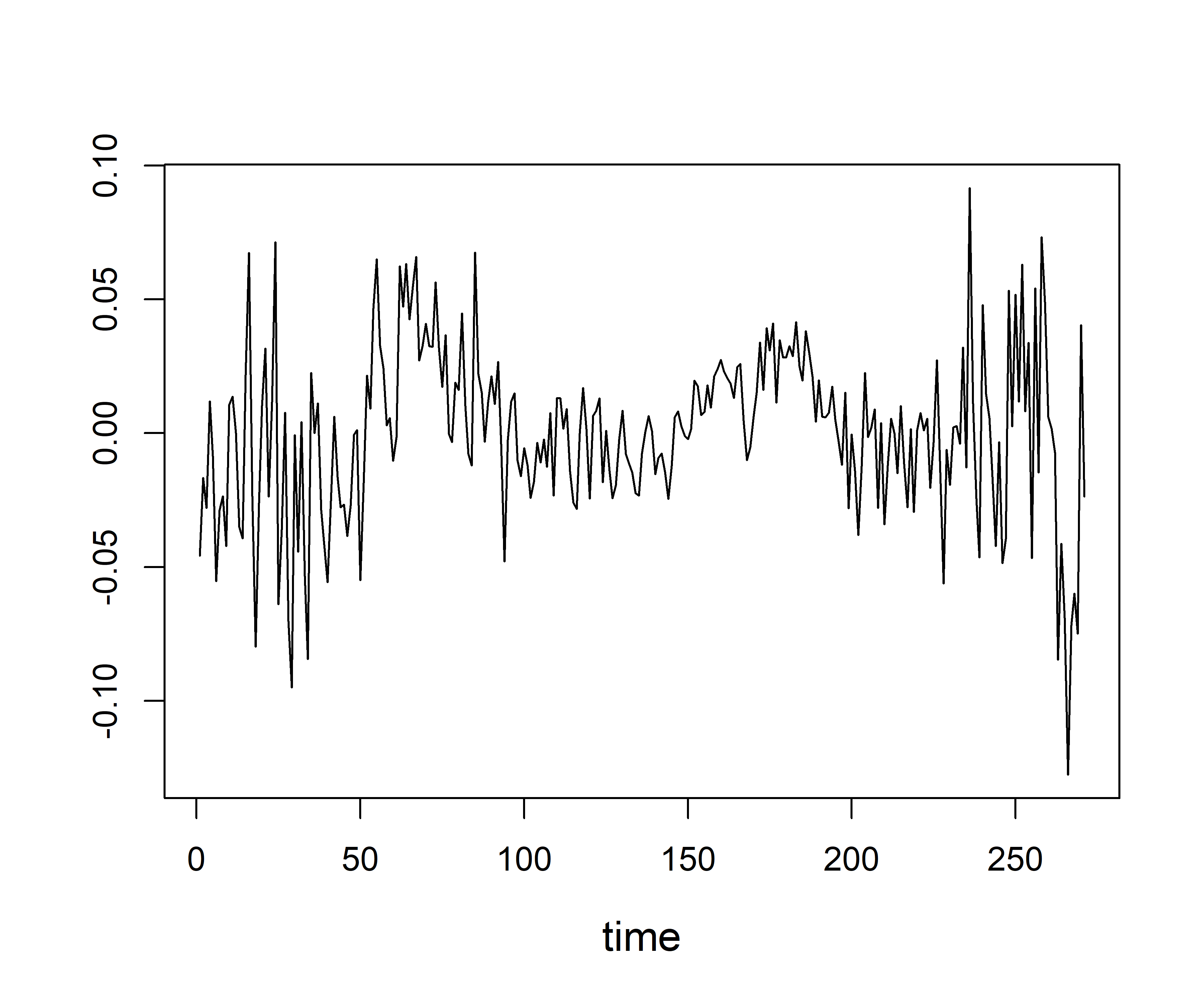}
	\subcaption{$\langle Z_t^0, \hat{v}_{6} \rangle$-male data}
\end{subfigure}
\begin{subfigure}{.32\linewidth}
	\includegraphics[width = \linewidth, trim={0 0 0 4em},clip]{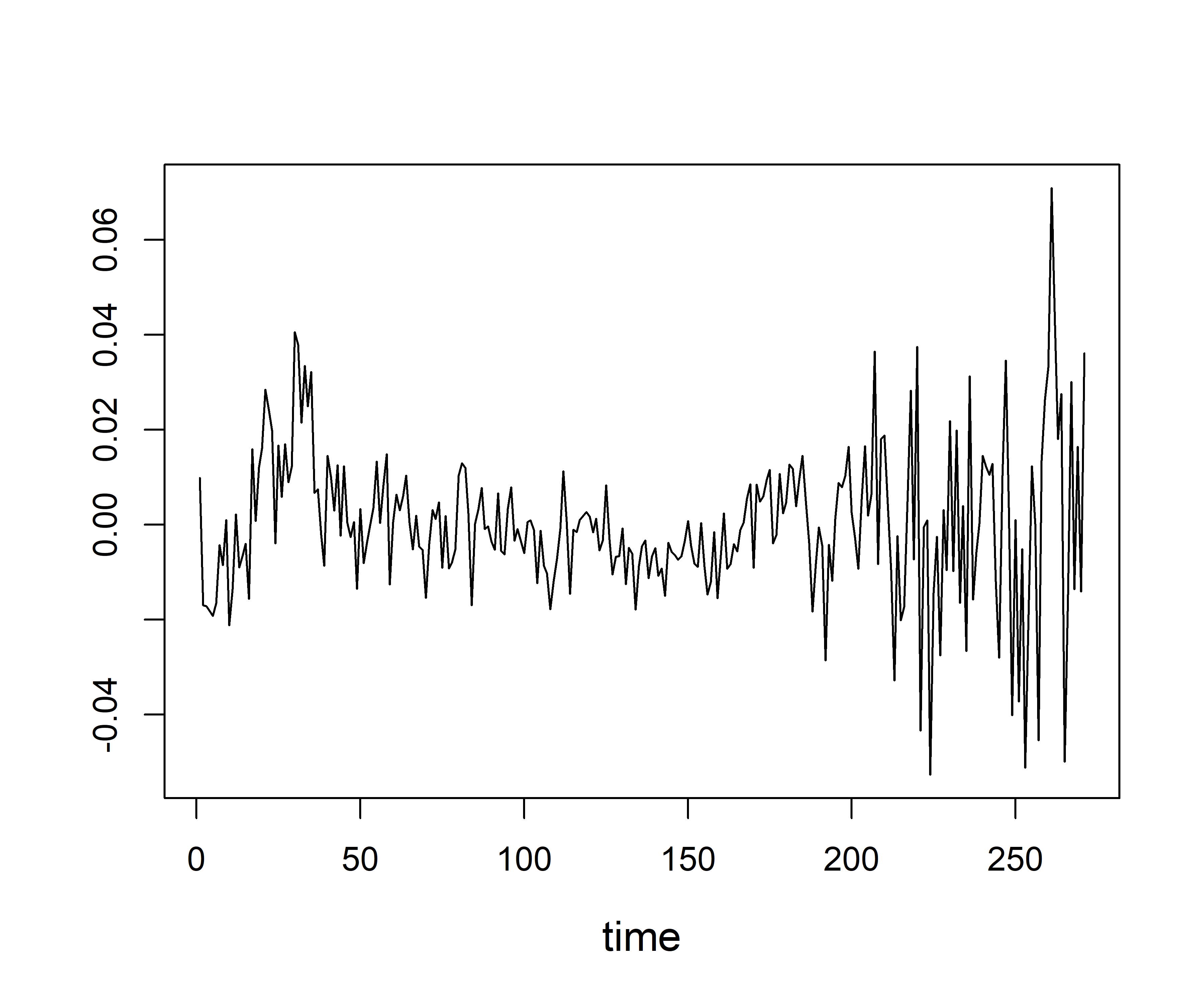}
	\subcaption{$\langle Z_t^0, \hat{v}_{12} \rangle$-male data}
\end{subfigure} \\
\begin{subfigure}{.32\linewidth}
	\includegraphics[width = \linewidth, trim={0 0 0 4em},clip]{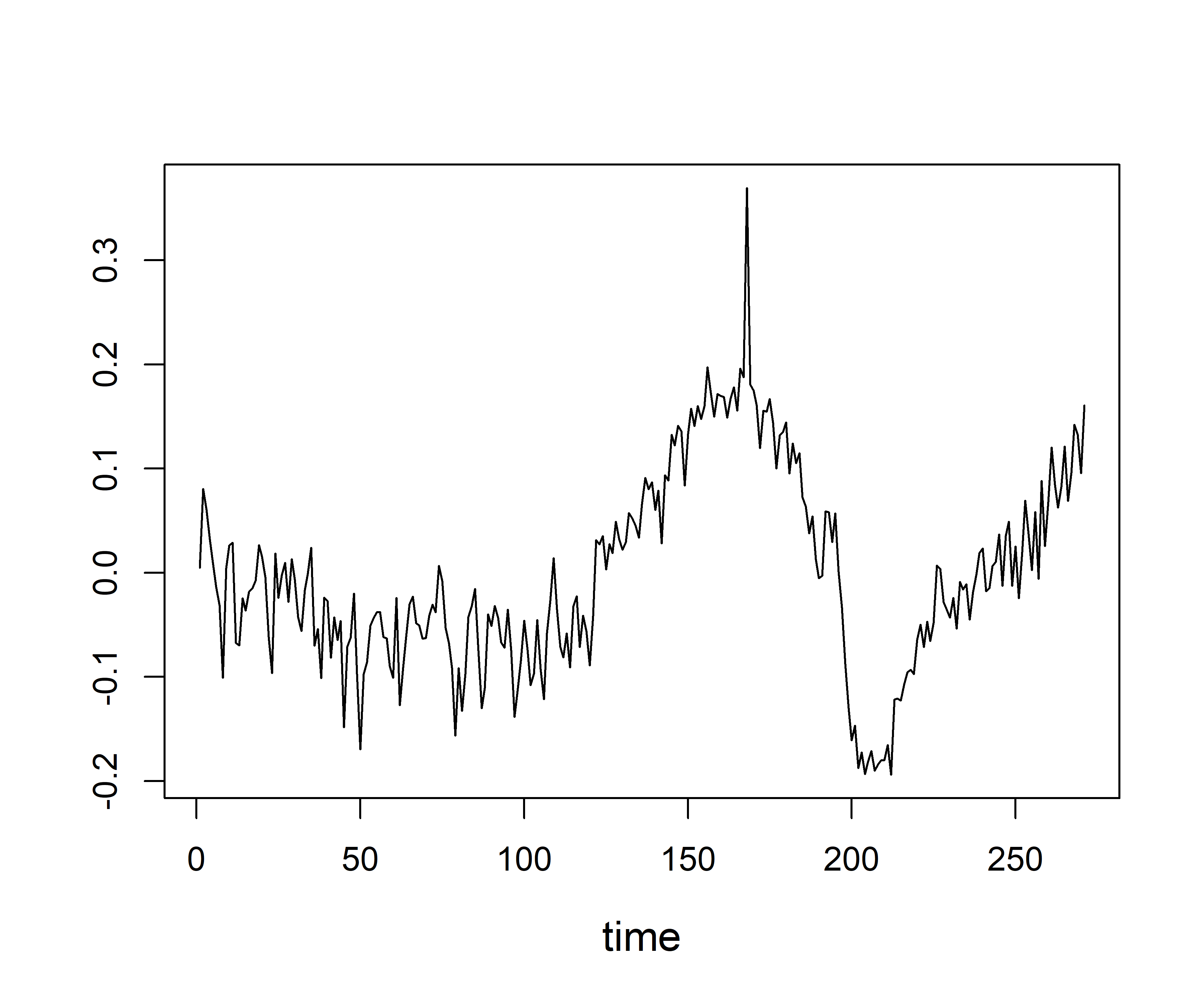}
	\subcaption{$\langle Z_t^0, \hat{v}_{3} \rangle$-female data}
\end{subfigure}
\begin{subfigure}{.32\linewidth}
	\includegraphics[width = \linewidth, trim={0 0 0 4em},clip]{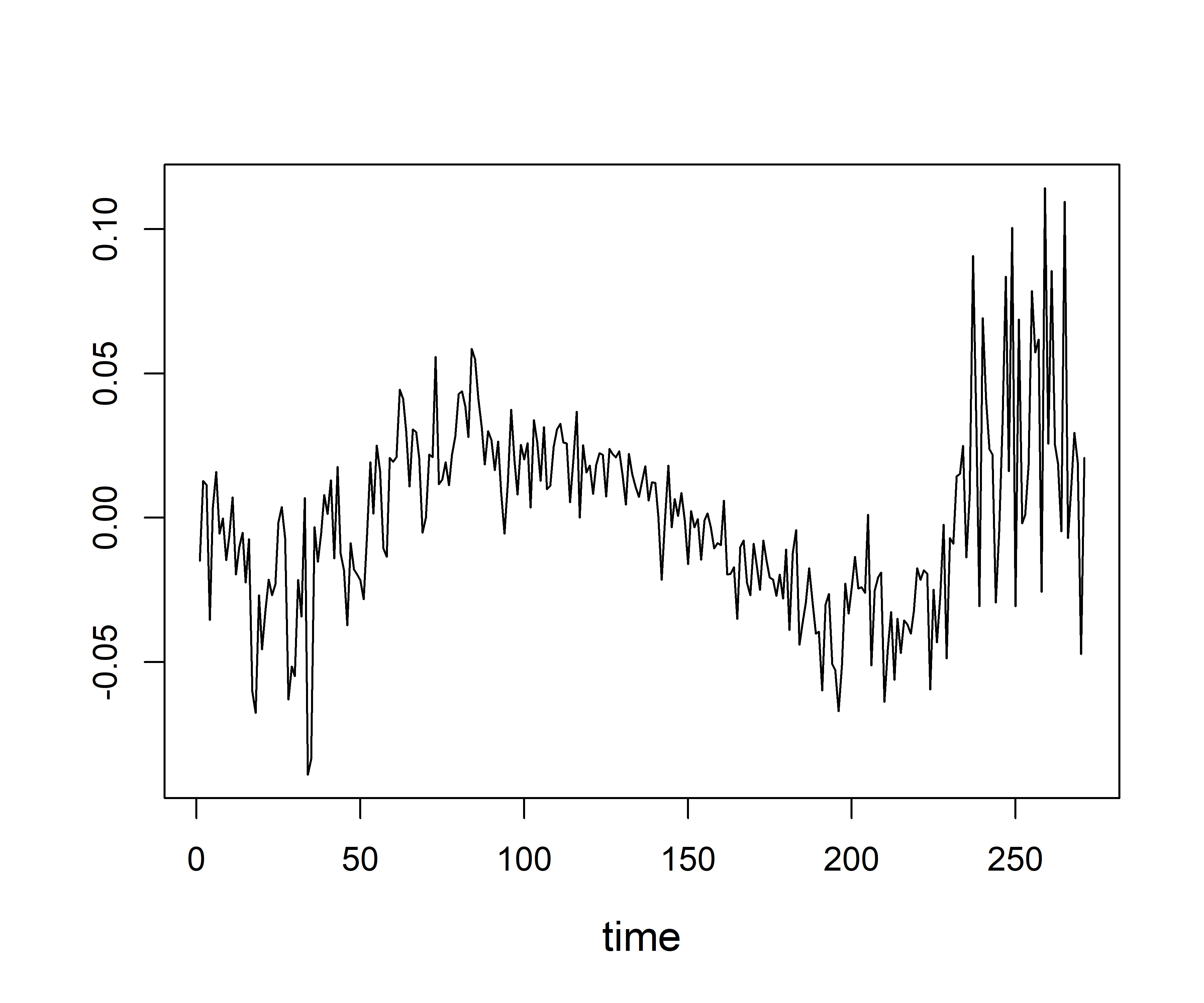}
	\subcaption{$\langle Z_t^0, \hat{v}_{6} \rangle$-female data}
\end{subfigure}
\begin{subfigure}{.32\linewidth}
	\includegraphics[width = \linewidth, trim={0 0 0 4em},clip]{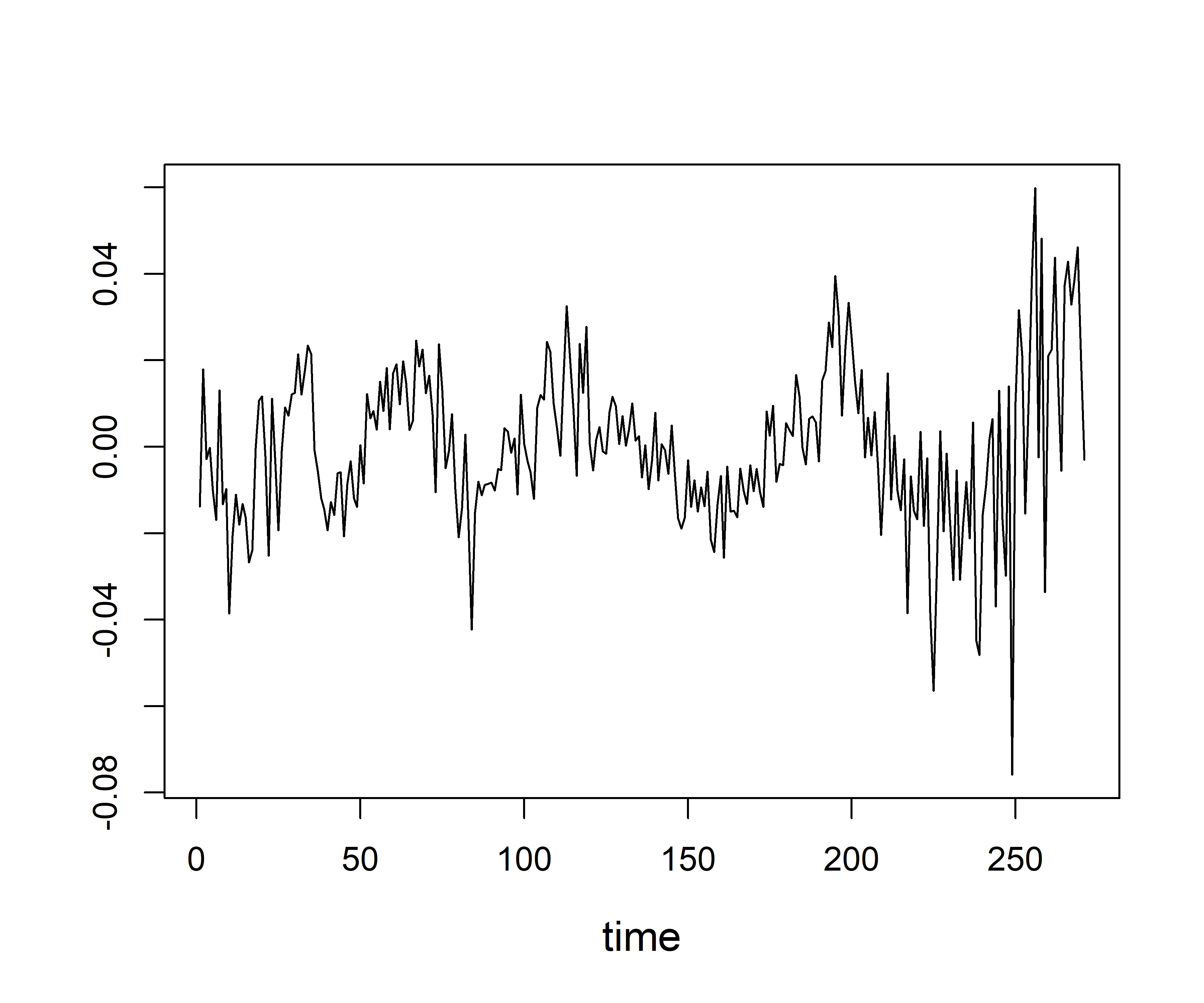}
	\subcaption{$\langle Z_t^0, \hat{v}_{12} \rangle$-female data}
\end{subfigure}
\caption{Sets of estimated principal component scores for the Swedish female and male age-specific mortality rates} \label{fig2}
\end{figure}

\commRV{
\subsection{Empirical application -- Canadian yield curves} \label{sec_yield}
In this section, we apply our methodology to the end-of-month Canadian zero-coupon bond yield curve data for the period spanning January 1991 to April 2023  \phantomsection\label{rvv2sample2}{($T=388$)}; the data used in this section is publicly available at \url{https://www.bankofcanada.ca/rates/interest-rates/bond-yield-curves/}. 
Each observation consists of zero-coupon bond yields at 120 reguarly spaced maturities varying from 0.25 to 30 (years). As in the previous section, we view yields at various maturities as functional observations as in, for example, \cite{hays2012functional} and \cite{MartinezHernandez2020}, and then apply our inferential methods to the considered data. The yield curves ($Y_t$) and their mean-corrected versions ($Y_t-T^{-1}\sum_{t=1}^T Y_t$) are visualized in Figure~\ref{fig1add}. 
\begin{figure}
\begin{subfigure}{.495\linewidth}
	\includegraphics[width = \linewidth, trim={0 0 0 5em},clip]{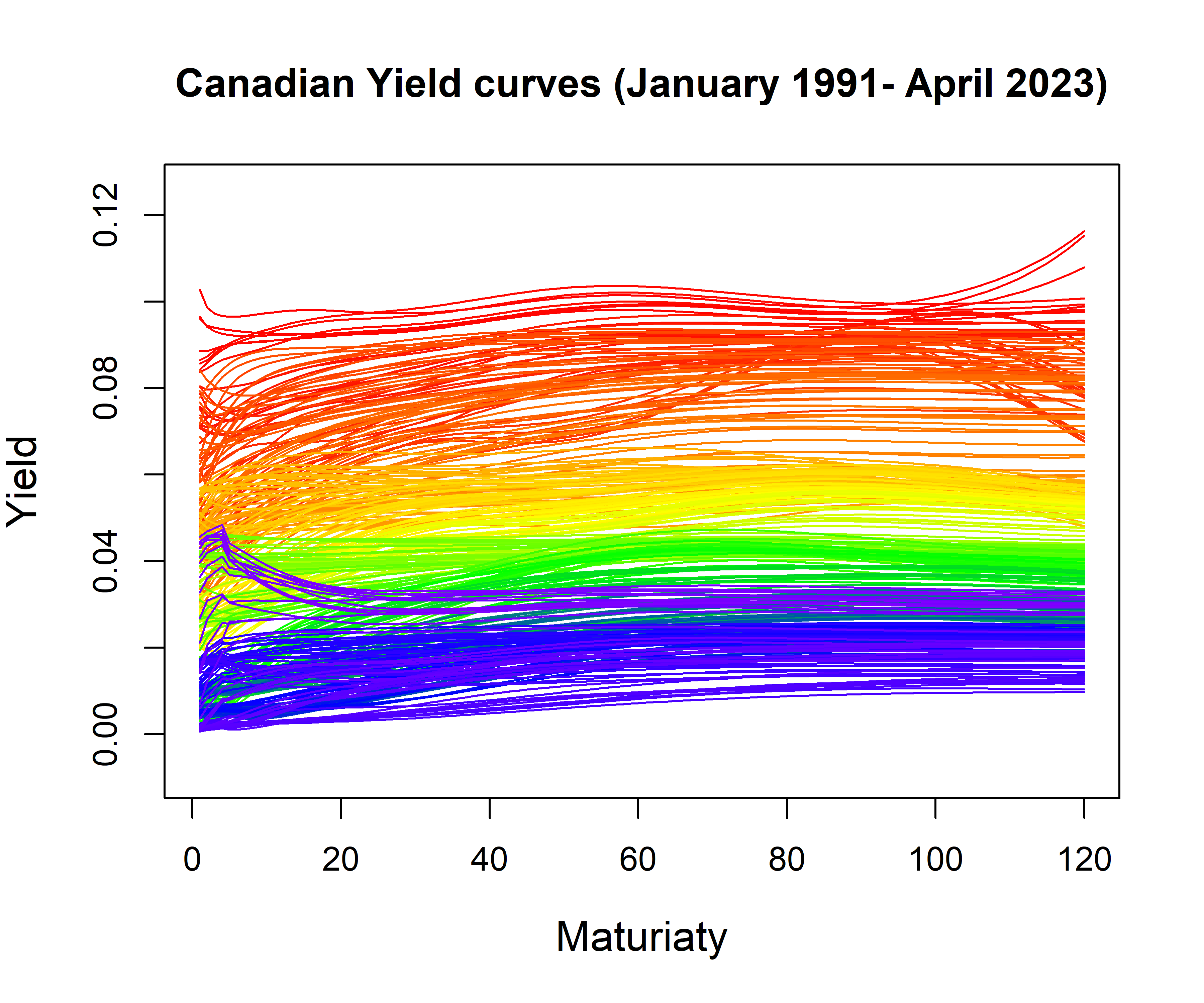}
	\subcaption{Yield curves}
\end{subfigure}
\begin{subfigure}{.495\linewidth}
	\includegraphics[width = \linewidth, trim={0 0 0 5em},clip]{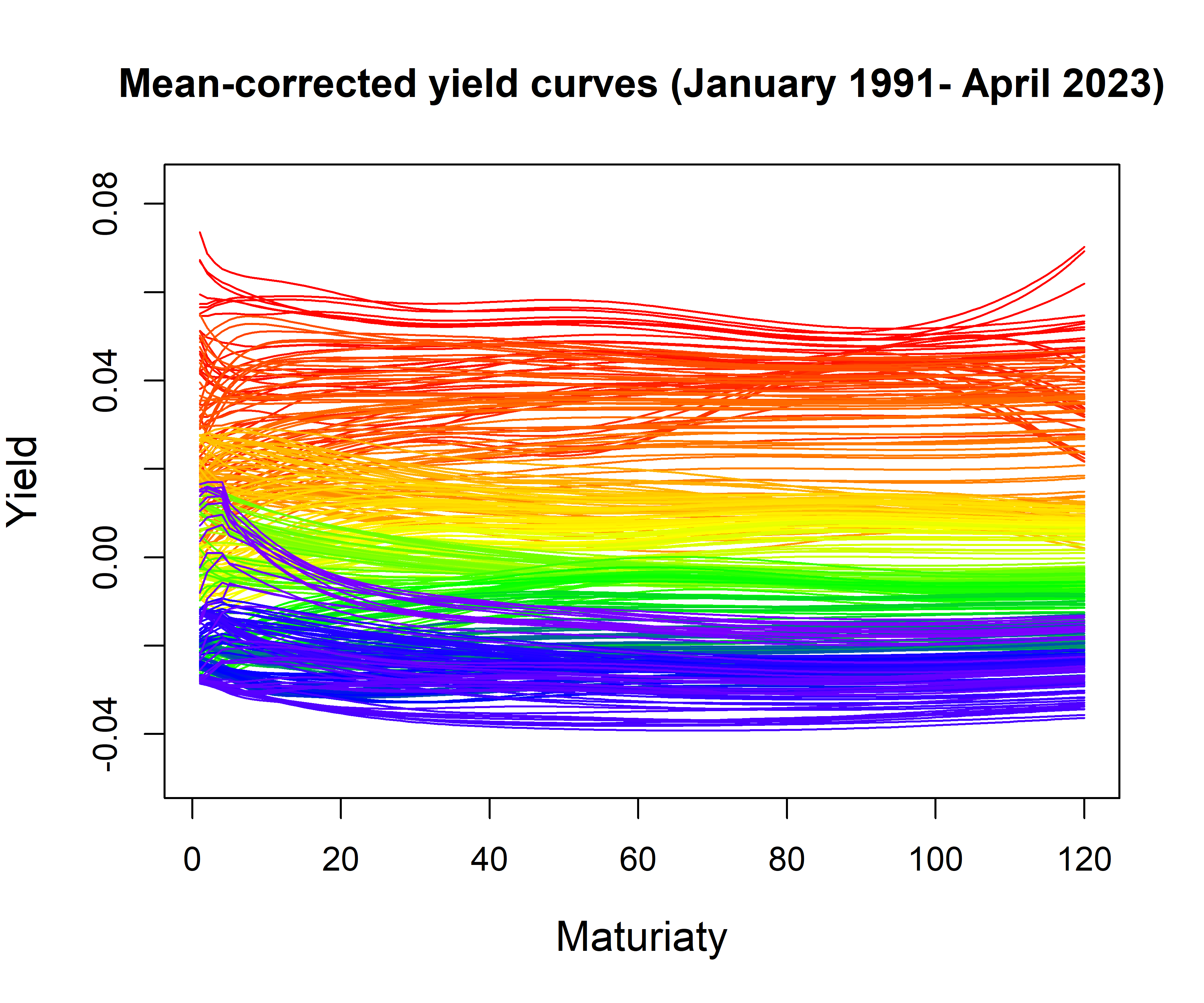}
	\subcaption{Mean-corrected yield curves}
\end{subfigure}
\caption{Canadian yield curves} \label{fig1add}
\begin{justify}
	{\footnotesize{Notes: The data for a specific month is given by a 120-dimensional vector of yields at various maturities from age 0.25 to 30 (years), and each of such vectors is plotted as a function of maturity. }}
\end{justify}
\end{figure}

In our statistical analysis, we represent the observed yield curves as functions, similar to the approach used in the previous mortality example; we conducted a sensitivity analysis as in Section~\ref{sec:mortality} by varying the number of basis functions and found that the reported results remained largely unchanged, with only minor variations. We first  estimate the memory parameter $\dn$ of the time series. The top rows of Table~\ref{tabemp1add} report the local Whittle estimates. As observed in the literature (e.g., \cite{BuschNautz2010} and \cite{Li2020}), the time series exhibits a high degree of persistence, with the memory close to unity, leading to the conclusion that the time series is nonstationary. We then apply the proposed testing procedure to estimate the dimension of the dominant subspace for the yield curves. The testing results, along with the eigenvalue-ratio estimates ($\hat{q}_{\dnn}$) for comparison, are reported in Table~\ref{tabemp2add}. 
The estimated dimension of the dominant subspace from the proposed testing procedure is~6 while the eigenvalue-ratio estimate is~1. As in Section~\ref{sec:mortality}, we are also inclined to conclude that $q_{\dnn} = 6$ due to the overall superior performance of our testing procedure observed in the previous simulation study. Then, the dominant subspace may be estimated by the span of the eigenvectors corresponding to the first six largest eigenvalues of the sample covariance operator.  

Assuming that the additional conditions given in Assumption~\ref{assum2add} are satisfied, $\ds$ and $q_{\dss}$ can also be estimated by using the proposed methods. The bottom rows of Tables~\ref{tabemp1add} and~\ref{tabemp2add} report the estimates. As in Section~\ref{sec:mortality}, these results might not be that meaningful if Assumption~\ref{assum2add} is violated. In parallel to Section~\ref{sec:mortality}, in Figure~\ref{fig2add}, we report the time series of $\langle Z_t^0, \hat{v}_{j} \rangle$ for a few different values of $j$, which are selected so that they exhibit different behaviors. While we can find strong evidence of nonstationarity from panel (a), it is unclear if Assumption~\ref{assum2add} holds and/or if the estimate of $q_{\dss}$ is reliable. This uncertainty arises from the observation that the time series in panel (c) does not appear to exhibit short-range dependence. This may be due to inaccuracy of our proposed estimator of $q_{\dss}$ in finite samples as discussed in Section~\ref{sec_simulation1}. 

\begin{table}[!htb]
\caption{Local Whittle estimation -- Canadian yield curve.} \label{tabemp1add}
\begin{tabular*}{1\linewidth}{@{\extracolsep{\fill}}lcc@{}}
	\toprule
	Target	& Method  &  Estimate \\ \midrule	
	$\dn$ 	&  Proposed  & 0.997   \\ 
	&  LRS-type  & 0.984 \\ \midrule	
	$\ds$ 	&  Proposed  & 0.384  \\ 
	&  LRS-type  & 0.377  \\ \bottomrule	
\end{tabular*} 
{\footnotesize{Notes: The estimators considered in this table are equivalent to those in Table~\ref{tabemp1}.}}
\end{table}

\begin{table}[!htb]
\caption{Dimension estimation -- Canadian yield curve.} \label{tabemp2add}
\begin{tabular*}{1\linewidth}{@{\extracolsep{\fill}}lcc@{}}
	\toprule
	Target	& Method  &  Estimate   \\ \midrule	
	$q_{\dnn}$ 	&  Proposed  &6  \\ 
	&  LRS-type  & 1   \\ \midrule	
	$q_{\dss}$ 	&  Proposed  & 5   \\ \bottomrule	
\end{tabular*} 
{\footnotesize{Notes: The estimators considered in this table are equivalent to those in Table~\ref{tabemp2} with $q_{\max} = 7$}.} 
\end{table}

\begin{figure}[!htb]
\centering
\begin{subfigure}{.43\linewidth}
	\includegraphics[width = \linewidth, trim={0 0 0 4em},clip]{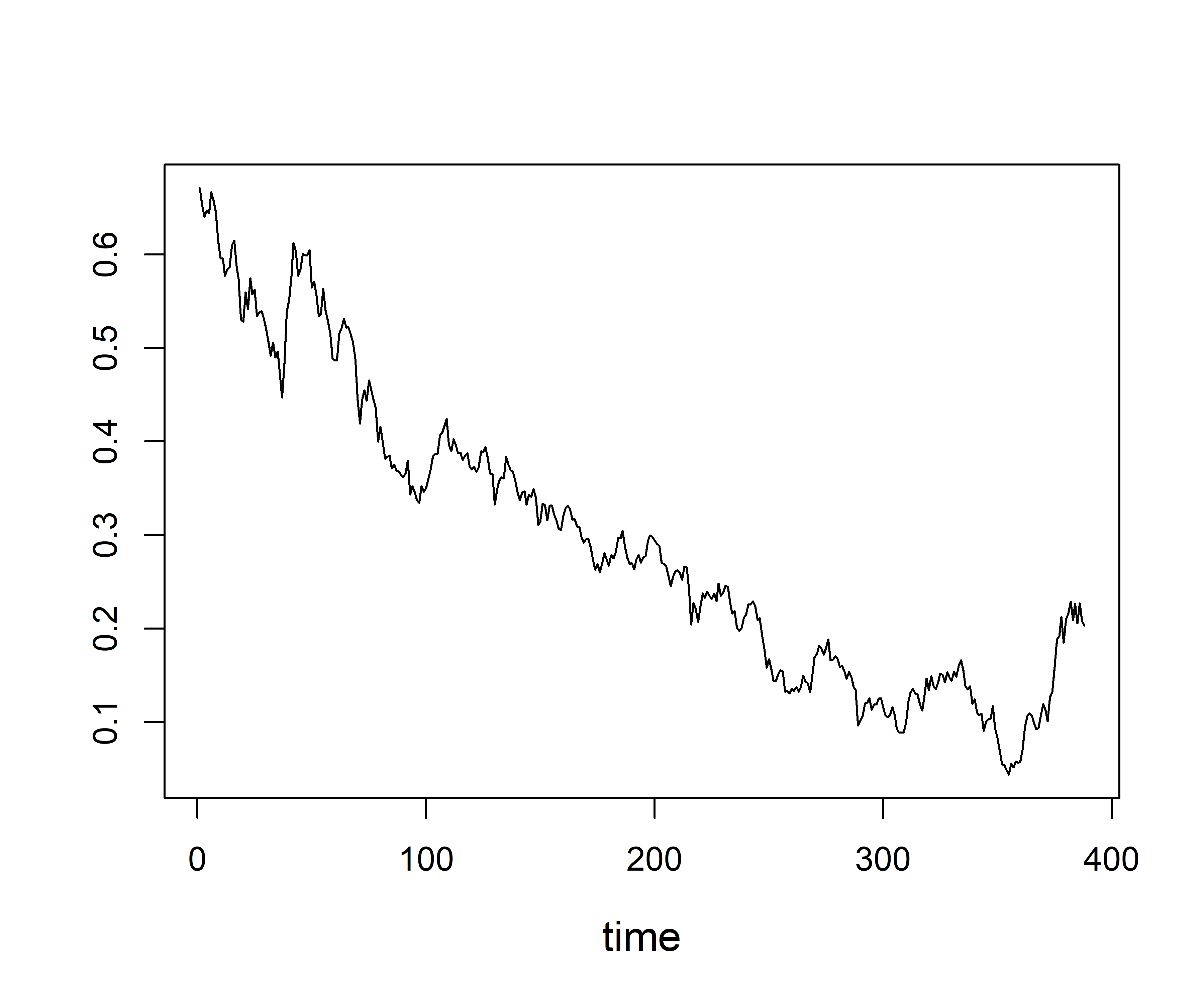}
	\subcaption{$\langle Z_t^0, \hat{v}_{1} \rangle$}
\end{subfigure}
\begin{subfigure}{.43\linewidth}
	\includegraphics[width = \linewidth, trim={0 0 0 4em},clip]{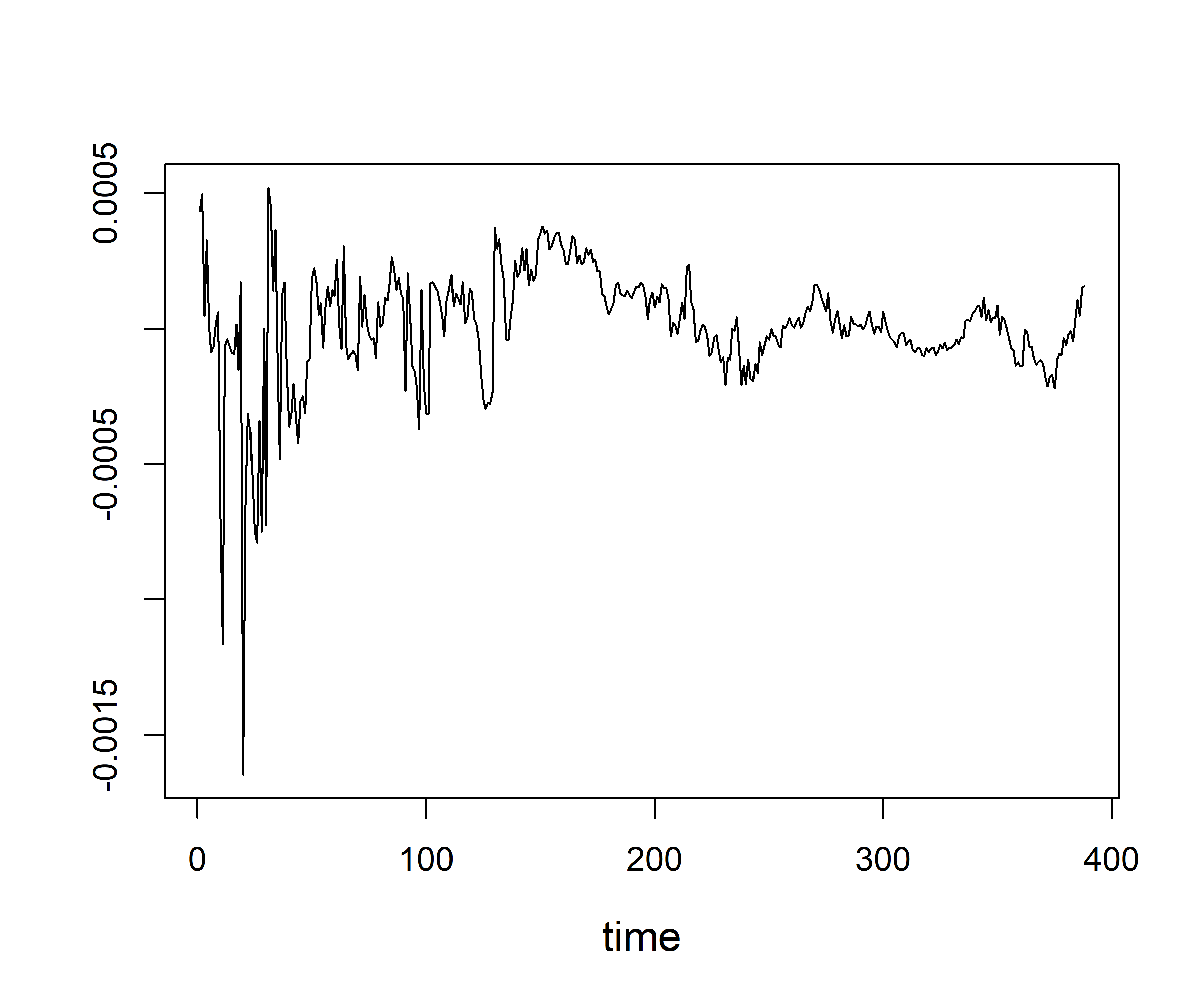}
	\subcaption{$\langle Z_t^0, \hat{v}_{7} \rangle$}
\end{subfigure}\\
\begin{subfigure}{.43\linewidth}
	\includegraphics[width = \linewidth, trim={0 0 0 4em},clip]{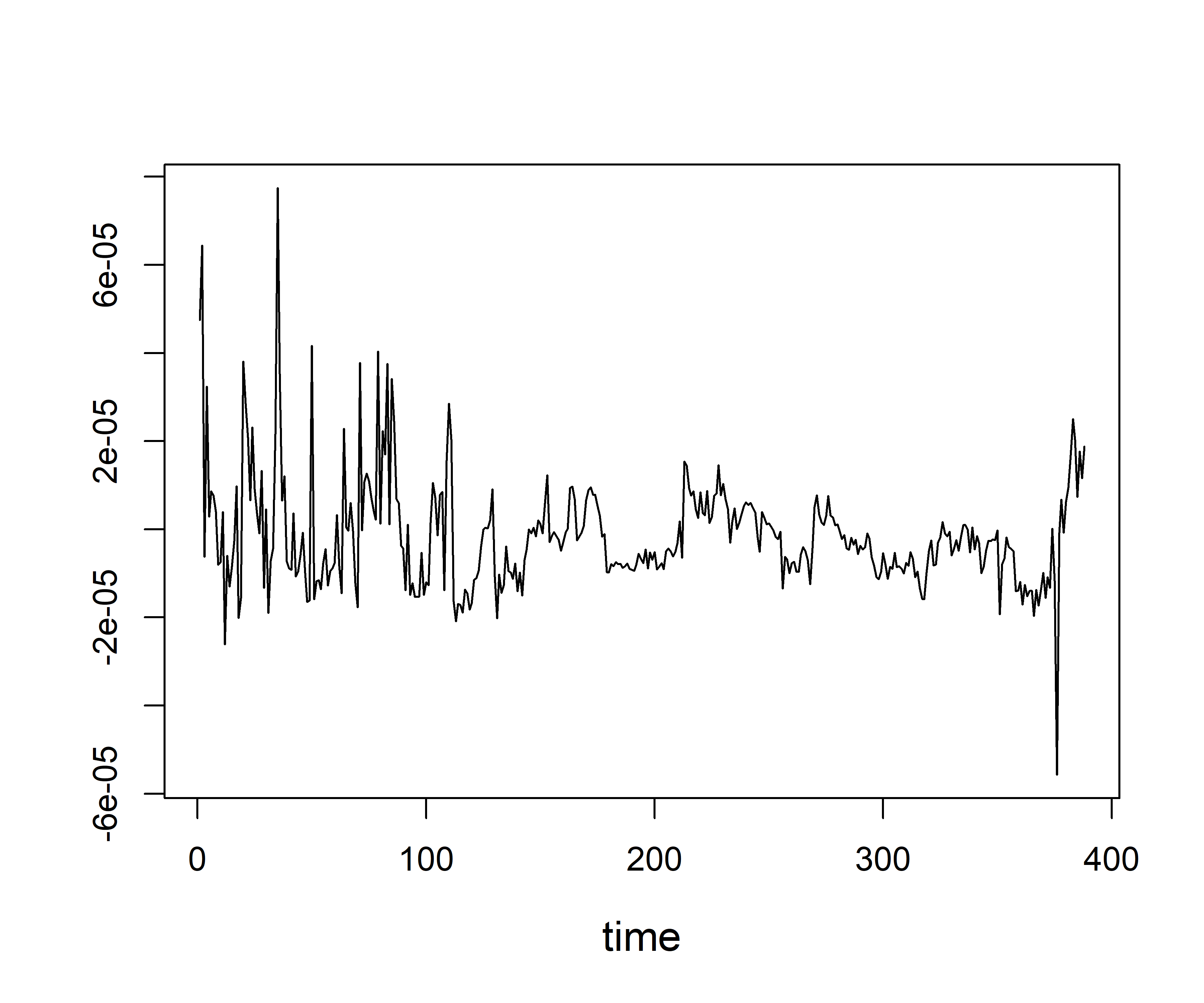}
	\subcaption{$\langle Z_t^0, \hat{v}_{12} \rangle$}
\end{subfigure}  
\begin{subfigure}{.43\linewidth}
	\includegraphics[width = \linewidth, trim={0 0 0 4em},clip]{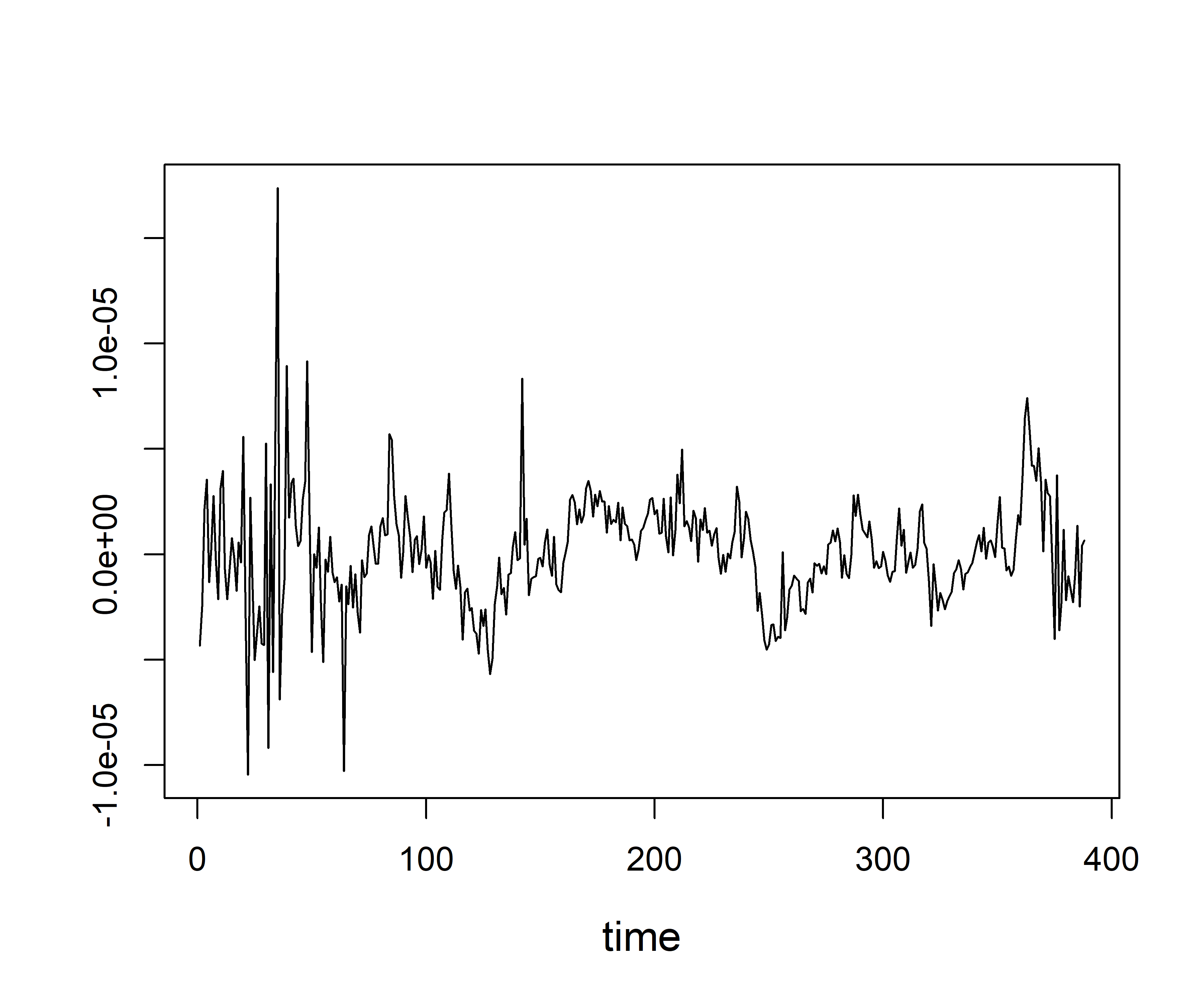}
	\subcaption{$\langle Z_t^0, \hat{v}_{15} \rangle$}
\end{subfigure}  
\caption{Sets of estimated principal component scores for the Canadian yield curves} \label{fig2add}
\end{figure}
}

\section{Conclusion}\label{sec:conclude}

This article has introduced a fractionally cointegrated curve time series with long-range dependence and derived some relevant asymptotic theorems. The functional dependence structure is specified via the projections of the curve process onto different subspaces spanned by additive orthonormal functions. The subspaces can be split into nonstationary and stationary components. The determination of the dimensions of the subspaces is carried out via our proposed tests, which outperform the modified eigenvalue ratio estimator in terms of correct identification. We have shown that the projection of the curve-valued process onto the subspaces retains most of the sample information from the original process. We also present a local Whittle estimator to estimate the memory parameter. The methodologies are illustrated via simulation and empirical applications to Swedish age-specific mortality rates and Canadian yield curves.

The article might be extended in two directions: 
\begin{inparaenum}
\item[(i)] nonstationary cointegration and 
\item[(ii)] cointegration in long-range dependent processes. 
\end{inparaenum}
In this paper, we only consider the case with $\dn>1/2$ and $\ds<1/2$. However, it is also possible to have $\dn>1/2$ and $1/2\leq \ds< \dn$, which corresponds to the case with nonstationary cointegration. As may be expected from the recent paper by \cite{Johansen2019}, this research direction will require a new theoretical approach. Considering the case where $\dn<1/2$ but $0<\ds<\dn$ may also be interesting. Given that the memory of a certain linear combination of the original time series is strictly smaller than the highest memory, this may be understood as a cointegration in long-range dependent processes. It is reasonable to assume that functional time series exhibiting long-range dependence may allow this kind of memory reduction while relevant theoretical results are currently absent.

\section*{Acknowledgment}

The authors wish to thank the editor and two reviewers, whose valuable comments led to an improved version of this article.

\appendix 
\section{Mathematical Appendix} \label{sec_appen}
It will be convenient to define some notation for the subsequent discussion. We first define 
\begin{align} \label{eqomega}
\Omega^2  =  P\left(\sum_{j=0}^\infty \psi_j\right) C_{\varepsilon}\left(\sum_{j=0}^\infty \psi_j\right)^\ast P.
\end{align}
and let $\{\beta_j,u_j\}_{j=1}^{q_{\dnn}}$ be the eigenvalues and the eigenvectors of  $\Omega^2$. The square-root operator of $\Omega^2$ is well-defined and it is simply denoted by $\Omega$. We then let $W_{\dn}$ and $W_{{\dn}+\alpha}$ denote Type II fractional Brownian motions of order $\dn$ and $\dn+\alpha$ taking values in $\mathcal H_N = \ran P$ driven by the common Brownian motion whose covariance operator is given by $\sum_{j=1}^{q_{\dnn}} u_j \otimes u_j$. Define
\begin{align*}
\bar{W}_{\dn}(r) &= 	 W_{\dn}(r) - \int_{0}^1 W_{\dn}(s)ds,\\
\widetilde{W}_{\dn+\alpha}(r) &=  W_{\dn+\alpha}(r) - \left(\int_{0}^1 W_{\dn+\alpha}(s)ds\right)\left(\int_{0}^r \frac{(r-s)^{\alpha-1}}{\Gamma(\alpha)}ds\right).
\end{align*} 

We first provide a useful lemma that will be used in the subsequent sections.
\begin{lemma}\label{lemwkc} 
Suppose that $\bar{Z}_t$ and $\tilde{Z}_t$ are defined as in \eqref{eqzbar} and \eqref{eqztilde} for $\alpha> 0$, and the time series $Z_t$ satisfies Assumption~\ref{assum1}.  Then 
\begin{align} 
	T^{1/2-\dn}\bar{Z}_{\truc{Tr}} 	&\Rightarrow \Omega \bar{W}_{\dn}(r), \label{eqlem1} \\
	T^{1/2-\dn-\alpha}\widetilde{Z}_{\truc{Tr}}  &\Rightarrow \Omega \widetilde{W}_{\dn+\alpha}(r),  \notag
\end{align}
where $\Rightarrow$ denotes the weak convergence in $\mathcal D[0,1]$ of $\mathcal H_N$-valued functions.
\end{lemma}
\begin{proof} 
We first show (i). \commWS{Note that}
\begin{align}
 	T^{1/2-\dn}\bar{Z}_{\truc{Tr}}  = &   T^{1/2-\dn} P{Y}_{\truc{Tr}} -  T^{-1/2-\dn} \sum_{t=1}^T  PY_t  +  T^{1/2-\dn} (I-P){Y}_{\truc{Tr}} \notag \\ &-  T^{-1/2-\dn} \sum_{t=1}^T  (I-P)Y_t,\label{eq01a}
\end{align}
where $T^{1/2-\dn} (I-P)Y_{\truc{Tr}} \pto 0$ uniformly in $r \in [0,1]$ and $ T^{-1/2-\dn} \sum_{t=1}^T  (I-P)Y_t \pto 0$ since $(I-P)Y_t$ is stationary and $\dn>1/2$. We thus only consider the first term of \eqref{eq01a}. We apply Proposition 2.1 of \cite{Li2020} and the continuous mapping theorem to find that 
\begin{equation*}
	T^{1/2-\dn} \bar{Z}_{\truc{Tr}} \Rightarrow \Omega \left( W_{\dn}(r) - \int_{0}^1 W_{\dn}(s)ds  \right).
\end{equation*}
We next show (ii). Note that 
\begin{equation}\label{eqpff01}
	T^{1/2-\dn-\alpha} \widetilde{Z}_{\truc{Tr}} = T^{1/2-\dn-\alpha} \Delta^{-\alpha}_+ Y_{\truc{Tr}} - T^{1/2-\dn-\alpha} \Delta^{-\alpha}_+  T^{-1} \sum_{t=1}^TY_t. 
\end{equation}
Given that $\Delta^{-\alpha}_+\Delta^{-\dn}_+Y_t =\Delta^{-\dn-\alpha}_+Y_t$ and $\dn+\alpha>1/2$, we find that the first term of~\eqref{eqpff01} satisfies that 
\begin{equation}\label{eq0001}
	T^{1/2-\dn-\alpha} \Delta^{-\alpha}_+ Y_{\truc{Tr}}  \Rightarrow \Omega W_{\dn+\alpha}(r).
\end{equation}
On the other hand, let $\bar{Y}_T = T^{-1} \sum_{t=1}^T Y_t$. Then the second term of~\eqref{eqpff01} is equal to 
\begin{equation}\label{eq0002}
	T^{1/2-\dn-\alpha} \Delta^{-\alpha}_+  \bar{Y}_T = T^{1/2-\dn-\alpha} \sum_{k=1}^{\truc{Tr}} \pi_{{\truc{Tr}-k}}(\alpha) \bar{Y}_T,
\end{equation}
where $\pi_k (\alpha) = \frac{\Gamma(k+\alpha)}{\Gamma(\alpha)\Gamma(k+1)}$. Note that  
\begin{equation*}
	T^{1/2-\dn} \bar{Y}_T\dto \Omega \int_{0}^1  W_{\dn+\alpha}(s)ds
\end{equation*}
and
\begin{equation} \label{eq0003}
	T^{-\alpha}\sum_{k=1}^{\truc{Tr}} \pi_{{\truc{Tr}-k}}(\alpha)	\to \int_{0}^r \frac{(r-s)^{\alpha-1}}{\Gamma(\alpha)}ds, 
\end{equation}
where the convergence result given in~\eqref{eq0003} may be deduced from equation (35) of \cite{Nielsen2010}. An application of Slutsky's theorem (see e.g., p.\ 35 of \cite{VanDerVaart1996}) and the continuous mapping theorem with~\eqref{eq0001}-\eqref{eq0003} give the following convergence result: 
\begin{equation*}
	T^{1/2-\dn-\alpha} \widetilde{Z}_{\truc{Tr}} \Rightarrow \Omega \left( W_{\dn+\alpha}(r) - \int_{0}^1 W_{\dn+\alpha}(s)ds \left(\int_{0}^r \frac{(r-s)^{\alpha-1}}{\Gamma(\alpha)}ds\right)\right)
\end{equation*}
as desired.
\end{proof}	
\begin{remark}
Suppose that $\mu = 0$ and thus $Z_t= Y_t$. Under some appropriate conditions similar to \commWS{ours}, \cite{Li2020} shows that $T^{1/2-\dn}PZ_{\truc{Tr}} \Rightarrow \Omega W_{\dn}(r)$. In the case where $\mu=0$, \eqref{eqlem1} is slightly different from their result because $P$ does not appear on the left-hand side.      
\end{remark}

\subsection*{Proofs of the main results}

\begin{proof}[Proof of Proposition~\ref{prop3}]

We first deduce from Lemma~\ref{lemwkc} and the continuous mapping theorem that
\begin{equation} \label{eqref1}
	T^{-2\dn}\sum_{t=1}^T P\bar{Z_t} \otimes P\bar{Z}_t \dto \int_{0}^{1} \Omega \bar{W}_{\dn}(s)\otimes \Omega\bar{W}_{\dn} (s) ds.
\end{equation}
Note that $\bar{Z}_t=Y_t-\bar{Y}_T$, where \commWS{$\bar{Y}_T = T^{-1}\sum_{t=1}^T Y_t$} and $(I-P)Y_t$ is stationary. From the law of large numbers of stationary ergodic sequences, we find that
\begin{equation} \label{eqref1a}
	T^{-1}\sum_{t=1}^T (I-P)\bar{Z}_t \otimes (I-P)\bar{Z}_t \pto \mathbb{E}[(I-P){Y}_t \otimes (I-P){Y}_t].
\end{equation}
Define $\hat{P}_K$  as in \eqref{defpk}, i.e., $\hat{P}_K = \sum_{j=1}^{K} \hat{v}_j \otimes \hat{v}_j$, where ($\hat{v}_1,\ldots,\hat{v}_K$) are the eigenvectors corresponding to the first $K$ largest eigenvalues of $\hat{C}_{\bar{Z}}$. We may deduce from \eqref{eqref1}, the Skorohod representation theorem, and Lemma 4.3 of \cite{Bosq2000} that the first $q_{\dnn}$ eigenvectors $(\hat{v}_1,\ldots,\hat{v}_{q_{\dnn}})$ converge to a random orthonormal basis of $\ran P$; this proves \eqref{eqde3}. Note also that $(\hat{v}_{q_{\dnn}+1},\ldots,\hat{v}_{K})$ are the eigenvalues of $(I-\hat{P}_{q_{\dnn}})  	\hat{C}_{\bar{Z}}(I-\hat{P}_{q_{\dnn}})$. Since $I-\hat{P}_{q_{\dnn}} \pto (I-P)$ and $(I-P) \hat{C}_{\bar{Z}}(I-P) \pto \mathbb{E}[(I-P)Y_t \otimes (I-P) Y_t]$ (see \eqref{eqref1a}), we have 
\begin{equation*} 
	(I-\hat{P}_{q_{\dnn}})  	\hat{C}_{\bar{Z}}(I-\hat{P}_{q_{\dnn}}) \pto \mathbb{E}[(I-P)Y_t \otimes (I-P) Y_t].
\end{equation*}
Since the $(K-q_{\dnn})$-th largest eigenvalue of  $\mathbb{E}[(I-P)Y_t \otimes (I-P) Y_t]$ is distinct from the next one, the projection $\sum_{j=1}^{K-q_{\dnn}}   v_j^S \otimes  v_j^S$ (where $ v_j^S$ is the eigenvector corresponding to the $j$-th largest eigenvalue) is a well-defined fixed bounded linear operator regardless of if any $j$-th eigenvalue for $j<K-q_{\dnn}$ is repeated (and thus $ v_j^S$ is not uniquely determined) or not. Moreover, in this case, we may deduce from Lemma 4.4 of \cite{Bosq2000} that  
\begin{equation} \label{eqref3}
	\hat{P}_K - \hat{P}_{q_{\dnn}} = \sum_{j=q_{\dnn}+1}^{K} \hat{v}_j \otimes \hat{v}_j  \pto \sum_{j=1}^{K-q_{\dnn}}  v_j^S \otimes v_j^S.  
\end{equation}
We now consider the limiting behavior of $T^{-1} \sum_{t=1}^{T} \hat{P}_K\bar{Z}_t \otimes \hat{P}_K\bar{Z}_t$ of which (almost surely) nonzero eigenvalues are given by $(\hat{\mu}_1,\ldots, \hat{\mu}_K)$. As will be shown later in our proof of Proposition~\ref{prop4}, \eqref{eqref1}-\eqref{eqref3} imply that the first $q_{\dnn}$ eigenvalues, multiplied by $T^{1-2\dn}$, converge to positive (and almost surely bounded) random eigenvalues while the remaining eigenvalues converge to fixed and positive eigenvalues as long as $ \mathbb{E}[(I-P){Y}_t \otimes (I-P){Y}_t]$ allows $K-q_{\dnn}$ nonzero eigenvalues; in particular, see \eqref{pfeq7}. This proves the desired results. 
\end{proof}
\begin{proof}[Proof of Proposition~\ref{prop4}]
We will first show the limiting behaviors of two random operators given by $A_T = \sum_{t=1}^T \bar{Z}_t\otimes \bar{Z}_t$ and $B_T = \sum_{t=1}^T \tilde{Z}_t\otimes \tilde{Z}_t$ when they are understood as the maps acting on $\ran \hat{P}_K$ (the span of the first $K$ eigenvectors $(\hat{v}_1,\ldots,\hat{v}_j)$ of $A_T$). In our proof of Proposition~\ref{prop3}, we showed that $\hat{P}_{q_{\dnn}} = \sum_{j=1}^{q_{\dnn}}\hat{v}_j \otimes \hat{v}_j  \pto P$. Combining this with \eqref{eqref3}, we find that 
\begin{equation*}
	\hat{P}_K =  \hat{P}_{q_{\dnn}} + (\hat{P}_K -  \hat{P}_{q_{\dnn}}) \pto P + \sum_{j=1}^{K-q_{\dnn}}   v_j^S \otimes  v_j^S =: P_K,
\end{equation*}
where $P_K$ is a well-defined and fixed projection. 

Let $\hat{P}_K^N = \hat{P}_K P$, $\hat{P}_K^S = \hat{P}_K (I-P)$, ${P}_K^N = {P}_K P$, ${P}_K^S = {P}_K (I-P)$ and $D_T = \left( \begin{smallmatrix} T^{-\dn} I_1 & 0 \\ 0 & T^{-1/2} I_2 \end{smallmatrix}\right)$, where $I_1$ and $I_2$ are the relevant identity maps of rank $q_{\dnn}$ and $K-q_{\dnn}$, respectively.  Given that $B_T = \hat{P}_KB_T\hat{P}_K$ and $\hat{P}_K = \hat{P}_K^N + \hat{P}_K^S$ holds, we may understand $T^{-2\alpha} D_TB_TD_T$ as the following operator matrix:
\begin{align*}
	&T^{-2\alpha}D_TB_TD_T = \\ & \left(\begin{matrix}
		T^{-2\dn-2\alpha}\sum_{t=1}^T \hat{P}_K^N\tilde{Z}_t\otimes \hat{P}_K^N\tilde{Z}_t  &  T^{-\dn-1/2-2\alpha}\sum_{t=1}^T \hat{P}_K^S\tilde{Z}_t\otimes \hat{P}_K^N\tilde{Z}_t  \\
		T^{-\dn-1/2-2\alpha} \sum_{t=1}^T \hat{P}_K^N\tilde{Z}_t\otimes \hat{P}_K^S\tilde{Z}_t  &  T^{-1-2\alpha}  \sum_{t=1}^T \hat{P}_K^S\tilde{Z}_t\otimes \hat{P}_K^S\tilde{Z}_t  
	\end{matrix}\right).
\end{align*}
Note that $\hat{P}_K \pto P_K$ and $$T^{-2\dn-2\alpha}\sum_{t=1}^T \hat{P}_K^N\tilde{Z}_t\otimes \hat{P}_K^N\tilde{Z}_t= \hat{P}_K \left(T^{-2\dn-2\alpha}\sum_{t=1}^T P\tilde{Z}_t\otimes P\tilde{Z}_t  \right) \hat{P}_K.$$  We thus deduce from Lemma~\ref{lemwkc} that 
\begin{align} \label{eqpf0}
	T^{-2\dn-2\alpha}\sum_{t=1}^T P_K^N\tilde{Z}_t\otimes  P_K^N\tilde{Z}_t  \dto  \int_{0}^{1} \Omega \widetilde{W}_{\dn+\alpha}  (s)\otimes \Omega \widetilde{W}_{\dn+\alpha} (s) ds.  
\end{align}
Combining these results, we find that 
\begin{align} \label{pfeq1}
	T^{-2\dn-2\alpha}\sum_{t=1}^T \hat{P}_K^N\tilde{Z}_t\otimes \hat{P}_K^N\tilde{Z}_t  \dto  \int_{0}^{1} \Omega \widetilde{W}_{\dn+\alpha} (s) \otimes \Omega \widetilde{W}_{\dn+\alpha} (s) ds. 
\end{align} 
Using the isomorphism between $\mathbb{R}^K$ and any $K$-dimensional Hilbert space and the arguments used in the proof of Lemma 6(f) of \cite{Nielsen2010}, we may deduce the following: 
\begin{equation*}
	T^{-\dn-\psi-\alpha}(\log T)^{-\mathds{1}\{\psi=1/2\}}\sum_{t=1}^T {P}_K^S\tilde{Z}_t\otimes {P}_K^N\tilde{Z}_t = O_p(1), 
\end{equation*}
where $\psi = \max\{\ds+\alpha, 1/2\}$. Note that $\psi < 1/2 + \alpha$ and $\hat{P}_K \pto P_K$, from which we find that
\begin{equation} \label{pfeq2}
	T^{-\dn-1/2-2\alpha}\sum_{t=1}^T \hat{P}_K^S\tilde{Z}_t\otimes \hat{P}_K^N\tilde{Z}_t \pto 0.
\end{equation}
With nearly identical arguments, we also find that 
\begin{equation*} 
	T^{-\dn-1/2-2\alpha}\sum_{t=1}^T {P}_K^N\tilde{Z}_t\otimes {P}_K^S\tilde{Z}_t \pto 0.
\end{equation*} 
Lastly, we note that $\sum_{t=1}^T \hat{P}_K^S\tilde{Z}_t\otimes \hat{P}_K^S\tilde{Z}_t = O_p(T^{-1})$ if $\ds+\alpha < 1/2$. On the other hand, if  $\ds+\alpha\geq 1/2$, it can be shown from Lemma 6(e) of \cite{Nielsen2010} that $\sum_{t=1}^T \hat{P}_K^S\tilde{Z}_t\otimes \hat{P}_K^S\tilde{Z}_t = O_p(T^{-2\psi}(\log T)^{-\mathds{1}\{\psi=1/2\}})$. Note that $2\psi < 1 + 2\alpha $ since $\ds<1/2$, we thus find that  
\begin{equation}\label{pfeq3}
	T^{-1-2\alpha}  \sum_{t=1}^T \hat{P}_K^S\tilde{Z}_t\otimes \hat{P}_K^S\tilde{Z}_t  \pto 0.
\end{equation}
Combining \eqref{pfeq1}-\eqref{pfeq3}, we find that 
\begin{align}\label{pfeq3a}
	T^{-2\alpha}D_TB_TD_T \dto  \left(\begin{matrix}
		\int_{0}^{1} \Omega \widetilde{W}_{\dn+\alpha}(s) \otimes \Omega \widetilde{W}_{\dn+\alpha} (s) ds &  0\\
		0 &  0
	\end{matrix}\right).
\end{align}
We next consider $D_TA_TD_T$. Given that $A_T = \hat{P}_KA_T\hat{P}_K$ holds, we may understand $T^{-2\alpha} D_TA_TD_T$ as the following operator matrix:
\begin{align*}
	D_TA_TD_T = \left(\begin{matrix}
		T^{-2\dn}\sum_{t=1}^T \hat{P}_K^N\bar{Z}_t\otimes \hat{P}_K^N\bar{Z}_t  &  T^{-\dn-1/2}\sum_{t=1}^T \hat{P}_K^S\bar{Z}_t\otimes \hat{P}_K^N\bar{Z}_t  \\
		T^{-\dn-1/2} \sum_{t=1}^T \hat{P}_K^S\bar{Z}_t\otimes \hat{P}_K^N\bar{Z}_t  &  T^{-1}  \sum_{t=1}^T \hat{P}_K^S\bar{Z}_t\otimes \hat{P}_K^S\bar{Z}_t  
	\end{matrix}\right).
\end{align*}
Similarly as in~\eqref{eqpf0}, we find that
\begin{equation*}
	T^{-2\dn}\sum_{t=1}^T P_K^N\bar{Z}_t\otimes  P_K^N\bar{Z}_t  \dto  \int_{0}^{1} \Omega W_{\dn}(s) \otimes \Omega W_{\dn} (s) ds.  
\end{equation*}
Since $T^{-2\dn}\sum_{t=1}^T \hat{P}_K^N\bar{Z}_t\otimes \hat{P}_K^N\bar{Z}_t = \hat{P}_K \left(T^{-2\dn}\sum_{t=1}^T P\bar{Z}_t\otimes P\bar{Z}_t  \right) \hat{P}_K$ and $\hat{P}_K \pto P_K$, we conclude that
\begin{equation} \label{pfeq4}
	T^{-2\dn}\sum_{t=1}^T \hat{P}_K^N\bar{Z}_t\otimes  \hat{P}_K^N\bar{Z}_t  \dto  \int_{0}^{1} \Omega W_{\dn}(s) \otimes \Omega W_{\dn} (s) ds. 
\end{equation}
Moreover, from similar arguments used in the proof of Lemma 6(c) in \cite{Nielsen2010}, we may deduce that $\sum_{t=1}^T {P}_K^S\bar{Z}_t\otimes {P}_K^N\bar{Z}_t = O_p(T^{\ds-1/2})$ and thus
\begin{align}\label{pfeq5}
	&T^{-\dn-1/2}\sum_{t=1}^T {P}_K^N\bar{Z}_t\otimes {P}_K^S\bar{Z}_t  \pto 0,\\
	&T^{-\dn-1/2}\sum_{t=1}^T {P}_K^S\bar{Z}_t\otimes {P}_K^N\bar{Z}_t  \pto 0.
\end{align}
Lastly, we deduce the following from the  law of large numbers of stationary ergodic \commWS{sequences}:
\begin{align}\label{pfeq6}
	T^{-1}  \sum_{t=1}^T \hat{P}_K^S\bar{Z}_t\otimes \hat{P}_K^S\bar{Z}_t   \pto \mathbb{E}[{P}_K^SY_t \otimes {P}_K^SY_t]={P}_K \mathbb{E}[(I-P)Y_t \otimes (I-P)Y_t]{P}_K.
\end{align}
Combining \eqref{pfeq4}-\eqref{pfeq6}, we find that
\begin{align}\label{pfeq7}
	D_TA_TD_T \dto  \left(\begin{matrix}
		\int_{0}^{1} \Omega \bar{W}_{\dn}(s) \otimes \Omega \bar{W}_{\dn} (s) ds &  0 \\0 &  {P}_K \mathbb{E}[(I-P)Y_t \otimes (I-P)Y_t]{P}_K
	\end{matrix}\right).
\end{align}
Consider the eigenvalue problem given by  
\begin{equation*}
	\hat{\tau}_j  D_TA_TD_T v_k  =	T^{-2\alpha}  D_TB_TD_T \hat{v}_j, \quad \hat{\tau}_1\geq\hat{\tau}_2\geq\ldots\geq\hat{\tau}_K,
\end{equation*}
\commWS{where $T^{2\alpha} \hat{\nu}_j = \hat{\tau}_j^{-1}$ holds}. Then we know from \eqref{pfeq3a} and \eqref{pfeq7} that $\hat{\tau}_j \pto 0$ if $j >q_{\dnn}$ while $(\hat{\tau}_1,\ldots,\hat{\tau}_{q_{\dnn}})$ converge to the eigenvalues of $\mathcal A_{\dn}^{-1} \mathcal A_{\dn+\alpha}$, where  
\begin{equation}\label{eqA}
	\mathcal A_{\dn+\alpha} = \int_{0}^{1} \Omega \tilde{W}_{\dn+\alpha}(s) \otimes \Omega \tilde{W}_{\dn+\alpha} (s) ds, \quad  \mathcal A_{\dn} = \int_{0}^{1} \Omega \bar{W}_{\dn}(s) \otimes \Omega \bar{W}_{\dn} (s) ds.
\end{equation}
From these results, we find that $T^{2\alpha}(\hat{\nu}_{1},\ldots,\hat{\nu}_{q_{\dnn}})$ converge in distribution to the eigenvalues of $\mathcal A_{\dn+\alpha}^{-1}\mathcal A_{\dn}$; moreover, we know from the properties of the eigenvalues that these eigenvalues are distributionally equivalent to those of~\eqref{eqprop4} (see also Remark 5 of \cite{Nielsen2019}).
\end{proof}
\begin{proof}[Proof of Corollary~\ref{cor1prop4}]
The desired result immediately follows from Proposition~\ref{prop4}. The details are omitted. 
\end{proof}
\begin{proof}[Proof of Corollary~\ref{cor2prop4}]
In our proof of Proposition~\ref{prop4}, we showed that the first $q_{\dnn}$ eigenvalues, multiplied by $T^{2\alpha}$, converge to the eigenvalues of $\mathcal A_{\dn}^{-1} \mathcal A_{\dn+\alpha}$ (see \eqref{eqA}). From this result and the fact that the eigenvalues of $\mathcal A_{\dn}^{-1} \mathcal A_{\dn+\alpha}$ are almost surely distinct from each other, we may deduce that the corresponding eigenvectors converge to those of  $\mathcal A_{\dn}^{-1} \mathcal A_{\dn+\alpha}$ (Lemma 4.3 of \cite{Bosq2000}). This completes the proof.  
\end{proof}

\begin{proof}[Proof of Proposition~\ref{prop6}]
We first note that  $I-\overline{P} \pto I-{P}$ and, from the asymptotic result given in Proposition 2(i) of \cite{Li2021} (see also Proposition 2 of \cite{Li2020a}), we know that $h^{-2\ds}(I-P)\hat{\Lambda}(I-P) \pto \Lambda$ of rank $q_{\dss}$, whose eigenvectors span $\ran Q$. Combining all these results, we observe that 
\begin{equation*}
	h^{-2\ds}(I-\overline{P}) \hat{\Lambda}(I-\overline{P}) \pto \Lambda
\end{equation*}
and thus find that $(\hat{\mu}_1,\ldots, \hat{\mu}_{q_{\dss}})$ and the associated eigenvectors $(\hat{v}_1,\ldots, \hat{v}_{q_{\dss}})$ satisfy the following:
\begin{align}
	h^{-2\ds} \hat{\mu}_j  &\pto   \text{$j$-th largest eigenvalue of } \Lambda \text{ if $j \leq q_{\dss}$}, \label{eqde4}\\
	\hat{Q}_0 := \sum_{j=1}^{q_{\dss}} \hat{v}_j \otimes \hat{v}_j  &\pto  Q. \label{eqde5}
\end{align}
We next note that $(\hat{v}_{q_{\dss}+1},\ldots, \hat{v}_{K})$ are the eigenvectors of $(I-\hat{Q}_0)(I-\overline{P}) \hat{\Lambda}(I-\overline{P})(I-\hat{Q}_0)$ and 
\begin{align*}
	&(I-Q)(I-P)\hat{\Lambda}(I-P)(I-Q)  \\ &= \sum_{s=-T+1} ^{T-1}  \left (1-\frac{|s|}{h}\right ) (I-Q)(I-P)\hat{C}_s(I-P)(I-Q)  \pto \Lambda_0,
\end{align*}
where the convergence in probability follows from Theorem 4.2 of \cite{Hoermann2010} and that $(I-Q)(I-P)\hat{\Lambda}(I-P)(I-Q)$ is the sample long-run covariance operator of the SRD component. Combining this result with~\eqref{eqde5} and the fact that $I-\overline{P} \pto I-{P}$, we find that 	 
\begin{equation*}
	(I-\hat{Q}_0)(I-\overline{P}) \hat{\Lambda}(I-\overline{P})(I-\hat{Q}_0) \pto \Lambda_0.
\end{equation*}
This implies that $(\hat{\mu}_{q_{\dss}+1},\ldots,\hat{\mu}_{K})$ converge to the eigenvalues of $\Lambda_0$ (see Lemma 4.2 of \cite{Bosq2000}). Combining this with~\eqref{eqde4}, the desired results are obtained.
\end{proof}
\begin{proof}[Proof of Corollary~\ref{cor3}]
The desired result immediately follows from Proposition~\ref{prop6}. The details are omitted. 
\end{proof}	
\section{Local Whittle estimation} \label{sec_appen_d}
\subsection{A brief introduction to the local Whittle estimator}\label{sec_lw}
\commRV{
We first briefly introduce the local Whittle estimator. More detailed discussion can be found in e.g., \cite{Li2021} and \cite{Li2020}. Let $x_t$ be a univariate I($d$) process with the spectral density $f_x\sim G \delta^{-2d}$ in a vicinity of the origin; for our purpose, it will be sufficient to deal with the cases with $d \in (-1/2, 1/2)$ and $d \in (1/2, 3/2)$. We consider the following Gaussian objective function:
\begin{align}
\mathcal Q (G,d) = \frac{1}{m} \sum_{i=1}^m \left\{ \ln(G \delta_i^{-2d}) + \frac{I_x(\delta_i)}{G(\delta_i^{-2d})} \right\}, \label{eqobjective}
\end{align} 
where $I_x(\delta_i)$ denotes the sample periodogram defined by the square of discrete Fourier transform of the scores, $\delta_i = 2\pi i/ T$, $i=1,\ldots,m$ and $m$ is a positive integer satisfying $m=o(T)$; customarily, the choice $m=1+\lfloor T^{0.65} \rfloor$ can be used as in \cite{Li2020}. Let $(\hat{G},\hat{d})$ be the minimizer of \eqref{eqobjective} such that 
\begin{align}
(\hat{G},\hat{d}) = \underset{G \in (0,\infty), d \in [\Delta_1, \Delta_2]}{\text{arg\,min}}\mathcal Q (G,d),
\end{align} 
where $-1/2 < \Delta_1 < \Delta_2$ and $\Delta_2 < 1/2$ if we consider the case $d \in (-1/2, 1/2)$ and $\Delta_2 < \infty$ otherwise. We call $\hat{d}$ be the local Whittle estimator.   
}

\subsection{Statistical inference on $\dn$ and $\ds$.}
We employ the following assumptions associated with time series satisfying Assumption~\ref{assum1} and for an element $v \in \mathcal H$.
\begin{assumpLW}\label{assumlw}  
$\psi_j$  and  the spectral density  $f_v(\lambda)$  of the time series $\langle X_t, v \rangle$ for $v \in \mathcal H$ satisfy the following:
\begin{enumerate}[(i)]
	\item\label{assumsr0} $\psi_j = \phi_j A$ for $\phi_j \in \mathbb{R}$ and $A\in \mathcal L_{\mathcal H}$.
	\item\label{assumsr1} $\mathbb{P}(v \in \mathcal H_S) = 0$.	
	\item\label{assumsr2} $f_v(\lambda) $ is differentiable in a neighborhood of zero, and $\frac{d}{d\lambda}\log f_{v}(\lambda) = O(\lambda^{-1})$ as $\lambda \to 0+$.		
\end{enumerate}
\end{assumpLW}	
For convenience we let $\Delta_1^S$, $\Delta_2^S$, $\Delta_1^N$ and  $\Delta_2^N$ be real numbers satisfying that 
\begin{equation*}
-1/2 < \Delta_1^S < \Delta_2^S < 1/2 \quad \text{and}\quad -1/2 < \Delta_1^N < \Delta_2^N < \infty.
\end{equation*}
We will consider the local Whittle estimator that can be computed from the time series $Z_t^0, \Delta Z_t$ or $\bar{Z}_t$ on the range of admissible values given by $[\Delta_1^S,\Delta_2^S]$ or $[\Delta_1^N,\Delta_2^N]$ depending on the context. 

\subsubsection{Inference on $\dn$}\label{sec_appen_d1}

We first establish the following:
\begin{proposition} \label{prop1} 
Suppose that Assumptions~\ref{assum1} and~\ref{assumlw} hold. Then the following holds as $1/m + m/T \to 0$. 
\begin{enumerate}[(i)]	
	\item\label{prop1a} For $\dn \in (1/2,1)$ and $\dn \in [\Delta_1^N,\Delta_2^N]$,  $\hat{d}_{LW} (\langle	{Z}_t^0 ,v \rangle) \pto \dn.$
	\item\label{prop1b} For $\dn \in [1,3/2)$ and $\dn \in [\Delta_1^N,\Delta_2^N]$, $\hat{d}_{LW} (\langle 	{Z}_t^0 ,v \rangle) \pto 1.$
	\item\label{prop1c} For $\dn \in (1/2,3/2)$ and $\dn-1 \in [\Delta_1^S,\Delta_2^S]$, $\hat{d}_{LW} (\langle 	\Delta{Z}_t ,v \rangle) \pto \dn-1.$
\end{enumerate}	
\end{proposition} 
\begin{proof}
First, initialization ($Z_t \mapsto Z_t^0$) does not affect the periodogram. Since $\mathbb{P}(v \notin \mathcal H_S)=0$ and $\rank P\sum_{j=0}^\infty \psi_j = q_{\dnn}$, we find that the long-run variance of $\langle X_t,v\rangle$ is equal to 
$\left \langle v,P\left(\sum_{j=0}^\infty \psi_j\right) C_{\varepsilon} \left(\sum_{j=0}^\infty \psi_j\right)^\ast P v \right\rangle$, which is nonzero almost surely if Assumption~\ref{assumlw}(\ref{assumsr1}) is true. Moreover, under Assumption~\ref{assumlw}(\ref{assumsr0}), we have $\langle X_t, v \rangle = \langle \sum_{j=0}^\infty \psi_j \varepsilon_{t-j},v \rangle = \sum_{j=0}^\infty\langle   \varepsilon_{t-j},\psi_j^\ast v \rangle = \sum_{j=0}^\infty \phi_j u_{t-j}$, where $u_t = \langle   \varepsilon_{t}, A^\ast v \rangle$ which is an i.i.d.\ sequence with mean zero and positive variance. Combining these results with Assumption~\ref{assumlw}(\ref{assumsr2}), one can verify that $\langle Z_t^0,v \rangle$ satisfies all the assumptions employed in Section 3 of \cite{phillips2004local}. Then the desired results~(\ref{prop1a}) and~(\ref{prop1b}) immediately follow from Theorems 3.1 and 3.2 of \cite{phillips2004local}. From similar arguments, we may also deduce~(\ref{prop1c})  from Theorem 3.1 of \cite{SHIMOTSU2006209}.
\end{proof}
Of course, if the time series $\langle Z_t^0,v \rangle$ satisfies some additional conditions employed in \commWS{Section 4 of \cite{phillips2004local}}, we then may establish the asymptotic distribution of $\hat{d}_{LW}(Z_t^0,v)$ for $\dn \in (1/2,1]$. However, as shown by \cite{phillips2004local}, this asymptotic distribution depends on values of $\dn$. A more convenient result is given below:
\begin{proposition}\label{prop2add}
Suppose that Assumptions~\ref{assum1} and~\ref{assumlw} hold with  $f_v(\lambda) = G_v (1+O(\lambda^\beta))$ (as $\lambda \to 0+$) for some $G_v \in (0, \infty)$ and $\beta \in (0,2]$, the power transfer function $\phi(\lambda) = \sum_{j=0}^\infty \phi_j e^{ij\lambda}$ is differentiable around the origin with $\sum_{j \geq M} \phi_j = O(1/\log^{4}(M+1))$ and $\sum_{k\geq M}\sum_{j=0}^\infty \phi_{j}\phi_{j+k} = O(1/\log^{4}(M+1))$ uniformly in $M=1,2,\ldots$. Moreover, assume that 
\begin{equation*}
	\left| \frac{d}{d\lambda} \phi(\lambda) \right| = O(\lambda^{-1}) \quad \text{as $\lambda \to 0+$}.
\end{equation*} 
Then, for $\dn \in (1/2,3/2)$ and $\dn-1 \in [\Delta_1^S,\Delta_2^S]$, we have  
\begin{equation*}
	m^{1/2}(\hat{d}_{LW} \langle \Delta Z_t,v \rangle - (\dn-1)) \dto N(0,1/4).
\end{equation*} 
as $1/m + m^{1+2\beta}(\log m)^2/T^{2\beta} \to 0$ and $T\to \infty$.
\end{proposition} 
\begin{proof}
We note that $\langle X_t,v \rangle = \sum_{j=0}^\infty \phi_j u_{t-j}$, where $u_t = \langle \varepsilon_t, Av \rangle$ is an i.i.d.\ sequence. We thus have $\sum_{j \geq M} \mathbb{E}[\langle X_t,v \rangle\langle X_{t+k},v \rangle] = O(1/\log^{4}(M+1))$. Then one can easily verify that all the assumptions employed in \citet[Section 4]{SHIMOTSU2006209} are satisfied, and then we may deduce the desired result from their Theorem 4.1. 
\end{proof}

\subsubsection{Inference on $\ds$}\label{sec_appen_d2}

We then provide our estimation results for $\ds$. In this section, the following preliminary result will be used:  if Assumption~\ref{assum1} holds and  the first $K$ largest eigenvalues of $\mathbb{E}[(I-P){Y}_t\otimes (I-P){Y}_t]$ are distinct, then  $\tilde{v} = \sum_{j=q_{\dnn}+1}^{q_{\dnn}+K} a_j \hat{v}_j$ with $a_{q_{\dnn}+1} \neq 0$ (where $\hat{v}_j$ is defined in Section~\ref{sec_est_d2}) satisfies that 
\begin{equation}\label{eq001a}
\|\tilde{v} - \sgn{\langle \tilde{v}, v \rangle}v\|\pto 0
\end{equation}
for some fixed element $v$ with $(I-P)v \neq 0$. In this section, the asymptotic results given by Proposition 1 and Theorem 1 in \cite{Li2021} are crucial inputs. In this regard, it is worth mentioning that even if $\psi_j = \phi_j I$ is assumed by \cite{Li2021}, unlike in the present paper, their results can be extended to the case where Assumption~\ref{assumlw}(\ref{assumsr0}) is satisfied with only a slight modification.
\begin{proposition}\label{prop5}
Suppose that Assumption~\ref{assum1} holds and the first $K$ largest eigenvalues of $\mathbb{E}[(I-P){Y}_t\otimes (I-P){Y}_t]$ are distinct. Moreover, we suppose that Assumption~\ref{assumlw} holds for $v$ satisfying \eqref{eq001a} and $\ds \in [\Delta^S_1,\Delta^S_2]$.   Then,  as $1/m + m/T \to 0$, 
\begin{equation} \label{eq001ab}
	\hat{d}_{LW}(\langle \bar{Z}_t,\tilde{v} \rangle)  \pto  \ds,
\end{equation}
where $\tilde{v} = \sum_{j=q_{\dnn}+1}^{q_{\dnn}+K} a_j \hat{v}_j$ and $a_{q_{\dnn}+1} \neq 0$.
\end{proposition}
\begin{proof}[Proof of Proposition~\ref{prop5}]
If the first $K$ largest eigenvalues of $\mathbb{E}[(I-P){Y}_t\otimes (I-P){Y}_t]$ are distinct, we know from Proposition~\ref{prop3} and Lemma 4.3 of \cite{Bosq2000} that $\hat{v}_{q_{\dnn}+j}$ converges to the eigenvector $v_{q_{\dnn}+j}$ corresponding to the $j$-th largest eigenvalue of $\mathbb{E}[(I-P){Y}_t\otimes (I-P){Y}_t]$ in the following sense:
\begin{equation} \label{eqpf001}
	\|\hat{v}_{q_{\dnn}+j} - \sgn(\langle \hat{v}_{q_{\dnn}+j},v_{q_{\dnn}+j}\rangle)v_{q_{\dnn}+j}\| \pto 0, \quad \text{for $j=1,\ldots, \min\{q_{\dss},K\}$}.	 
\end{equation}
As may be deduced from the fact that the periodogram is not affected by demeaning and the proof of Theorem 1  in \cite{Li2021}, replacing the periodogram associated with $\langle \bar{Z}_t, \tilde{v} \rangle$ with that of $\langle \bar{Z}_t, {v} \rangle$  causes only negligible changes if $\|\tilde{v} - \sgn(\langle \tilde{v},v\rangle)v\| \pto 0$, and thus the difference between  $\hat{d}_{LW}(\langle \bar{Z}_t,\tilde{v} \rangle)$ and $\hat{d}_{LW}(\langle \bar{Z}_t,v \rangle)$ becomes negligible.  Note that we may write $\tilde{v} = \tilde{v}_1 + \tilde{v}_2$, where  $\tilde{v}_1 = \sum_{j=q_{\dnn}+1}^{q_{\dnn}+q_{\dss}}  a_j \hat{v}_{j}$ and $\tilde{v}_2=  \sum_{j=q_{\dnn}+q_{\dss}+1}^K a_j \hat{v}_{j}$. \eqref{eqpf001} implies that,  for $k=1$ and $2$, $\|\tilde{v}_k - \sgn(\langle \tilde{v}_k,\bar{v}_k\rangle)\bar{v}_k\| \pto 0$, where $\bar{v}_1=\sum_{j=q_{\dnn}+1}^{q_{\dnn}+q_{\dss}}  a_j {v}_{j}$  and $\bar{v}_2=  \sum_{j=q_{\dnn}+q_{\dss}+1}^K a_j {v}_{j}$. We thus find that $\langle Y_t,v  \rangle$ for $v =\bar{v}_1+\bar{v}_2$ is not only stationary I$(\ds)$ but also satisfies all the requirements for Proposition 1(i) of \cite{Li2021} under Assumption~\ref{assumlw}. We thus conclude that $\hat{d}_{LW}(\langle \bar{Z}_t,v \rangle) \pto \dn$, which completes the proof given that the distance between $\hat{d}_{LW}(\langle \bar{Z}_t,\tilde{v} \rangle)$ and $\hat{d}_{LW}(\langle \bar{Z}_t,v \rangle)$ is negligible.
\end{proof}
Note that \eqref{eq001aa} is a special case of \eqref{eq001ab} when $K=1$. If some additional conditions given by Assumption 2$^\ast$ in \cite{Li2021} hold, the following may also be deduced from the proof of Theorem 1 of \cite{Li2021}:
\begin{equation*} 
m^{1/2}(\hat{d}_{LW}(\langle \bar{Z}_t,\tilde{v} \rangle) - \ds)  \dto N(0,1).
\end{equation*}
A detailed proof of this result is omitted since it is, in fact, similar to that of Proposition~\ref{prop5}; under all of the aforementioned assumptions, one may show that  (i) the time series $\langle Y_t,v \rangle$ becomes an I$(\ds)$ stationary linear process, (ii) $m^{1/2}(\hat{d}_{LW}(\langle \bar{Z}_t,v \rangle) - \ds)  \dto N(0,1)$ (Proposition 1 of \cite{Li2021}) and (iii) replacing $\tilde{v}$ with $v$ only has a negligible impact  (Theorem 1 and Remark 4 of \cite{Li2021}). 

\section{Additional simulation results}
\subsection{Supplementary results}  \label{sec_supplement_table} 
\renewcommand{\arraystretch}{0.9}
\begin{table}[!htb]
\tabcolsep 0.2in
	\caption{Finite sample performance of the estimators of $q_{\dnn}$} \label{tab1add}
	\begin{tabular}{@{\extracolsep{\fill}}lccccc@{}}
		\toprule 
		& \multicolumn{5}{c}{Relative frequency of $\hat{q}_{\dnn} = 3$ or $4$ (the true value + 1)}\\\midrule 
		$q_{\max}$ or $K$ & Method	& $T=200$ & $T=350$ & $T=500$  & $T=1000$  \\ \midrule 
		$4$ 	&{Proposed}	& 0.978& 0.971& 0.960 &0.948\\ 
		&LRS-type &0.348& 0.621 &0.775& 0.924\\ \midrule 
		$5, 6$ 	&Proposed	&0.962 &0.969 &0.960& 0.948 \\ 
		&LRS-type &0.348 &0.621 &0.775& 0.924\\ \midrule
	\end{tabular}
	\begin{justify}
		{\footnotesize Notes: This table provides additional information on the results reported in Table~\ref{tab1}.} 
	\end{justify}
\end{table}  

\begin{table}[!htb]
\tabcolsep 0.15in
	\caption{Finite sample performance of the estimators of $q_{\dnn}$ for different tuning parameters} \label{tab1addrv1}
	\begin{tabular}{@{\extracolsep{\fill}}lcccccc@{}}
		\toprule 
		& &\multicolumn{5}{c}{Relative frequency of correct determination of $q_{\dnn}$}\\\midrule 
	 $\alpha$& 	$q_{\max}$ or $K$ & Method	& $T=200$ & $T=350$ & $T=500$  & $T=1000$  \\ \midrule 
	0.4	&	$4$ &Proposed 	& 0.839& 0.936& 0.949& 0.945\\ 
	&	$5, 6$ 	&Proposed	&0.838 &0.936 &0.949& 0.945 \\ \midrule
	0.6 & $4$	&Proposed 	&0.748& 0.877& 0.910& 0.937\\ 
&	$5$ &Proposed 	&0.743& 0.877 &0.910& 0.937\\ 
 	&$6$ &Proposed 	& 0.742& 0.877& 0.910& 0.937\\
\midrule
	\end{tabular}
	\begin{justify}
		{\footnotesize Notes: This table provides additional information on the results reported in Table~\ref{tab1}, where the test statistics are computed with $\alpha=0.5$.} 
	\end{justify}
\end{table} 

\subsection{{Size-power properties of the variance-ratio test}} 

In Tables~\ref{tabsizepower}-\ref{tabsizepowerrv2}, we report the size-power properties of the variance-ratio test. 
\renewcommand{\arraystretch}{0.7}
\begin{table}[!htb]
\caption{Size and power of the variance-ratio test, $\alpha=0.5$} \label{tabsizepower}
\begin{tabular*}{1\linewidth}{@{\extracolsep{\fill}}lllcccc@{}}
	\toprule  
	$q_{\max}$	&Method & Hypothesis &  $T=200$ & $T=350$ & $T=500$ & $T=1000$ \\ \midrule
	4&  max-test	  & Size & 0.021 &0.028& 0.038& 0.052\\   
	&	&   Power & 0.781& 0.934& 0.967& 0.997\\ 
	&  trace-test	  & Size & 0.020 &0.025& 0.033& 0.046 \\   
	&	& Power & 0.734& 0.921& 0.968&0.998  \\ \midrule
	5&  max-test	  & Size & 0.020& 0.028 &0.036& 0.051\\   
	&	&   Power &  0.749& 0.928& 0.966& 0.996\\ 
	&  trace-test	  &Size &0.019& 0.024& 0.032& 0.046\\   
	&	&  Power & 0.700& 0.911&0.965 &0.998  \\ \midrule
	6&  max-test	  & Size & 0.018& 0.026& 0.035& 0.051\\   
	&	&   Power &  0.723&0.923& 0.965 &0.996\\ 
	&  trace-test	  & Size & 0.016& 0.024 &0.031& 0.045 \\   
	&	&  Power & 0.667& 0.903& 0.964 &0.998   \\ \bottomrule
\end{tabular*}  
\begin{justify}
{\footnotesize Notes: Based on 2,000 Monte Carlo replications. The reported power is computed by testing $H_0:q_{\dnn}=4$ under the simulation DGP. The tests are implemented based on $\Lambda_{s,\alpha}^0$ (max-test) and $\Lambda_{s,\alpha}^1$ (trace-test), respectively, with $\alpha = 0.5$,  $K =q+2$ and significance level $\eta = 0.05$.} 
\end{justify}
\end{table}

\begin{table}[!htb]
\caption{Size and power of the variance-ratio test, $\alpha=0.4$} \label{tabsizepowerrv1}
\begin{tabular*}{1\linewidth}{@{\extracolsep{\fill}}lllcccc@{}}
	\toprule  
	$q_{\max}$	&Method & Hypothesis &  $T=200$ & $T=350$ & $T=500$ & $T=1000$ \\ \midrule
	4&  max-test	  & Size & 0.024 &0.032 &0.041& 0.053\\   
	&	&  Power & 0.832& 0.964& 0.988& 0.999\\ 
	&  trace-test	  & Size & 0.020& 0.028& 0.034& 0.047 \\   
	&	& Power & 0.753& 0.940& 0.983& 0.999  \\ \midrule
	5&  max-test	  & Size & 0.022& 0.030& 0.040 &0.053 \\   
	&	&   Power &   0.803& 0.960 &0.987& 0.999\\ 
	&  trace-test	  & Size & 0.018 &0.027& 0.032& 0.046 \\   
	&	&  Power & 0.725& 0.934& 0.981& 0.999  \\ \midrule
	6&  max-test	  & Size & 0.020 &0.030& 0.039& 0.053\\   
	&	&  Power &   0.777& 0.953& 0.985& 0.999\\ 
	&  trace-test	  & Size & 0.017& 0.026& 0.032& 0.046 \\   
	&	&  Power & 0.700 &0.928& 0.980& 0.999   \\ \bottomrule
\end{tabular*}  
\begin{justify}
{\footnotesize Notes: The tests are implemented as in Table~\ref{tabsizepower}, but with $\alpha=0.4$.} 
\end{justify}
\end{table}

\begin{table}[!htb]
\caption{Size and power of the variance-ratio test, $\alpha=0.6$} \label{tabsizepowerrv2}
\begin{tabular*}{1\linewidth}{@{\extracolsep{\fill}}lllcccc@{}}
	\toprule  
	$q_{\max}$	&Method & Hypothesis &  $T=200$ & $T=350$ & $T=500$ & $T=1000$ \\ \midrule
	4&  max-test	  & Size & 0.020& 0.030& 0.035& 0.050 \\   
	&	& Power & 0.733& 0.890 &0.944& 0.986\\ 
	&  trace-test	  & Size & 0.020& 0.026& 0.033& 0.048 \\   
	&	&  Power & 0.700 &0.884& 0.946& 0.994  \\ \midrule
	5&  max-test	  & Size & 0.020& 0.028 &0.034 &0.049  \\   
	&	&    Power &  0.697& 0.875 &0.937& 0.986\\ 
	&  trace-test	  &Size &0.018 &0.025& 0.032& 0.048  \\   
	&	&  Power & 0.664 &0.874& 0.942& 0.992  \\ \midrule
	6&  max-test	  & Size & 0.018 &0.027& 0.034 &0.049 \\   
	&	&   Power &  0.658& 0.865& 0.931& 0.985\\ 
	&  trace-test	  & Size & 0.017&0.023& 0.032& 0.048  \\   
	&	& Power & 0.621& 0.864& 0.938& 0.992   \\ \bottomrule
\end{tabular*}  
\begin{justify}
{\footnotesize Notes: The tests are implemented as in Table~\ref{tabsizepower}, but with $\alpha=0.6$.} 
\end{justify}
\end{table}


\subsection{Sensitivity analysis and coverage performance of the Local Whittle estimators} \label{sec_sensitivity} 
In Tables~\ref{tab3a} and~\ref{tab4a}, we study sensitivity analysis of the local Whittle estimators. The accuracy of the confidence intervals of the memory parameters is documented in Table~\ref{tab5a}. 
\renewcommand{\arraystretch}{0.9}
\begin{table}[!htb]
\caption{Finite-sample performance of the Local Whittle estimators of $d$} \label{tab3a}
\begin{tabular*}{1\linewidth}{@{\extracolsep{\fill}}lcccccc@{}}
	\toprule  
	$m=\lfloor1+ T^{0.6} \rfloor$ &	Method & & $T=200$ & $T=350$ & $T=500$ & $T=1000$ \\ \midrule 
	Mean Bias &	Proposed&	&-0.0496&-0.0349& -0.0258& -0.0112
	\\  
	&	LRS-type &&   -0.1298& -0.0874& -0.0645& -0.0344
	\\ \midrule 
	Variance &	Proposed&	&   0.0125& 0.0078 &0.0063& 0.0036
	\\  
	&	LRS-type &&   0.0246& 0.0148& 0.0106& 0.0050
	\\ \midrule 
	MSE &	Proposed&	&   0.0149& 0.0090& 0.0070 &0.0038
	\\  
	&	LRS-type && 0.0415& 0.0224& 0.0147& 0.0062
	\\ \midrule  
	$m=\lfloor1+ T^{0.7} \rfloor$ &	Method & & $T=200$ & $T=350$ & $T=500$ & $T=1000$ \\ \midrule 
	Mean Bias &	Proposed&	& -0.0427& -0.0306& -0.0244& -0.0124
	\\  
	&	LRS-type && -0.1061& -0.0672 &-0.0472& -0.0229
	\\ \midrule 
	Variance &	Proposed&	&0.0089& 0.0055& 0.0041& 0.0025\\  
	&	LRS-type && 0.0240 &0.0149& 0.0104& 0.0045
	\\ \midrule 
	MSE &	Proposed&	& 0.0107& 0.0064& 0.0047& 0.0027
	\\  
	&	LRS-type &&    0.0353 &0.0194& 0.0127& 0.0050
	\\ \midrule 
\end{tabular*} 
{\footnotesize Notes: Based on 2,000 Monte Carlo replications. The estimates are computed as in Table~\ref{tab3}.}
\end{table}
\begin{table}[!htb]
\caption{Finite-sample performance of the Local Whittle estimators of $\ds$}\label{tab4a}
\begin{tabular*}{1\linewidth}{@{\extracolsep{\fill}}lcccccc@{}}
\toprule  
$m=\lfloor 1+ T^{0.6} \rfloor$ &	Method & & $T=200$ & $T=350$ & $T=500$ & $T=1000$ \\ \midrule 
Mean Bias &	Proposed&	&-0.1034& -0.0746& -0.0585& -0.0403\\  
&	LRS-type && -0.1590 &-0.1211& -0.1002 &-0.0774
\\ \midrule 
Variance &	Proposed&	&   0.0145 &0.0106& 0.0081 &0.0048
\\  
&	LRS-type & &0.0158& 0.0133& 0.0111& 0.0067
\\ \midrule 
MSE &	Proposed&	&  0.0252& 0.0162& 0.0115& 0.0065
\\  
&	LRS-type &&  0.0410& 0.0279& 0.0212& 0.0126
\\ \midrule 
$m=\lfloor 1+T^{0.7} \rfloor$&	Method & & $T=200$ & $T=350$ & $T=500$ & $T=1000$ \\ \midrule 
Mean Bias &	Proposed&	& -0.0537& -0.0294& -0.0194& -0.0107
\\  
&	LRS-type &&-0.1077 &-0.0662& -0.0498& -0.0325
\\ \midrule 
Variance &	Proposed&	&   0.0110& 0.0067 &0.0046& 0.0025
\\  
&	LRS-type && 0.0149& 0.0096& 0.0068& 0.0035
\\ \midrule 
MSE &	Proposed&	&  0.0139& 0.0076& 0.0050& 0.0026
\\  
&	LRS-type && 0.0264 &0.0140& 0.0093& 0.0046	
\\ \midrule 
\end{tabular*} 
\begin{justify}
{\footnotesize Notes: Based on 2,000 Monte Carlo replications. The estimates are computed as in Table~\ref{tab4}.} 
\end{justify}
\end{table}

\begin{table}[!htb]
\caption{Coverage performance of the pointwise confidence intervals of the memory parameter estimated by the local Whittle estimators with the 80\% nominal level}\label{tab5a}
\begin{tabular*}{1\linewidth}{@{\extracolsep{\fill}}lcccccc@{}}
\toprule  
& \multicolumn{6}{c}{Coverage probability differences} \\\midrule
$m$ &	Target & Method& $T=200$ & $T=350$ & $T=500$ & $T=1000$ \\ \midrule 
$\lfloor 1+ T^{0.6} \rfloor$ &	$d_{\dnn}$&Proposed	&0.1450& 0.0960&  0.0740 &0.0590
\\  
&&LRS-type	&0.3500& 0.2870&0.2515&  0.2015		\\  
&	$d_{\dss}$ & Proposed & 0.2545&0.1990& 0.1485& 0.1175
\\ 
&&LRS-type	& 0.4310& 0.3370	 	&  0.3135	 	 & 0.2950
\\  
\midrule 
$\lfloor 1+ T^{0.65} \rfloor$ & $d_{\dnn}$&Proposed	&0.1425& 0.1075& 0.0890& 0.0655
\\  
&&LRS-type	& 0.3295&  0.2715&  0.2385 &  0.1885
\\  
&	$d_{\dss}$ &Proposed&0.2080& 0.1345& 0.1080 & 0.0600
\\ 
&&LRS-type	&0.3765 & 0.2800 & 0.2405&  0.2080	
\\  
\midrule 
$\lfloor 1+ T^{0.7} \rfloor$ &	$d_{\dnn}$&Proposed	&0.1555&  0.1490& 0.1370 & 0.1125
\\  
&&LRS-type	&   0.3100&   0.2775		&0.2495&  0.2135
\\  
&$d_{\dss}$ &Proposed& 0.1820&  0.1370&0.0855&  0.0445
\\
&&LRS-type	& 0.3505& 0.2375&  0.2015&  0.1650
\\ \midrule 
& \multicolumn{6}{c}{Interval scores} \\\midrule
$m$ &	Target & Method& $T=200$ & $T=350$ & $T=500$ & $T=1000$ \\ \midrule 
$\lfloor 1+ T^{0.6} \rfloor$ &	$d_{\dnn}$&Proposed	&  0.5109&0.3944&0.3427  &0.2601
\\  
&&LRS-type	& 1.1273& 0.8195& 0.6354&0.4037	\\  
&	$d_{\dss}$ & Proposed & 0.6528&0.4961 & 0.4069 &0.2959 
\\ 
&&LRS-type	&0.9718&  0.7610& 0.6417   &0.4745\\  
\midrule 
$\lfloor 1+ T^{0.65} \rfloor$ & $d_{\dnn}$&Proposed	& 0.4577 &0.3493 & 0.3014&  0.2249
\\  
&&LRS-type	&  1.0650&  0.7695&  0.5995 &0.3698
\\  
&	$d_{\dss}$ &Proposed&0.5303&   0.3798& 0.3132  &  0.2227
\\ 
&&LRS-type	& 0.8501 & 0.5988&   0.4844 & 0.3288	
\\  
\midrule 
$\lfloor 1+ T^{0.7} \rfloor$ &	$d_{\dnn}$&Proposed	& 0.4292& 0.3385 & 0.2863 &   0.2244 	\\  
&&LRS-type	&1.0493& 0.7841 &  0.6212& 0.3971
\\  
&$d_{\dss}$ &Proposed&0.4629&  0.3272&  0.2559&    0.1814 
\\
&&LRS-type	& 0.7614&0.5001& 0.3885 &0.2580 
\\ \midrule 
\end{tabular*} 
\begin{justify}
{\footnotesize Notes: Based on 2,000 Monte Carlo replications. The estimates are computed as in Table~\ref{tab5}, and the reported number in each case is computed as the absolute value of the difference between the computed coverage rate and the nominal level 0.8. The interval score in each case is computed with the quantiles 0.1 and 0.9.}
\end{justify}
\end{table}
 


\newpage
\bibliographystyle{apalike}
\bibliography{FCIFTS_ref}	

\end{document}